\documentclass[a4paper,10pt]{article}
\usepackage{amssymb,amsmath,amsthm}
\usepackage{graphicx}

\usepackage{amsfonts}
\usepackage{latexsym}
\usepackage[english]{babel}
\usepackage{hyperref} 

\numberwithin{equation}{section}

\newcommand{\norma}{|\!\!|}

\newtheorem{theorem}{Theorem}[section]
\newtheorem{proposition}{Proposition}[section]
\newtheorem{lemma}{Lemma}[section]
\newtheorem{corollary}{Corollary}[section]

\newtheorem{remark}{Remark}[section]
\newtheorem{remarks}{Remark}[section]
\newtheorem{definition}{Definition}[section]
\newcommand{\be}{\begin{equation}}
\newcommand{\ee}{\end{equation}}

\newcommand{\e}{\varepsilon}

\newcommand{\R}{\mathbb R}
\newcommand{\C}{\mathbb C}

\newcommand{\Z}{\mathbb Z}

\newcommand{\N}{\mathbb N}
\newcommand{\T}{\mathbb T}

\newcommand{\ii }{{\rm i} }
\renewcommand{\d }{\delta }

\newcommand{\g }{\gamma}

\newcommand{\vphi}{\varphi }

\renewcommand{\t }{\tau }

\newcommand{\p}{\pi}

\newcommand{\Lipg}{{\rm{Lip}(\g)}}
\newcommand{\lip}{{\rm lip}}

\begin{document}

\title{\textbf{A reducibility result for a class of linear wave equations on $\T^d$}} 

\date{}
\author{ Riccardo Montalto \footnote{Supported in part by the Swiss National Science Foundation}}

\maketitle

\noindent
{\bf Abstract.}
We prove a {\it reducibility} result for a class of quasi-periodically forced linear wave equations on the $d$-dimensional torus $\T^d$ of the form 
$$
\partial_{tt} v - \Delta v + \e {\cal P}(\omega t)[v] = 0
$$
where the perturbation ${\cal P}(\omega t)$ is a second order operator of the form ${\cal P}(\omega t) = - a(\omega t) \Delta  -  {\cal R}(\omega t)$, the frequency $\omega \in \R^\nu$ is in some Borel set of large Lebesgue measure, the function $a : \T^\nu \to \R$ (independent of the space variable) is sufficiently smooth and ${\cal R}(\omega t)$ is a time-dependent finite rank operator. This is the first reducibility result for linear wave equations with unbounded perturbations on the higher dimensional torus $\T^d$. As a corollary, we get that the linearized Kirchhoff equation at a smooth and sufficiently small quasi-periodic function is {\it reducible}.

\smallskip


\smallskip


\tableofcontents

\section{Introduction and main result}
We consider a linear quasi-periodically forced wave equation of the form  
\begin{equation}\label{main equation}
\partial_{tt} v -  \Delta v + \e {\cal P}(\omega t)[v] = 0, \quad x \in \T^d
\end{equation}
where $\T := \R/(2 \pi \Z)$,  $\e > 0 $ is a small parameter, $\omega \in \Omega \subseteq \R^\nu$, with $\Omega$ a closed bounded domain and the operator ${\cal P}(\omega t)$ is given by 
\begin{equation}\label{perturbazione P omega t}
{\cal P}(\vphi)[v] := - a(\vphi) \Delta v - {\cal R}(\vphi)[v]\,, \quad \vphi \in \T^\nu\,, \quad v \in L^2_0(\T^d, \R)
\end{equation}
with ${\cal R}(\vphi)$ being an operator of the form 
\begin{equation}\label{definizione perturbazione rango finito}
{\cal R}(\vphi)[v] :=  \sum_{k = 1}^N b_k(\vphi, x) \int_{\T^d} c_k(\vphi, y) v(y)\, d y + c_k(\vphi, x) \int_{\T^d} b_k(\vphi, y) v(y)\, d y\,, \qquad \vphi \in \T^\nu\,, \quad v \in L^2_0(\T^d, \R)\,. 
\end{equation}
Here $\nu, d \geq 1$ are integer numbers, $L^2_0(\T^d, \R)$ denotes the space of the real valued $L^2$ functions with zero average and the functions $a : \T^\nu \to \R, b_k, c_k : \T^\nu \times \T^d \to \R$, $k = 1, \ldots, N$ are assumed to be sufficiently smooth, namely $a \in {\cal C}^q(\T^\nu, \R), b_k, c_k \in {\cal C}^q(\T^\nu \times \T^d, \R)$ for some $q > 0$ large enough. Note that the operator ${\cal R}(\vphi)$ is symmetric with respect to the real $L^2$-inner product. 
Our aim is to prove a reducibility result for the equation \eqref{main equation} for $\e$ small enough and for $\omega$ in a suitable Borel set of parameters $\Omega_\e \subset \Omega$ with asymptotically full Lebesgue measure. The PDE \eqref{main equation} may be written as the first order system
\begin{equation}\label{Kirchoff first order}
\begin{cases}
\partial_t v = \psi \\
\partial_t \psi = \Big( 1 + \e a(\omega t)  \Big) \Delta v + \e {\cal R}(\omega t)[v]
\end{cases}
\end{equation}
which is a real Hamiltonian system of the form 
\begin{equation}\label{sistema hamiltoniano generale}
\begin{cases}
\partial_t v = \nabla_\psi H(\omega t, v, \psi) \\
\partial_t \psi = - \nabla_v H(\omega t, v, \psi)
\end{cases}
\end{equation}
whose $\vphi$-dependent Hamiltonian is given by 
\begin{equation}\label{Hamiltoniana Kirchoff}
H(\vphi, v, \psi) := \frac12 \int_{\T^d} \Big( \psi^2 + (1 + \e a(\vphi))|\nabla v|^2 \Big)\, dx  - \e \frac12\int_{\T^d} {\cal R}(\vphi)[v]\, v\, dx\,.  
\end{equation}
In \eqref{sistema hamiltoniano generale}, $\nabla_\psi H$ and $\nabla_v H$ denote the $L^2$-gradients of the Hamiltonian $H$ with respect to the variables $v$ and $\psi$. We assume that the functions $b_k(\vphi, x), c_k(\vphi, x)$, $k = 1, \ldots, N$ have zero average with respect to $x \in \T^d$, namely 
\begin{equation}\label{condizione media f}
\int_{\T^{d}} b_k(\vphi, x)\, d x = 0\,, \quad \int_{\T^d} c_k(\vphi, x)\, d x= 0 \quad \forall \vphi \in \T^\nu\,, \quad k = 1, \ldots, N\,. 
\end{equation}

In order to precisely state the main result of this paper, let us introduce some more notations. For any $s \geq 0$, we define the Sobolev spaces $H^s(\T^d) = H^s(\T^d, \C)$, $H^s(\T^d, \R)$ respectively of complex and real valued functions
\begin{equation}\label{Sobolev x}
H^s(\T^d) := \big\{ u(x) = \sum_{j \in \Z^d} u_j e^{\ii j x} : \| u \|_{H^s_x}^2 := \sum_{j \in \Z^d} \langle j \rangle^{2 s} |u_j|^2 < + \infty \big\}\,,\quad H^s(\T^d, \R) := \big\{ u \in H^s(\T^d) : u = \overline u \big\}\,
\end{equation}
where 
$$
 \langle j \rangle := {\rm max}\{1, |j|\}\,, \quad |j | := \sqrt{j_1^2 + \ldots + j_d^2}\,, \quad \forall j = (j_1, \ldots, j_d) \in \Z^d\,. 
$$
Moreover we define 
\begin{equation}\label{Sobolev x media nulla}
H_0^s(\T^d) := \big\{ u \in H^s(\T^d) : \int_{\T^d} u(x)\, d x = 0  \big\}\,, \quad H^s_0(\T^d, \R) := \big\{ u \in H^s(\T^d, \R) : \int_{\T^d} u(x)\, d x = 0  \big\}\,
\end{equation}
and introduce the real subspace ${\bf H}^s_0(\T^d)$ of $H^s_0(\T^d) \times H^s_0(\T^d)$  
$$
{\bf H}^s_0(\T^d) := \big\{ {\bf u} := (u, \overline u) : u \in H^s_0(\T^d) \big\}\,, \quad \text{equipped\,with\,the\,norm}\quad \| {\bf u}\|_{{\bf H}^s_x} := \| u \|_{H^s_x}\,.
$$
Given a linear operator ${\cal R} : L^2_0(\T^d) \to L^2_0(\T^d)$ (where $L^2_0(\T^d) := H^0_0(\T^d)$), we define its Fourier coefficients with respect to the exponential basis $\{ e^{\ii j \cdot x} : j \in \Z^d \setminus \{ 0 \} \}$ of $L^2_0(\T^d)$ as 
\begin{equation}\label{coefficienti spazio operatore}
{\cal R}_j^{j'} := \frac{1}{(2 \pi)^d} \int_{\T^d} {\cal R}[e^{\ii j' \cdot x}] e^{- \ii j \cdot x}\, d x\,, \qquad \forall j, j' \in \Z^d \setminus \{ 0 \}\,. 
\end{equation}
We introduce the linear operator $\overline{\cal R}$, defined by $\overline{\cal R} [u] = \overline{{\cal R}[\overline u]}$, for any $u \in L^2_0(\T^d)$. 

\noindent
We say that the operator ${\cal R}$ is block diagonal if ${\cal R}_j^{j'} = 0$ for any $j, j' \in \Z^d \setminus \{ 0 \}$ with $|j| \neq |j'|$. 

\noindent
Because of the hyphothesis \eqref{condizione media f} the Hamiltonian vector field 
\begin{equation}\label{campo vettoriale main equation}
{\cal L}(\vphi) := \begin{pmatrix}
0 & 1 \\
\Delta - \e {\cal P}(\vphi) & 0
\end{pmatrix} \stackrel{\eqref{perturbazione P omega t}}{=}\begin{pmatrix}
0 & 1 \\
(1 + \e a(\vphi)) \Delta + \e {\cal R}(\vphi) & 0
\end{pmatrix}\,, \qquad \vphi \in \T^\nu\,,
\end{equation}
leaves the space of functions with zero average invariant. More precisely for any $0 \leq s \leq q$
$$
{\cal L}(\vphi) : H^{s + 2}_0(\T^d, \R) \times H^{s + 1}_0(\T^d, \R) \to H^{s + 1}_0(\T^d, \R) \times H^{s }_0(\T^d, \R)\,, \quad \forall \vphi \in \T^\nu
$$
and therefore we can choose $H_0^1(\T^d, \R) \times L^2_0(\T^d, \R)$ as phase space for the Hamiltonian $H$ defined in \eqref{Hamiltoniana Kirchoff}. 
Now we are ready to state the main result of the present paper. 
\begin{theorem}\label{theorem linear stability} 
Let $\nu, d$ be integer numbers greater or equal than $1$. There exists a strictly positive integer $q_0 = q_0(\nu, d) > 1/2$ such that for any $q \geq q_0$ there exists $\e_q = \e(q, \nu, d) > 0$ and $\mathfrak S_q := \mathfrak S(q,  \nu, d)$, with $1/2 < \mathfrak S_q < q$ such that if $a \in {\cal C}^q(\T^\nu, \R), b_k, c_k \in {\cal C}^q(\T^\nu \times \T^d, \R)$, with $b_k, c_k$ satisfying the hyphothesis \eqref{condizione media f} for any $k = 1, \ldots, N$, then for any $\e \in (0, \e_q)$ there exists a Borel set $  \Omega_\e  \subset \Omega $ of asymptotically full Lebesgue measure, i.e. 
\begin{equation}\label{stima misura main theorem}
|\Omega_\e| \to |\Omega|\, \qquad \text{as} \qquad \e \to 0\,,
\end{equation}
such that the following holds: for all $\omega \in \Omega_\e$ and $\vphi \in \T^\nu$, there exists a bounded linear invertible operator ${\cal W}_\infty(\vphi) = {\cal W}_\infty(\vphi; \omega)$ such that for any $ \frac12 \leq s \leq \mathfrak S_q$ 
$$
{\cal W}_\infty (\vphi ) : {\bf H}^s_0(\T^d) \to H^{s + \frac12}_0(\T^d, \R) \times H^{s - \frac12}_0(\T^d, \R)
$$
satisfying the following property: $(v(t, \cdot), \psi(t, \cdot))$ is a solution of \eqref{Kirchoff first order} in $H^{s + \frac12}_0(\T^d, \R) \times H^{s - \frac12}_0(\T^d, \R)$ if and only if 
$$
{\bf u}(t, \cdot) = (u(t, \cdot), \overline u(t, \cdot)) = {\cal W}_\infty(\omega t)^{- 1}[(v(t, \cdot), \psi(t, \cdot)))]
$$
is a solution in ${\bf H}^s_0(\T^d)$ of the PDE with constant coefficients
$$
\partial_t {\bf u} = {\cal D}_\infty {\bf u}\,, \qquad {\cal D}_\infty := \ii \begin{pmatrix}
- {\cal D}_\infty^{(1)} & 0 \\
0 & \overline{\cal D}_\infty^{(1)}
\end{pmatrix}
$$
where for any $s \geq 1$, ${\cal D}_\infty^{(1)} : { H}_0^s(\T^d) \to { H}^{s - 1}_0(\T^d)$ is a linear, time-independent, $L^2$-self-adjoint, block-diagonal operator.
\end{theorem}
The following corollary holds: 
\begin{corollary}\label{crescita norme di sobolev}
For any $ \omega \in \Omega_\e$ and any initial data $(v^{(0)}, \psi^{(0)}) \in H^{s + \frac12}_0(\T^d, \R) \times H^{s - \frac12}_0(\T^d, \R)$ with $1/2  \leq s \leq  \mathfrak S_q$, the solution $ t \in \R \mapsto (v(t, \cdot), \psi(t, \cdot))\in H^{s + \frac12}_0(\T^d, \R) \times H^{s - \frac12 }_0(\T^d , \R)$ of the Cauchy problem 
\begin{equation}\label{equazione linearizzata}
\begin{cases}
\partial_t v = \psi \\
\partial_t \psi = \Big( 1 + \e a(\omega t)  \Big) \Delta v + \e {\cal R}(\omega t)[v]  \\
v(0, \cdot) = v^{(0)} \\
\psi(0, \cdot) = \psi^{(0)}
\end{cases}
\end{equation}
 is stable, namely 
$$
\sup_{t \in \R} \Big(\| v(t, \cdot) \|_{H^{s + \frac12}_x} + \| \psi (t, \cdot)\|_{H^{s - \frac12 }_x} \Big) \leq C_q \big( \|  v^{(0)}\|_{H^{s + \frac12}_x} + \| \psi^{(0)}\|_{H^{s - \frac12}_x} \big)\,.
$$
for some constant $C_q = C(q, \nu, d) > 0$. 
\end{corollary}
\begin{remark}\label{remark costanti q}
Note that the constants $\e_q$, $\mathfrak S_q$ in Theorem \ref{theorem linear stability} and the constant $C_q$ in Corollary \ref{crescita norme di sobolev} depend also on the $\| \cdot \|_{q}$ Sobolev norms of the functions $a$, $b_k, c_k$, $k = 1, \ldots, N$ appearing in the definition of the perturbation ${\cal P}$ given in \eqref{perturbazione P omega t}, \eqref{definizione perturbazione rango finito}. 
\end{remark}
Theorem \ref{theorem linear stability} implies a reducibility result for the linearized Kirchhoff equation at a small and sufficiently smooth quasi-periodic function $\e v_0(\omega t , x)$. The Kirchhoff equation  
\begin{equation}\label{Kirchoff equation Td}
 K(v ) := \partial_{tt} v - \Big( 1 + \int_{\T^d} |\nabla v|^2\, d x\Big)\Delta v = 0
\end{equation}
describes nonlinear vibrations of a $d$-dimensional body (in particular, a string for $d = 1$ and a membrane for $d = 2$). The Cauchy problem for the Kirchhoff equation has been extensively studied, starting from the pioneering paper of Bernstein \cite{Bernstein}. Both local and global existence results have been established for initial data in Sobolev and analytic class, see \cite{arosio-spagnolo}, \cite{arosio-panizzi}, \cite{dancona-spagnolo}, \cite{dickey}, \cite{lions}, \cite{manfrin}, \cite{pozohaev} and the recent survey \cite{Ruzansky}. The existence of periodic solutions for the Kirchhoff equation has been proved by Baldi \cite{Baldi Kirchhoff}. This result is proved via Nash-Moser method and thanks to the special structure of the nonlinearity (it is diagonal in space), the linearized operator at any approximate solution can be inverted by Neumann series. This approach does not imply the linear stability of the solutions, since only {\it the first order Melnikov conditions} are required along the proof. In one space-dimension ($d = 1$), the existence of quasi-periodic solutions and the reducibility of the linearized equation have been established in \cite{Montalto}. In dimension greater or equal than two, there are no results concerning the existence of quasi-periodic solutions. It is well-known that a good strategy for proving the existence and the linear stability of quasi-periodic solutions is to prove the reducibility of the linearized equations at small quasi-periodic approximate solutions obtained along a suitable iterative scheme. Hence our result (Theorem \ref{teorema corollario Kirchhoff} below) could be used to prove the existence of quasi-periodic solutions for the nonlinear Kirchhoff equation.  

\noindent
Linearizing the operator $K$ in \eqref{Kirchoff equation Td} at a quasi-periodic function $\e v_0(\omega t , x)$ and writing the {\it linearized equation} $K'(\e v_0) [v] = 0$ as a first order system, one gets a system of differential equations of the form \eqref{Kirchoff first order} where 
$$
a(\vphi) = \int_{\T^d} |\nabla v_0(\vphi, x)|^2\, d x\,,\quad  {\cal R}(\vphi)[v] = - 2 \Delta v_0(\vphi, x) \int_{\T^d} \Delta v_0(\vphi, y) v(y)\, d y\,, \quad \vphi \in \T^\nu\,, \quad v \in L^2_0(\T^d, \R)\,. 
$$
Note that the operator ${\cal R}(\vphi)$ defined above has the same form as the one defined in \eqref{definizione perturbazione rango finito}, by taking $N = 1$, $b_1 = - \Delta v_0$, $c_1 = \Delta v_0$. We point out that $\Delta v_0$ has zero average in $x \in \T^d$, hence the hyphothesis \eqref{condizione media f} is satisfied. An immediate consequence of Theorem \ref{theorem linear stability} and Corollary \ref{crescita norme di sobolev} is then the following 
\begin{theorem}\label{teorema corollario Kirchhoff}
Let $q_0, q, \e_q, \mathfrak S_q$ as in Theorem \ref{theorem linear stability} and $v_0 \in {\cal C}^{q + 2}(\T^{\nu} \times \T^d, \R)$. Then the conclusions of Theorem \ref{theorem linear stability} and Corollary \ref{crescita norme di sobolev} hold for the linearized Kirchhoff equation $K'(\e v_0)[v] = 0$ at the quasi-periodic function $\e v_0(\omega t, x)$. 
\end{theorem}


\noindent
Now we outline some related works concerning the reducibility of quasi-periodically forced linear partial differential equations. Let us consider a linear differential equation of the form 
\begin{equation}\label{riducibilita generale}
\partial_t u = {\cal D} u + \e {\cal P}(\omega t) u
\end{equation}
where ${\cal D}$ is a diagonal operator with discrete spectrum and ${\cal P}(\omega t)$ is a linear quasi-periodically forced vector field with non constant coefficients. We say that such an equation is {\it reducible} if there exists a quasi-periodically forced change of variable $u = \Phi(\omega t) [v]$ such that in the new coordinate $v$, the equation \eqref{riducibilita generale} is reduced to constant coefficients. 
Typically, it is necessary to assume that $\e$ (size of the perturbation) is small enough and that the frequency $\omega$, together with the eigenvalues of the operator ${\cal D}$, satisfy the so-called {\it second order Melnikov} non-resonance conditions. These non resonance conditions involve the differences of the eigenvalues of the operator ${\cal D}$. We point out that the reducibility of linear equations is the main ingredient for proving the existence of quasi-periodic solutions (KAM tori) for nonlinear PDEs. Indeed the first reducibility results for linear PDEs have been obtained as a corollary   of KAM theorems. We mention the pioneering papers of Kuksin \cite{Ku}, and Wayne \cite{W1} concerning the existence of invariant tori for Schr\"odinger and wave equations in one space dimension with Dirichlet boundary conditions and with bounded perturbations. The first KAM results for PDEs with unbounded perturbations have been obtained by Kuksin \cite{K2}, Kappeler-P\"oschel \cite{KaP} for analytic perturbations of the KdV equation. Here the unperturbed vector field is $\partial_{xxx}$ and the perturbation contains one space derivative $\partial_x$. Concerning unbounded perturbations of the quantum Harmonic oscillator on the real line, the first result is due to Bambusi-Graffi \cite{Bambusi-Graffi}. In all these aforementioned results, the perturbation contains derivatives of order $\delta < n - 1$, where $n$ is the order of the highest derivative appearing in the linear constant coefficients term. In the case of critical unbounded perturbations, i.e. $\delta = n - 1$, we mention \cite{LY}, \cite{ZGY} concerning the derivative NLS with Dirichlet boundary conditions, in which the authors generalized appropriately the so-called {\it Kuksin Lemma}, developed in \cite{K2}. We also mention the KAM results for the derivative Klein-Gordon equation \cite{BBiP1}, \cite{BBiP2} in which the generalization of the Kuksin Lemma developed in \cite{LY}, \cite{ZGY} does not apply because of the weaker dispersion relation.  

\noindent
It is well known that the ideas used to deal with the case $\delta \leq n - 1$ do not apply in the quasi-linear and fully nonlinear case, i.e. $\delta = n$. The first KAM results in this case have been obtained in \cite{BBM-Airy}, \cite{BBM-auto}, \cite{BBM-mKdV}, \cite{Giuliani} for quasi-linear perturbations of the Airy, KdV and m-KdV equations, in \cite{Feola}, \cite{Feola-Procesi} for quasi-linear Hamiltonian and reversible NLS equations, in \cite{Montalto} for the Kirchhoff equation and in \cite{BM16}, \cite{BertiMontalto} for the water waves equations. The key idea in these series of papers is to split the reduction to constant coefficients of the linearized equation into two parts: the first part is to reduce the equation to another one which is constant coefficients plus a bounded remainder and this is inspired by the breakthrough result of Iooss, Plotnikov and Toland \cite{Ioo-Plo-Tol}. In a second step, one applies a convergent KAM reducibility scheme which reduces quadratically the size of the perturbation and completes the diagonalization of the equation. This method has been extended also by Bambusi in \cite{Bambusi1}, \cite{Bambusi2} to deal with unbounded quasi-periodic perturbations of the Schr\"odinger operator on the real line. 

\noindent
Another difficulty for the reduction procedures and the KAM schemes concerns the multiplicity of the eigenvalues of the unperturbed part of the equation. The first result in this direction is due to Chierchia-You \cite{ChierchiaYou} in which the authors prove a KAM result for analytic bounded perturbations of nonlinear wave equations with periodic boundary conditions (double eigenvalues). We mention also the more recent papers \cite{BertiKappelerMontalto}, \cite{Feola},  \cite{Montalto} concerning Schr\"odinger and Kirchhoff equations with periodic boundary conditions. 

\noindent
There are very few results for PDEs in higher space dimension since the second order Melnikov non-resonance conditions are violated, typically due to the high multiplicity of the eigenvalues. The first KAM and reducibility results in higher space dimension have been obtained by Eliasson-Kuksin \cite{EK1}, \cite{EK} for the linear Schr\"odinger equation on $\T^d$ with a multiplicative analytic potential and for the nonlinear Schr\"odinger equation with a convolution potential. The second order Melnikov non resonance conditions are verified {\it blockwise}, by introducing the notion of T\"oplitz-Lipschitz Hamiltonians. A KAM result for the completely resonant Nonlinear Schr\"odinger equation on $\T^d$ has been proved by Procesi-Procesi \cite{PP}, by using Quasi-T\"oplitz Hamiltonians. We also mention the KAM theorem for the beam equation obtained by Eliasson-Grebert-Kuksin in \cite{KAMbeam}. Recently, Grebert and Paturel \cite{GrebertPaturel} proved a reducibility result for the quantum harmonic oscillator on $\R^d$ with an analytic multiplicative potential and in \cite{Grebert-Paturel-sfera} they proved a KAM result for the nonlinear Klein Gordon equation on the $d$-dimensional sphere. In \cite{BB12}, \cite{BBo10}, \cite{BCP}, the authors proved the existence of quasi-periodic solutions for Nonlinear wave and Schr\"odinger equations on $\T^d$ and on Lie groups, by using the {\it multiscale method}, introduced by Bourgain \cite{Bo1}, \cite{B3}, \cite{B5} in the analytic setup. This approach does not imply the linear stability of the quasi-periodic solutions since it requires to impose only the {\it first order Melnikov conditions}. 

\noindent
The reduciblity for the quasi-periodically forced Klein-Gordon equation with a small multiplicative potential $\partial_{tt} u - \Delta u + m u + \e V(\omega t, x) u = 0$ on $\T^d$ is still open. In \cite{almostreducibility}, Eliasson-Grebert-Kuksin proved that this equation is {\it almost reducible} in the sense that it can be reduced to constant coefficients up to a {\it small remainder}. The aim of the present paper is to provide a class of linear wave equations with unbounded perturbations on $\T^d$ which are reducible. We point out that the main difference between Schr\"odinger and wave (Klein-Gordon) equations is the following: for the Schr\"odinger equation, the eigenvalues of the linear part of the equation grow like $\sim |j|^2, j \in \Z^d$, whereas the wave equation, written as a first order system in complex coordinates, has eigenvalues growing as $\sim |j|, j \in \Z^d$. It turns out that the second order Melnikov non-resonance conditions 
\begin{equation}\label{Melnikov conditions standard}
|\omega \cdot \ell + \mu_j - \mu_{j'}| \geq \frac{\gamma}{\langle \ell \rangle^\tau}\,, \quad \forall (\ell, j, j') \in \Z^\nu \times \Z^d \times \Z^d\,, \quad (\ell, |j|, |j'|) \neq (0, |j|, |j|) 
\end{equation}
in the case of the wave (Klein Gordon) equation, i.e. $\mu_j \sim |j|$, $j \in \Z^d$ are violated. 


\medskip

\noindent
In the following we shall explain the main ideas of the proof of Theorem \ref{theorem linear stability}. The proof consists in reducing the quasi-periodically forced linear vector field ${\cal L}(\omega t)$ defined in \eqref{campo vettoriale main equation} to a time-independent block-diagonal operator. This reduction procedure is split into two parts. 

\medskip

\noindent
{\it Regularization of the vector field ${\cal L}(\omega t)$.} Our first goal is to conjugate the vector field ${\cal L}(\omega t)$ to another one which is diagonal up to a sufficiently regularizing perturbation. This is achieved by using a change of variables induced by a reparametrization of time (so that the highest order term has constant coefficients) and time dependent Fourier multipliers (introduced in Section \ref{sezione pseudo diff}), see Section \ref{riduzione linearizzato}. We point out that this procedure involve only a reduction in time, since our unbounded perturbation ${\cal P}(\omega t)$ is assumed to be diagonal in space up to the finite rank operator ${\cal R}(\omega t)$, which is already regularizing, see \eqref{perturbazione P omega t}, \eqref{definizione perturbazione rango finito}. 

\medskip

\noindent
{\it KAM reducibility scheme.} After the preliminary reduction of the order of derivatives, we deal with a time dependent vector field which is a small and regularizing perturbation of a diagonal time-independent vector field. We then perform a KAM reducibility scheme, see Theorem \ref{thm:abstract linear reducibility}. The key feature of the scheme is that since the perturbation is regularizing, along the KAM iteration, we can impose non-resonance conditions with a {\it loss of derivatives} in space, namely 
\begin{equation}\label{Melnikov conditions standard modificata}
|\omega \cdot \ell + \mu_j - \mu_{j'}| \geq \frac{\gamma}{| j |^{\mathtt d} | j' |^{\mathtt d}\langle \ell \rangle^\tau}\,, \quad \forall (\ell, j, j') \in \Z^\nu \times (\Z^d \setminus \{ 0 \}) \times (\Z^d \setminus \{ 0 \})\,, \quad (\ell, |j|, |j'|) \neq (0, |j|, |j|) 
\end{equation}
for some constant exponents $\mathtt d$ and $\tau$ large enough and $\gamma \in (0, 1)$. Neverthless, all the canonical transformations defined along the iteration will be bounded linear operators (on Sobolev spaces), since the regularizing property of the remainder balances the {\it loss of space derivatives} in the Melnikov conditions \eqref{Melnikov conditions standard modificata}. This strategy has been used also in \cite{KAM ww f depth}, to prove a KAM result for gravity water waves in finite depth without capillarity and we implement it within this context. 

\noindent
The conditions \eqref{Melnikov conditions standard modificata} are much weaker that the ones given in \eqref{Melnikov conditions standard} and we are able to prove that they are fullfilled for a {\it large set} of parameters $\omega$. We use the block-decay norm $| \cdot |_s$ (see \eqref{decadimento Kirchoff}) to estimate the size of the remainders along the iteration. This is convenient since the class of operators having finite block-decay norm is closed under composition (Lemma \ref{interpolazione decadimento Kirchoff}), solution of the homological equation (Lemma \ref{homologica equation}) and projections (Lemma \ref{lemma smoothing decay}). This norm is well adapted to finite rank operators of the form \eqref{definizione perturbazione rango finito} and it gives a strong decay of the blocks arising in the spectral decomposition with respect to the eigenspaces of the operator $\sqrt{- \Delta}$, see Sections \ref{sezione algebrica blocchi}, \ref{sezione norma decadimento a blocchi}. 

\medskip

\noindent
The paper is organized as follows. In Section \ref{sec:2} we introduce some notations and abstract technical tools needed along the proof of Theorem \ref{theorem linear stability}. The proof of the Theorem is developed in Sections \ref{riduzione linearizzato}-\ref{sezione stime in misura}. In Section \ref{riduzione linearizzato} we perform the regularization procedure for the linear Hamiltonian vector field ${\cal L}$ and we conjugate it to the vector field ${\cal L}_4$, defined in \eqref{cal L5}. In Section \ref{sec:redu}, we prove the block-diagonal reducibility of the vector field ${\cal L}_4$, showing that it is conjugated to the block diagonal operator ${\cal D}_\infty$ defined in \eqref{cal D infinito}. In Section \ref{sezione stime in misura} we provide the measure estimate of the set of {\it good parameters} $\Omega_\infty^{2 \gamma}$ defined in \eqref{Omegainfty}. Finally, in Section \ref{conclusione proof}, we conclude the proof of Theorem \ref{theorem linear stability} and we prove the Corollary \ref{crescita norme di sobolev}. 

\bigskip

\noindent
{\it Acknowledgements}. I warmly thank Massimiliano Berti, Giuseppe Genovese, Benoit Grebert, Thomas Kappeler, Sergei Kuksin and Michela Procesi for many useful discussions and comments.

\section{Function spaces, linear operators and norms}\label{sec:2} 
For a function $u \in L^2_0(\T^d) \equiv L^2_0(\T^d, \C)$ we consider its Fourier series 
\begin{equation}\label{Fourier Td}
u(x) = \sum_{j \in \Z^d \setminus \{ 0 \}} u_j e^{\ii j \cdot x}\,, \qquad u_j := \frac{1}{(2 \pi)^d} \int_{\T^d} u(x) e^{- \ii j \cdot x}\, d x \,, \quad \forall j \in \Z^d \setminus \{ 0 \}\,. 
\end{equation}
We denote by $\sigma_0(\sqrt{- \Delta})$ the spectrum of the operator $\sqrt{- \Delta}$ restricted to the zero-average functions, i.e. 
\begin{equation}\label{definizione cal N}
\sigma_0(\sqrt{- \Delta}) := \Big\{ |j | = \sqrt{j_1^2 + \ldots + j_d^2} : j = (j_1, \ldots, j_d) \in \Z^d \setminus \{ 0 \} \Big\}\,
\end{equation}
and for any eigenvalue $\alpha \in \sigma_0(\sqrt{- \Delta})$, we denote by ${\mathbb E}_\alpha$ the corresponding eigenspace, i.e. 
\begin{equation}\label{bf E alpha}
{\mathbb E}_\alpha := {\rm span} \{ e^{\ii j \cdot x} : j \in \Z^d\,, \quad |j| = \alpha\,\}\,.
\end{equation}
Then, any function $u \in L^2_0(\T^d)$ can be written as 

\begin{equation}\label{raggruppamento modi Fourier}
u(x) = \sum_{\alpha \in \sigma_0(\sqrt{- \Delta})} {\mathfrak u}_\alpha(x)\,, \quad {\mathfrak u}_\alpha(x) = \sum_{|j| = \alpha} u_j e^{\ii j \cdot x} \in {\mathbb E}_\alpha
\end{equation}
and if $u \in H^s_0(\T^d)$ for some $s \geq 0$, one has 
\begin{align}
\| u \|_{H^s_x}^2 & = \sum_{\begin{subarray}{c}
 j \in \Z^d \setminus \{ 0 \} 
 \end{subarray}} |  j |^{2 s} |u_j|^2  = \sum_{ \alpha \in \sigma_0(\sqrt{- \Delta})} \alpha^{2 s}  \sum_{|j| = \alpha} |u_j|^2  = \sum_{\alpha \in \sigma_0(\sqrt{- \Delta})} \alpha^{2 s} \| \mathfrak u_\alpha\|_{L^2_x}^2\,. \label{altro modo norma s}
\end{align}

We also deal with functions $ u \in L^2_0 (\T^\nu \times \T^d ) = L^2(\T^\nu, L^2_0(\T^d))$ which can be regarded as
 $ \vphi $-dependent family of  functions $ u(\vphi, \cdot ) \in L^2_0 (\T^d) $ that we expand in Fourier series as
\be
u(\vphi, x ) =   \sum_{j \in \Z^d \setminus \{ 0 \}  } u_{j} (\vphi) e^{\ii j \cdot x } =
\sum_{\begin{subarray}{c}
\ell \in \Z^\nu \\
 j \in \Z^d \setminus \{ 0 \} 
 \end{subarray}}  \widehat u_j(\ell)  e^{\ii (\ell \cdot \vphi + j \cdot  x)}   
\ee
where 
$$
u_j(\vphi) := \frac{1}{(2 \pi)^d} \int_{\T^d} u(\vphi, x) e^{- \ii j \cdot x}\, d x\,, \quad \widehat u_j(\ell) :=  \frac{1}{(2 \pi)^{\nu + d}} \int_{\T^{\nu + d}} u(\vphi, x) e^{- \ii (\ell \cdot \vphi + j \cdot  x)}\,d \vphi\, d x\,.
$$
According to \eqref{raggruppamento modi Fourier}, we can write 
\begin{equation}\label{bf h ell alpha 0}
u(\vphi, x) = \sum_{\alpha \in \sigma_0(\sqrt{- \Delta})} {\mathfrak u}_\alpha(\vphi, x)  = \sum_{\begin{subarray}{c}
\ell \in \Z^\nu \\
\alpha \in \sigma_0(\sqrt{- \Delta})
\end{subarray}} \widehat{\mathfrak u}_\alpha(\ell) e^{\ii \ell \cdot \vphi}
\end{equation}
where 
\begin{equation}\label{bf h ell alpha}
{\mathfrak u}_\alpha(\vphi, x) := \sum_{|j| = \alpha} u_j(\vphi) e^{\ii j \cdot x}\,, \quad \widehat{\mathfrak u}_\alpha(\ell) \equiv \widehat{\mathfrak u}_\alpha(\ell, x)  := \frac{1}{(2 \pi)^\nu} \int_{\T^\nu} {\mathfrak u}_\alpha(\vphi,x) e^{- \ii \ell \cdot \vphi}\, d \vphi = \sum_{|j| = \alpha} \widehat u_j(\ell) e^{\ii j \cdot x}\,.
\end{equation}

We define for any $s \geq 0$ the Sobolev spaces $H^s_0(\T^{\nu + d}) = H^s_0(\T^{\nu + d}, \C)$ as 
\begin{equation}\label{Sobolev media nulla}
H^s_0(\T^{\nu + d}) := \big\{ u \in L^2_0 (\T^\nu \times \T^d ) : \| u \|_{s}^2 := \sum_{\begin{subarray}{c}
\ell \in \Z^\nu \\
j \in \Z^d \setminus \{ 0 \}
\end{subarray}} \langle \ell, j \rangle^{2 s} |\widehat u_j(\ell)|^2 < +\infty \big\}\,, 
\end{equation}
where $ \langle \ell, j \rangle := {\rm max}\{1 , |\ell|, |j| \}$, and for any $\ell = (\ell_1, \ldots, \ell_\nu) \in \Z^\nu$, $|\ell| := \sqrt{\ell_1^2 + \ldots + \ell_\nu^2}$. 
One has 
\begin{equation}\label{altro modo norma s vphi x}
\| u \|_s^2 = \sum_{\begin{subarray}{c}
\ell \in \Z^\nu \\
j \in \Z^d \setminus \{ 0 \}
\end{subarray}} \langle \ell, j \rangle^{2 s} |\widehat u_j(\ell)|^2 = \sum_{\begin{subarray}{c}
\ell \in \Z^\nu \\
\alpha \in \sigma_0(\sqrt{- \Delta}) 
\end{subarray}} \langle \ell, \alpha \rangle^{2 s} \sum_{|j| = \alpha} |\widehat u_j(\ell)|^2 = \sum_{\begin{subarray}{c}
\ell \in \Z^\nu \\
\alpha \in \sigma_0(\sqrt{- \Delta}) 
\end{subarray}} \langle \ell, \alpha \rangle^{2 s} \| \widehat{\mathfrak u}_\alpha(\ell) \|_{L^2}^2 
\end{equation}
where  $\langle \ell, \alpha \rangle := {\rm max}\{ 1, |\ell|, \alpha\}$, for any $\ell \in \Z^\nu, \alpha \in \sigma_0(\sqrt{- \Delta})$. 

\noindent
In a similar way, we define the  spaces of real valued functions $L^2_0(\T^d, \R)$, $L^2_0(\T^{\nu + d}, \R)$, $H^s_0(\T^d, \R)$, $H^s_0(\T^{\nu + d}, \R)$ and we also deal with Sobolev functions $x$-independent, belonging to the Sobolev space $H^s(\T^\nu)$ (or $H^s(\T^\nu, \R)$). For $u \in H^s(\T^\nu)$ we denote by $\| u \|_s$ its Sobolev norm, given by 
$$
\| u \|_s := \sum_{\ell \in \Z^\nu}\langle \ell \rangle^{2 s} |\widehat u(\ell)|^2\,, \qquad \widehat u(\ell) := \frac{1}{(2 \pi)^\nu} \int_{\T^\nu} u(\vphi) e^{- \ii \ell \cdot \vphi}\, d \vphi\,.  
$$ 

\noindent
Given a Banach space $(E, \| \cdot \|_E )$, we denote by $L^\infty(\T^\nu, E)$ the space of the essentially bounded functions $\T^\nu \to E$ equipped with the norm 
$$
\| u \|_{L^\infty(\T^\nu, E)} := {\rm esssup}_{\vphi \in \T^\nu} \|u(\vphi) \|_E\,.
$$ 
For any $p \in \N$ we denote by $W^{p, \infty}(\T^\nu, E)$ the space of the p-times weakly differentiable functions $\T^\nu \to E$ equipped with the norm 
$$
 \quad \| u \|_{W^{p, \infty}(\T^\nu, E)} := {\rm max}_{|a| \leq p}  \| \partial_\vphi^a u \|_{L^\infty(\T^\nu, E)}\,.
$$
In the above formula, for any multi-index $a = (a_1, \ldots, a_\nu) \in \N^\nu$, we use the notations $|a| := a_1 + \ldots + a_\nu$ and $\partial_\vphi^a = \partial_{\vphi_1}^{a_1} \ldots \partial_{\vphi_\nu}^{a_\nu}$. We also denote by ${\cal C}^0(\T^\nu, E)$ the space of continuous functions $\T^\nu \to E$ equipped with the norm 
$$
\| u \|_{{\cal C}^0(\T^\nu, E)} := {\rm sup}_{\vphi \in \T^\nu} \|u (\vphi) \|_E
$$ 
and we denote by ${\cal C}^p(\T^\nu, E)$ the space of the $p$-times differentiable functions with continuous derivatives equipped with the norm 
$$
\| u \|_{{\cal C}^p(\T^\nu, E)} := {\rm max}_{|a| \leq p} \| \partial_\vphi^a u \|_{{\cal C}^0(\T^\nu, E)}\,. 
$$
We recall the standard property 
\begin{equation}\label{immersione W k infinito Ck}
W^{p + 1, \infty}(\T^\nu, E) \subset {\cal C}^p(\T^\nu, E)\,. 
\end{equation}


\medskip

\noindent
 For a function $f : \Omega_o \to E$, $\omega \mapsto f(\omega)$, where $(E, \| \cdot \|_E)$ is a Banach space and 
$ \Omega_o $ is a subset of $\R^\nu$, we define the sup-norm and the lipschitz semi-norm as 
\be \label{def norma sup lip}
\| f \|^{\sup}_{E, \Omega_o} 
:= \sup_{ \omega \in \Omega_o } \| f(\omega) \|_E \,,\quad \| f \|_{E, \Omega_o}^{\lip} := \sup_{\begin{subarray}{c}
\omega_1, \omega_2 \in \Omega_o \\
\omega_1 \neq \omega_2
\end{subarray}} \frac{\| f(\omega_1) - f(\omega_2)\|_E}{|\omega_1 - \omega_2|}
\ee
and, for $ \g > 0 $, we define the weighted Lipschitz-norm
\be \label{def norma Lipg}
\| f \|^{\Lipg}_{E, \Omega_o}  
:= \| f \|^{\sup}_{E, \Omega_o} + \g \|  f \|^{\lip}_{E, \Omega_o}  \, . 
\ee
To shorten the above notations we simply omit to write $\Omega_o$, namely $\| f \|^{\sup}_{E} = \| f \|^{\sup}_{E, \Omega_o}$, $\| f \|_{E}^{\lip} = \| f \|_{E, \Omega_o}^{\lip}$, $\| f \|^{\Lipg}_{E} = \| f \|^{\Lipg}_{E, \Omega_o}$.
If $f : \Omega_o \to \C$, we simply denote $\| f \|_{\C}^{\Lipg}$ by $|f|^\Lipg$
and if $ E = H^s(\T^{\nu + d}) $ we simply denote $ \| f \|^{\Lipg}_{H^s} := \| f \|^{\Lipg}_s $. Given two Banach spaces $E, F$, we denote by ${\mathcal B}(E, F)$ the space of the bounded linear operators $E \to F$. If $E = F$, we simply write ${\mathcal B}(E)$. 

\medskip

\noindent
{\it Notation:} From now on we fix 
\begin{equation}\label{definition s0}
s_0 := \Big[\frac{\nu+ d}{2} \Big] + 1
\end{equation}
where for any real number $x \in \R$, we denote by $[x]$ its integer part. We write  
$$
a \lesssim_s b \quad \ \Longleftrightarrow \quad a \leq C(s) b
$$ 
for some constant $ C(s) $ depending on the data of the problem, namely the Sobolev norms $\| a \|_{s}, \| b_k \|_{s}, \| c_k \|_s$ of the functions $a, b_k, c_k$ appearing in \eqref{perturbazione P omega t}, the number $ \nu $ of frequencies, the dimension $d$ of the space variable $x$, the diophantine exponent $ \tau >  0 $  in the non-resonance conditions, which will be required along the proof.  For $ s = s_0$  we only write $ a \lesssim b $. 
Also the small constants $ \d $ in the sequel depend on the data of the problem.

\medskip

\noindent
We recall the classical estimates for the operator $(\omega \cdot \partial_\vphi)^{- 1}$ defined as 
\begin{equation}\label{om d vphi inverso}
(\omega \cdot \partial_\vphi)^{- 1}[1] = 0\,, \quad (\omega \cdot \partial_\vphi)^{- 1}[e^{\ii \ell \cdot \vphi}] = \frac{1}{\ii (\omega \cdot \ell)} e^{\ii \ell \cdot \vphi}\,, \qquad \forall \ell \neq 0\,,
\end{equation}
for $\omega \in DC(\gamma, \tau)$, where for $\gamma, \tau > 0$, 
\begin{equation}\label{diofantei Kn}
DC(\gamma, \tau) := \Big\{ \omega \in \Omega : |\omega \cdot \ell| \geq \frac{\gamma}{| \ell |^\tau}\,, \quad \forall \ell \in \Z^\nu \setminus \{ 0 \} \Big\}\,. 
\end{equation}
 If $h(\cdot ; \omega) \in H^{s + 2 \tau + 1}(\T^{\nu})$, with $\omega \in DC(\gamma, \tau)$, we have 
\begin{equation}\label{stima om d vphi inverso}
\| (\omega \cdot \partial_\vphi)^{- 1} h \|_{s} \leq \gamma^{- 1} \| h \|_{s + \tau}\,, \qquad \| (\omega \cdot \partial_\vphi)^{- 1} h \|_{s}^\Lipg \leq \gamma^{- 1} \| h \|_{s + 2 \tau + 1}^\Lipg\,.
\end{equation}

\noindent
\noindent
We also recall some classical Lemmas on the composition operators and on the interpolation. Since the variables $ (\vphi, x)$ have the same role, we present it for a  generic Sobolev space  $ H^s (\T^n ) $. For any $s \geq 0$ integer, for any domain $A \subseteq \R^n$ we denote by ${\mathcal C}^s(A)$ the space of the s-times continuously differentiable functions equipped by the usual $\| \cdot \|_{{\mathcal C}^s}$ norm.

\begin{lemma}\label{interpolazione C1 gamma}{\bf (Interpolation)}
Let $u, v \in H^s(\T^{n})$ with $s \geq s_n$, $s_n := [n/ 2] + 1$. Then, there exists an increasing function $s \mapsto C(s)$ such that 
$$
\| u v \|_{s} \leq C(s) \| u \|_{s} \| v \|_{s_n}+ C(s_n)\| u \|_{s_n} \| v \|_{s}\,.
$$
If $u(\cdot; \omega)$, $v(\cdot; \omega)$, $\omega \in \Omega_o \subseteq \R^\nu$ are $\omega$-dependent families of functions in $H^{s}(\T^{n})$, with $s \geq s_n$ then the same estimate holds replacing $\| \cdot \|_{s}$ by $\| \cdot \|_{s}^\Lipg$.
\end{lemma}
Iterating the above inequality one gets that, for some constant $K(s)$, for any $n \geq 0$, 
\begin{equation}\label{interpolazione iterata}
\| u ^k\|_s \leq K(s)^k \| u\|_{s_0}^{k - 1} \| u \|_s\,
\end{equation}
and if $u(\cdot; \omega) \in H^s$, $s \geq s_n$ is a family of Sobolev functions, the same inequality holds repacing $\| \cdot \|_s$ by $\| \cdot \|_s^\Lipg$.

\noindent
 We consider the composition operator
$$
u(y) \mapsto {\mathtt f}(u)(y) := f(y, u(y))\,. 
$$ 
The following lemma is a classical result due to Moser.
  \begin{lemma}\label{Moser norme pesate} {\bf (Composition operator)}
Let $ f \in {\mathcal C}^{s + 1}(\T^n \times \R, \R )$, with  $   s \geq s_n :=   [n/2] + 1 $. If $u \in H^s(\T^n)$, with $\| u \|_{s_n} \leq 1$, then $\| \mathtt f(u)\|_s \leq C(s, \| f \|_{{\mathcal C}^s})(1 + \| u \|_s)$.  
If   $u(\cdot, \omega) \in H^s(\T^n)$, $\omega \in \Omega_o \subseteq \R^\nu$  is a family of Sobolev functions
satisfying $\| u \|_{s_n}^{\Lipg} \leq 1$, then,  
$  \| {\mathtt f}(u) \|_s^{\Lipg} \leq C(s,  \| f\|_{{\mathcal C}^{s + 1}} ) ( 1 + \| u \|_{s}^{\Lipg}) $. 
\end{lemma}
Now we state the tame properties of the composition operator $ u(y) \mapsto  u(y+p(y)) $
induced by a diffeomorphism of the torus $ \T^n $. The Lemma below, can be proved as Lemma 2.21 in \cite{BertiMontalto}.
\begin{lemma} {\bf (Change of variables)}  \label{lemma:utile} 
Let $p:= p( \cdot; \omega ):\R^n \to \R^n$, $\omega \in \Omega_o \subset \R^\nu $ be a family of $2\p$-periodic functions satisfying   
\begin{equation}\label{mille condizioni p}
  \| p \|_{{\mathcal C}^{s_n + 1}} \leq 1/2\,,\quad  \| p \|_{s_n}^{\Lipg} \leq 1
\end{equation}
where $s_n := [n/2] + 1$. Let $ g(y) := y + p(y) $,
$ y \in \T^n $. 
Then the composition operator 
$$
A : u(y) \mapsto (u\circ g)(y) = u(y+p(y))
$$ 
satisfies for all $s \geq s_n$, the tame estimates
\begin{equation}\label{stima cambio di variabile dentro la dim}
\| A u\|_{s_n} \lesssim_{s_n} \| u \|_{s_n}\,, \qquad \| A u\|_s \leq  C(s)\| u \|_s + C(s_n)\| p \|_s \| u \|_{s_n + 1}\,.
\end{equation}
Moreover, for any family of Sobolev functions $u(\cdot; \omega)$
\begin{align}
& \| A u \|_{s_n}^{\Lipg} \lesssim_{s_n}  \| u \|_{s_n + 1}^{\Lipg}\,, \label{stima tame cambio di variabile pietro s0}\\
\label{stima tame cambio di variabile pietro}
    & \| Au \|_s^{\Lipg} \lesssim_{s} \| u \|_{s + 1}^{\Lipg} + \| p \|_{s}^{\Lipg} \| u \|_{s_n + 2}^{\Lipg}\,, \quad \forall  s > s_n  \,.
\end{align}
The map $ g $ is invertible with inverse  $ g^{- 1}(z) = z + q(z) $ and 
there exists a constant $\delta := \delta(s_n) \in (0,1) $ such that, if  $ \| p \|_{2 s_n + 2}^{\Lipg} \leq \d$, then  
\begin{equation}\label{stime-lipschitz-q}
\| q \|_s \lesssim_s \| p \|_s\,,\qquad \| q \|_{s}^{\Lipg}  \lesssim_{s}  \| p \|_{s + 1}^{\Lipg} \,. 
\end{equation}
Furthermore, the composition operator $A^{- 1} u(z) := u(z + q(z))$ satisfies the estimate  
\begin{equation}\label{tame-cambio-di-variabile-inverso}
\| A^{- 1} u\|_s \lesssim_s \| u \|_{s} + \| p \|_{s} \| u \|_{s_n + 1}\,, \quad \forall s \geq s_n\,
\end{equation}
and for any family of Sobolev functions $u(\cdot; \omega)$
\begin{equation}\label{tame-lipschitz-cambio-di-variabile}
\| A^{- 1} u\|_s^{\Lipg} \lesssim_s \| u \|_{s + 1}^{\Lipg} + \| p \|_{s + 1}^{\Lipg} \| u \|_{s_n + 2}^{\Lipg}\,, \quad \forall s \geq s_n\,.
\end{equation}
\end{lemma}

\subsection{Linear operators}
Let $ {\cal R} \in {\cal B}( L^2_0(\T^d))  $. The action of this operator on a function $ u \in L^2_0(\T^{d}) $ 
is given by  
\begin{equation}\label{action toplitz operator}
{\cal R} [u]  = \sum_{j , j' \in \Z^d \setminus \{ 0 \}} {\cal R}_j^{j'}  u_{j'}e^{\ii j \cdot  x} 
\ee
where the Fourier coefficients ${\cal R}_j^{j'}$ of ${\cal R}$ are defined in \eqref{coefficienti spazio operatore}.  
We shall identify the operator $ {\cal R} $ with the infinite-dimensional matrix of its Fourier coefficiens
\begin{equation}\label{matrice operatore Toplitz}
\Big({\cal R}_j^{j'}\Big)_{\begin{subarray}{c}
 j, j' \in \Z^d \setminus \{ 0 \} 
\end{subarray}}\,.
\end{equation}

 \noindent
 We define the conjugated operator $\overline{\cal R}$ by 
 \begin{equation}\label{definizione operatore coniugato}
 \overline{\cal R} u := \overline{{\cal R} \bar u}\,.
 \end{equation}
 One gets easily that the operator $\overline{\cal R}$ has the matrix representation 
 \begin{equation}\label{operatore coniugato matrice}
\Big(\overline{{\cal R}_{- j}^{- j'}} \Big)_{j, j' \in \Z^d \setminus \{ 0 \}}\,.
 \end{equation}
 An operator ${\cal R}$ is said to be real if it maps real-valued functions on real valued functions and it is easy to see that ${\cal R}$ is real if and only if ${\cal R} = \overline{\cal R}$.
 
 \noindent
 We define also the transpose operator ${\cal R}^T$ by the relation 
 \begin{equation}\label{definizione operatore trasposto}
 \langle {\cal R}[u]\,,\, v \rangle_{L^2_x} = \langle u\,,\, {\cal R}^T[v] \rangle_{L^2_x}\,, \qquad \forall u, v \in L^2_0(\T^d)\,, \quad \forall \vphi \in \T^\nu
 \end{equation}
 where 
 \begin{equation}\label{prodotto scalare reale L2}
 \langle u, v \rangle_{L^2_x} := \int_{\T^d} u(x) v(x)\,,d x \,, \qquad \forall u, v \in L^2_0(\T^d)\,.
 \end{equation}
 Note that the operator ${\cal R}^T$ has the matrix representation 
 \begin{equation}\label{matrice RT}
 ({\cal R}^T)_j^{j'} = {\cal R}_{- j'}^{ - j}\,, \quad \forall j, j' \in \Z^d\,.
 \end{equation}
 An operator ${\cal R}$ is said to be symmetric in ${\cal R} = {\cal R}^T$.
 
 \noindent
 We define also the adjoint operator ${\cal R}^*$ as 
 \begin{equation}\label{operatore aggiunto cal R}
 \big( {\cal R}[u]\,,\, v \big)_{L^2_x} = \big(u\,,\, {\cal R}^* [v] \big)_{L^2_x}\,, \quad \forall u, v \in L^2_0(\T^d)\,,
 \end{equation}
 where $ \big( \cdot\,,\, \cdot \big)_{L^2_x}$ is the scalar product on $L^2_0(\T^d)$, namely 
 \begin{equation}\label{prodotto scalare complesso}
 \big(u\,,\, v \big)_{L^2_x} := \langle u\,,\, \overline v \rangle_{L^2_x} = \int_{\T^d} u(x) \overline{v(x)}\,, d x\,, \quad \forall u, v \in L^2_0(\T^d)\,.
 \end{equation}
 An operator ${\cal R}$ is said to be self-adjoint if ${\cal R} = {\cal R}^*$.
 It is easy to see that ${\cal R}^* = \overline{\cal R}^T$ and its matrix representation is given by 
 $$
 ({\cal R}^*)_j^{j'} = \overline{{\cal R}_{j'}^{j}}\,, \quad \forall j, j' \in \Z^d \setminus \{ 0 \}\,.
 $$
  We also define the {\it commutator} between two linear operators ${\mathcal R}, {\cal T} \in {\cal B}(L^2_0(\T^d))$ by $[{\mathcal R}, {\mathcal T}] := {\mathcal R} {\mathcal T}- {\mathcal T}{\mathcal R}$.

\noindent
In the following we also deal with real operators $G \in {\mathcal B}\Big({ L}^2_0(\T^d, \R) \times { L}^2_0(\T^d, \R) \Big)$, of the form 
 \begin{equation}\label{operatore matriciale reale}
G := \begin{pmatrix}
 A & B \\
 C & D
 \end{pmatrix}
 \end{equation}
 where $A,B,C,D \in {\mathcal B}(L^2_0(\T^d, \R))$. By \eqref{definizione operatore trasposto}, the transpose operator $G^T$ with respect to the bilinear form 
 \begin{equation}\label{prodotto scalare prodotto L2}
\langle (v_1, \psi_1)\,,\, (v_2, \psi_2) \rangle_{{ L}^2_x} := \langle v_1, v_2 \rangle_{L^2_x} + \langle \psi_1\,,\, \psi_2  \rangle_{L^2_x}\,, 
\end{equation} 
$ \forall (u_1, \psi_1), (u_2, \psi_2) \in {L}^2_0(\T^d, \R) \times {L}^2_0(\T^d, \R)$, is given by 
\begin{equation}\label{operatore trasposto matriciale}
G^T  = \begin{pmatrix}
A^T & C^T \\
B^T & D^T
\end{pmatrix}\,.
\end{equation}
Then it is easy to verify that $G$ is symmetric, i.e. $G = G^T$ if and only if $A = A^T$, $B = C^T$, $D = D^T$. It is also convenient to regard the real operator $G$ in the complex variables
\begin{equation}\label{prima volta variabili complesse 0}
(v, \psi) = {\cal C}[(u, \overline u)] \,, \quad (u, \overline u) = {\cal C}^{- 1}[(v, \psi)]  
\end{equation}
where
\begin{equation}\label{matrice coordinate complesse}
{\mathcal C} := \frac{1}{\sqrt{2}} \begin{pmatrix}
1 & 1 \\
\frac{1}{\ii} & - \frac{1}{\ii}
\end{pmatrix}\, \qquad {\mathcal C}^{- 1} = \frac{1}{\sqrt{2}} \begin{pmatrix}
1 & \ii \\
1  & - {\ii}
\end{pmatrix}\,.
\end{equation}
The operators ${\cal C}, {\cal C}^{- 1}$ satisfies 
$$
{\mathcal C} : {\bf L}^2_0(\T^{d}) \to L^2_0(\T^d, \R) \times L^2_0(\T^d, \R)\,, \quad {\mathcal C}^{- 1} : L^2_0(\T^d, \R) \times L^2_0(\T^d, \R) \to {\bf L}^2_0(\T^d) 
$$
where ${\bf L}^2_0(\T^{d})$ is the real subspace of $L^2_0(\T^d) \times L^2_0(\T^d)$ defined by
\begin{align}
& {\bf L}^2_0(\T^{d}) :=\big\{  (u, \overline u) : u \in L^2_0(\T^{d}) \big\}\,. \label{sottospazio reale z bar z} 
\end{align}
If $G \in {\cal B}\Big(L^2_0(\T^d, \R) \times L^2_0(\T^d, \R) \Big)$ is a real operator of the form \eqref{operatore matriciale reale}, one has that the conjugated operator 
$$
{\cal R} : = {\cal C}^{- 1} G {\cal C} : {\bf L}^2_0(\T^d) \to {\bf L}^2_0(\T^d)
$$
 has the form 
\begin{align}\label{operatore trasformato in variabili complesse}
& {\mathcal R} = \begin{pmatrix}
{ R}_1 & {R}_2 \\
\overline{ R}_2 & \overline{R}_1
\end{pmatrix}\,,  \quad R_1 := \frac{A+ D  - \ii (B - C)}{2}\,, \quad R_2 := \frac{A - D + \ii (B + C)}{2}\,. 
\end{align}
For the sequel, we also introduce for any $s \geq 0$, the real subspace of $H^s_0(\T^{d}) \times H^s_0(\T^{d}) $ 
\begin{align}
& {\bf H}^s_0(\T^{d}) := \Big(H^s_0(\T^{d}) \times H^s_0(\T^{d}) \Big) \cap {\bf L}^2_0(\T^{d}) \label{definizione bf H s0} 
\end{align}
and we set 
\begin{equation}\label{norma s u bar u}
\| {\bf u} \|_{{\bf H}^s_x} := \| u \|_{H^s_x}\,, \qquad \forall {\bf u} = (u, \overline u) \in {\bf H}^s_0(\T^d)\,. 
\end{equation}
It is straightforward to verify that for any $s \geq 0$ 
\begin{equation}\label{proprieta coordinate complesse}
{\mathcal C} : {\bf H}^s_0(\T^{d}) \to H^s_0(\T^{d}, \R) \times H^s_0(\T^d, \R)\,, \quad {\mathcal C}^{- 1} :H^s_0(\T^{d}, \R) \times H^s_0(\T^d, \R) \to {\bf H}^s_0(\T^{d}) \,. 
\end{equation}

\subsection{Block representation of linear operators}\label{sezione algebrica blocchi}
We may regard an operator ${\cal R} : L^2_0(\T^d) \to L^2_0(\T^d)$
as a block matrix 
\begin{equation}\label{notazione a blocchi}
\Big( [{\cal R}]_\alpha^\beta\Big)_{\begin{subarray}{c}
\alpha, \beta \in \sigma_0(\sqrt{- \Delta})
\end{subarray}}
\end{equation}
where for all $\alpha, \beta \in \sigma_0(\sqrt{- \Delta})$ (recall \eqref{definizione cal N}), the block-matrix $[{\cal R}]_\alpha^\beta$ is defined by 
\begin{equation}\label{definizione blocco operatore}
[{\cal R}]_\alpha^\beta := \Big( {\cal R}_j^{j'} \Big)_{|j| = \alpha\,,\, |j'| = \beta}\,. 
\end{equation}
Note that the operator $[{\cal R}]_\alpha^\beta$ is a linear operator from ${\mathbb E}_\beta$ onto ${\mathbb E}_\alpha$
where for all $\alpha \in \sigma_0(\sqrt{- \Delta})$, the finite dimensional space ${\mathbb E}_\alpha$ is defined in \eqref{bf E alpha}. 
We identify the space ${\cal B}({\mathbb E}_\beta, {\mathbb E}_\alpha)$ of the linear operators from ${\mathbb E}_\beta$ onto ${\mathbb E}_\alpha$ with the space of the matrices of their Fourier coefficients, namely 
\begin{equation}\label{spazio matrici blocchi}
 {\cal B}({\mathbb E}_\beta, {\mathbb E}_\alpha) \simeq  \Big\{ M = \Big(M_j^{j'} \Big)_{\begin{subarray}{c}
j, j' \in \Z^d \setminus \{ 0 \} \\
|j| = \alpha\,,\, |j'| = \beta
\end{subarray}} \Big\}\,.
\end{equation}
Indeed if $M \in {\cal B}({\mathbb E}_\beta, {\mathbb E}_\alpha)$, its action is given by 
\begin{equation}\label{azione blocco finito dimensionale su funzioni}
M u (x)= \sum_{\begin{subarray}{c}
|j| = \alpha \\
|j'| = \beta 
\end{subarray}} M_j^{j'} u_{j'} e^{\ii j \cdot x}\,, \quad \forall u \in {\mathbb E}_\beta\,, \quad u(x) = \sum_{|j'| = \beta} u_{j'} e^{\ii j' \cdot x}\,.
\end{equation}
If $\beta = \alpha$, we use the notation ${\cal B}({\mathbb E}_\alpha) = {\cal B}({\mathbb E}_\alpha, {\mathbb E}_\alpha)$ and we denote by ${\mathbb I}_\alpha$ the identity operator on the space ${\mathbb E}_\alpha$, namely 
\begin{equation}\label{operatore identita su E alpha}
{\mathbb I}_\alpha : {\mathbb E}_\alpha \to {\mathbb E}_\alpha\,, \qquad u \mapsto u\,.
\end{equation}

\noindent
According to \eqref{raggruppamento modi Fourier}, \eqref{notazione a blocchi}, \eqref{azione blocco finito dimensionale su funzioni}, we may write the action of an operator ${\cal R}$ on a function $u(x)$ as 
\begin{equation}\label{azione operator Toplitz a blocchi}
{\cal R}u  = \sum_{\begin{subarray}{c}
\alpha, \beta \in \sigma_0(\sqrt{- \Delta})
\end{subarray}} [{\cal R}]_\alpha^\beta [{\mathfrak u}_\beta]\,.
\end{equation}
If $[{\cal R}]_\alpha^\beta = 0$, for any $\alpha \neq \beta$, we say that ${\cal R}$ is {\it block-diagonal} and we use the notation 
\begin{equation}\label{notazione operatore diagonale a blocchi}
{\cal R} = {\rm diag}_{\alpha \in \sigma_0(\sqrt{- \Delta})} [{\cal R}]_\alpha^\alpha\,.
\end{equation}
The action of a block-diagonal operator ${\cal R}$ on a function $u \in L^2_0(\T^d)$ is given by 
\begin{equation}\label{definition block diagonal operators}
{\cal R} u = \sum_{ \alpha \in \sigma_0(\sqrt{- \Delta})} [{\cal R}]_\alpha^\alpha[ {\mathfrak u}_\alpha ] \,.
\end{equation}

\noindent
Let $M \in {\cal B}({\mathbb E}_\beta, {\mathbb E}_\alpha)$. The transpose operator $M^T \in {\cal B}({\mathbb E}_\alpha, {\mathbb E}_\beta)$ has the matrix representation
\begin{equation}\label{definizione matrice trasposta}
(M^T)_j^{j'} := M_{- j'}^{- j} \,, \quad |j| = \beta\,,\quad |j'| = \alpha\,.
\end{equation}
The conjugate operator $\overline M \in {\cal B}({\mathbb E}_\beta, {\mathbb E}_\alpha)$ is given by  
\begin{equation}\label{definizione matrice coniugata}
({\overline M})_j^{j'} := \overline{M_{- j}^{ - j'}}\,, \quad |j | = \alpha\,, \quad |j'| = \beta
\end{equation}
and the adjoint operator $M^* \in {\cal B}({\mathbb E}_\alpha, {\mathbb E}_\beta)$ by 
\begin{equation}\label{definizione matrice aggiunta}
M^* := \overline M^T\,.
\end{equation}

Let $\alpha, \beta , \lambda \in \sigma_0(\sqrt{- \Delta})$. Given $A \in {\cal B}({\mathbb E}_\beta, {\mathbb E}_\alpha)$, $B \in {\cal B}({\mathbb E}_\lambda, {\mathbb E}_\beta)$, the operator $A B \in {\cal B}({\mathbb E}_\lambda, {\mathbb E}_\alpha)$ has the matrix representation  
\begin{equation}\label{composizione matrici blocchi}
(A B)_{j}^{j'} := \sum_{|k| = \beta} A_j^k B_k^{j'}\,, \quad \forall |j| = \alpha\,, \quad |j'| = \lambda\,.
\end{equation}
Given an operator $A \in {\cal B}({\mathbb E}_\alpha)$, we define its trace as 
\begin{equation}\label{definizione traccia}
{\rm Tr}(A) := \sum_{|j| = \alpha} A_j^j\,.
\end{equation}
It is easy to check that if $A, B \in {\cal B}({\mathbb E}_\alpha)$, then 
\begin{equation}\label{proprieta traccia}
{\rm Tr}(A B) = {\rm Tr}(B A)\,.
\end{equation}
For all $\alpha, \beta \in \sigma_0(\sqrt{- \Delta})$, the space ${\cal B}({\mathbb E}_\beta, {\mathbb E}_\alpha)$ defined in \eqref{spazio matrici blocchi}, is a Hilbert space equipped by the inner product given for any $X, Y \in {\cal B}({\mathbb E}_\beta, {\mathbb E}_\alpha)$ by
\begin{equation}\label{prodotto scalare traccia matrici}
\langle X, Y \rangle := {\rm Tr}(X Y^*)\,.
\end{equation}
This scalar product induces the {\it Hilbert-Schmidt} norm 
\begin{equation}\label{norma L2 blocco}
\| X \|_{HS } := \sqrt{{\rm Tr}(X X^*)} = \Big( \sum_{\begin{subarray}{c}
|j| = \alpha \\
|j'| = \beta
\end{subarray}} |X_j^{j'}|^2 \Big)^{\frac12}\,.
\end{equation}
For any operator $X \in {\cal B}({\mathbb E}_\beta, {\mathbb E}_\alpha)$, we define also the operator norm as 
\begin{equation}\label{norma operatoriale matrice alpha beta}
\| X \|_{{\cal B}({\mathbb E}_\beta, {\mathbb E}_\alpha)} := {\rm sup}\big\{  \| X u\|_{L^2} : u \in {\mathbb E}_\beta \,,\quad \| u \|_{L^2} \leq 1\big\}\,.
\end{equation}
First we recall some preliminary properties of these norms. 
\begin{lemma}\label{proprieta facili norma L2 matrici}
$(i)$ Let $\alpha, \beta \in \sigma_0(\sqrt{- \Delta})$, $M\in {\cal B}({\mathbb E}_\beta, {\mathbb E}_\alpha)$ and $u \in {\mathbb E}_\beta$. Then $\| M u \|_{L^2} \leq \| M\|_{HS} \| u \|_{L^2}$, implying that $\| M\|_{{\cal B}({\mathbb E}_\beta, {\mathbb E}_\alpha)} \leq \| M\|_{HS}$.

\noindent
$(ii)$ Let $\alpha, \beta , \lambda \in \sigma_0(\sqrt{- \Delta})$,  $M \in {\cal B}({\mathbb E}_\beta, {\mathbb E}_\alpha)$, $X \in {\cal B}({\mathbb E}_\lambda, {\mathbb E}_\beta)$. Then $\| M X\|_{HS} \leq \| M\|_{HS} \| X\|_{HS}$. 
\end{lemma}
\begin{proof}
The proof is a straightforward application of the Cauchy-Schwartz inequality.
\end{proof}
\noindent
Given a linear operator ${ L} : {\cal B}({\mathbb E}_\beta, {\mathbb E}_\alpha) \to {\cal B}({\mathbb E}_\beta, {\mathbb E}_\alpha)$, we denote by $\| L \|_{{\rm Op}(\alpha, \beta)}$ its operator norm, when the space ${\cal B}({\mathbb E}_\beta, {\mathbb E}_\alpha)$ is equipped with the Hilbert-Schmidt norm \eqref{norma L2 blocco}, namely
\begin{equation}\label{norma operatoriale su matrici alpha beta}
\| L\|_{{\rm Op}(\alpha, \beta)} := \sup\Big\{  \| L(M) \|_{HS} : M \in {\cal B}({\mathbb E}_\beta, {\mathbb E}_\alpha)\,, \quad \| M\|_{HS} \leq 1\Big\}\,.
\end{equation} 
We denote by ${\mathbb I}_{\alpha, \beta}$ the identity operator on ${\cal B}({\mathbb E}_\beta, {\mathbb E}_\alpha)$, namely 
\begin{equation}\label{operatore identita matrici alpha beta}
{\mathbb I}_{\alpha, \beta} : {\cal B}({\mathbb E}_\beta, {\mathbb E}_\alpha) \to {\cal B}({\mathbb E}_\beta, {\mathbb E}_\alpha)\,, \qquad X \mapsto X\,.
\end{equation}
For any operator $A \in {\cal B}({\mathbb E}_\alpha)$ we denote by $M_L(A) : {\cal B}({\mathbb E}_\beta, {\mathbb E}_\alpha) \to {\cal B}({\mathbb E}_\beta, {\mathbb E}_\alpha)$ the linear operator defined for any $X \in {\cal B}({\mathbb E}_\beta, {\mathbb E}_\alpha)$ as 
\begin{equation}\label{definizione moltiplicazione sinistra matrici}
M_L(A) X := A X\,.
\end{equation} 
Similarly, given an operator $B \in {\cal B}({\mathbb E}_\beta)$, we denote by $M_R(B) : {\cal B}({\mathbb E}_\beta, {\mathbb E}_\alpha) \to {\cal B}({\mathbb E}_\beta, {\mathbb E}_\alpha)$ the linear operator defined for any $X \in {\cal B}({\mathbb E}_\beta, {\mathbb E}_\alpha)$ as 
\begin{equation}\label{definizione moltiplicazione destra matrici}
M_R(B) X := X B\,.
\end{equation}
By Lemma \ref{proprieta facili norma L2 matrici}-$(ii)$, we have 
\begin{equation}\label{norma operatoriale ML MR}
\| M_L(A)\|_{{\rm Op}(\alpha, \beta)} \leq \| A\|_{HS}\,, \quad \| M_R(B)\|_{{\rm Op}(\alpha, \beta)} \leq \| B\|_{HS}\,.
\end{equation}
For any $\alpha \in \sigma_0(\sqrt{- \Delta})$, we denote by ${\cal S}({\mathbb E}_\alpha)$, the set of the self-adjoint operators form ${\mathbb E}_\alpha$ onto itself, namely
\begin{equation}\label{cal S E alpha}
{\cal S}({\mathbb E}_\alpha) := \Big\{ A \in {\cal B}({\mathbb E}_\alpha) : A = A^*\Big\}
\end{equation}
and given $A \in {\cal B}({\mathbb E}_\alpha)$ denote by ${\rm spec}(A)$ the spectrum of $A$. The next Lemma follows by elementary arguments of linear algebra and hence its proof is omitted. 
\begin{lemma}\label{properties operators matrices}
Let $A \in {\cal S}({\mathbb E}_\alpha)$, $B \in {\cal S}({\mathbb E}_\beta)$, then the following holds: 

\noindent
$(i)$ The operators $M_L(A)$, $M_R(B)$ defined in \eqref{definizione moltiplicazione sinistra matrici}, \eqref{definizione moltiplicazione destra matrici} are self-adjoint operators with respect to the scalar product defined in \eqref{prodotto scalare traccia matrici}.

\noindent
\noindent
$(ii)$ The spectrum of the operator $M_L(A) \pm M_R(B)$ satisfies 
$$
{\rm spec}\Big( M_L(A) \pm M_R(B) \Big) = \Big\{ \lambda \pm \mu : \lambda \in {\rm spec}(A)\,,\quad \mu \in {\rm spec}(B) \Big\}\,.
$$

%

%
\end{lemma}
We also deal with smooth $\vphi$-dependent families of linear operators 
\begin{equation}\label{matrice operatore Toplitz vphi}
{\cal R} : \T^\nu \to {\cal B}(L^2_0(\T^d))\,, \quad \vphi \mapsto {\cal R}(\vphi)\,.
\end{equation}
According to \eqref{matrice operatore Toplitz}, for any $\vphi \in \T^\nu$, the operator ${\cal R}(\vphi)$ has the matrix representation $({\cal R}_j^{j'}(\vphi))_{j, j' \in \Z^d \setminus \{ 0 \}}$.
 We can write the Fourier expansions  
$$
{\cal R}(\vphi) = \sum_{\ell \in \Z^\nu} \widehat{\cal R}(\ell) e^{\ii \ell \cdot \vphi}\,, \quad {\cal R}_j^{j'}(\vphi) = \sum_{\ell \in \Z^\nu } \widehat{\cal R}_j^{j'}(\ell) e^{\ii \ell \cdot \vphi}\,, \quad \forall \ell \in \Z^\nu\,, \quad \forall j, j' \in \Z^d \setminus \{ 0 \}
$$ 
where 
\begin{equation}\label{widehat cal R (ell)}
\widehat{\cal R}(\ell) := \frac{1}{(2 \pi)^\nu} \int_{\T^\nu} {\cal R}(\vphi) e^{- \ii \ell \cdot \vphi}\, d \vphi \in {\cal B}(L^2_0(\T^d))\,, \quad \forall \ell \in \Z^\nu\,,
\end{equation} 
\begin{equation}\label{coefficienti spazio tempo operatore}
\widehat{\cal R}_j^{j'}(\ell) := \frac{1}{(2 \pi)^{\nu }}  \int_{\T^\nu} {\cal R}_j^{j'}(\vphi) e^{- \ii \ell \cdot \vphi }\, d \vphi \,, \qquad \forall \ell \in \Z^\nu\,, \quad \forall j, j' \in \Z^d \setminus \{ 0 \}\,.
\end{equation}
Note that for any $\ell \in \Z^\nu$, the operator $\widehat{\cal R}(\ell) \in {\cal B}(L^2_0(\T^d))$ has the matrix representation 
\begin{equation}\label{rappresentazione widehat cal R}
\widehat{\cal R}(\ell) = \big( \widehat{\cal R}_j^{j'}(\ell) \big)_{j , j' \in \Z^d \setminus \{ 0 \}}\,. 
\end{equation}
Furthermore, by \eqref{notazione a blocchi}, for any $\vphi \in \T^\nu$, the operator ${\cal R}(\vphi)$ has the block representation $([{\cal R}(\vphi)]_\alpha^\beta)_{\alpha, \beta \in \sigma_0(\sqrt{- \Delta})}$ and for any $\ell \in \Z^\nu$, $\widehat{\cal R}(\ell)$ has the block representation $([\widehat{\cal R}(\ell)]_\alpha^\beta)_{\alpha, \beta \in \sigma_0(\sqrt{- \Delta})}$. For any $\alpha, \beta \in \sigma_0(\sqrt{- \Delta})$, we have the Fourier expansion $[{\cal R}(\vphi)]_\alpha^\beta = \sum_{\ell \in \Z^\nu}[\widehat{\cal R}(\ell)]_\alpha^\beta e^{\ii \ell \cdot \vphi}$ with
\begin{equation}\label{notazione a blocchi vphi}
[ \widehat{\cal R}(\ell)]_\alpha^\beta := \frac{1}{(2 \pi)^\nu} \int_{\T^\nu} [{\cal R}(\vphi)]_\alpha^\beta e^{- \ii \ell \cdot \vphi}\, d \vphi =  \Big( \widehat{\cal R}_j^{j'}(\ell) \Big)_{|j| = \alpha\,,\, |j'| = \beta} \qquad \forall \ell \in \Z^\nu\,,
\end{equation}
recall \eqref{coefficienti spazio tempo operatore}. 

\noindent
Let ${\cal R} : \T^\nu \to {\cal B}(L^2_0(\T^d))$ be differentiable and let $\omega \in \R^\nu$. For any $\vphi \in \T^\nu$, the operator $\omega \cdot \partial_\vphi {\cal R}(\vphi)$ is represented by the matrix $(\omega \cdot \partial_\vphi {\cal R}_j^{j'}(\vphi))_{j, j' \in \Z^d \setminus \{ 0 \}}$ and its block representation is given by $(\omega \cdot \partial_\vphi [{\cal R}(\vphi)]_\alpha^\beta)_{\alpha, \beta \in \sigma_0(\sqrt{- \Delta})}$. We also note that for any $\ell \in \Z^\nu$, the operator $\widehat{\omega \cdot \partial_\vphi {\cal R}}(\ell)$ admits the block representation $(\ii \omega \cdot \ell [\widehat{\cal R}(\ell)]_\alpha^\beta)_{\alpha, \beta \in \sigma_0(\sqrt{- \Delta})}$.

\noindent
Given ${\cal R} : \T^\nu \to {\cal B}(L^2_0(\T^d))$, recalling the notation \eqref{notazione operatore diagonale a blocchi}, we define the block-diagonal operator ${\cal R}_{diag}$ as  
\begin{equation}\label{media diagonale operatore cal R}
{\cal R}_{diag} := {\rm diag}_{\alpha \in \sigma_0(\sqrt{- \Delta})} [\widehat{\cal R}(0)]_\alpha^\alpha
\end{equation}
and for any $ N \in \N $, we define the smoothing operator $\Pi_N {\cal R}$ by 
\be\label{SM-block-matrix}
[\widehat{\Pi_N {\cal R}}(\ell)]_\alpha^{\beta}:= 
\begin{cases}
[\widehat{\cal R}(\ell)]_\alpha^{\beta} \qquad \, {\rm if} \qquad  {\rm max}\{ |\ell|, \alpha, \beta \} \leq N  \\
0 \quad \qquad \qquad {\rm otherwise}. 
\end{cases}
\ee
It is straightforward to verify that 
\begin{equation}\label{commutazione media proiettore}
(\Pi_N {\cal R})_{diag} = \Pi_N {\cal R}_{diag}\,. 
\end{equation}

\subsection{Block-decay norm for linear operators}\label{sezione norma decadimento a blocchi}
Given a smooth $\vphi$-dependent family ${\cal R} : \T^\nu \to {\cal B}(L^2_0(\T^d))$, $\vphi \mapsto {\cal R}(\vphi)$ as in \eqref{matrice operatore Toplitz vphi}, we define the {\it block-decay} norm 
\begin{equation}\label{decadimento Kirchoff}
| {\cal R} |_s := {\rm sup}_{\alpha, \beta \in \sigma_0(\sqrt{- \Delta})} \Big(\sum_{\ell \in \Z^\nu} \langle \ell, \alpha, \beta \rangle^{2 s} \| [\widehat{\cal R}(\ell)]_\alpha^\beta \|_{HS}^2 \Big)^{1/2}\,, \quad \langle \ell, \alpha, \beta \rangle := {\rm max}\{ 1, |\ell|, \alpha, \beta \}\,. 
\end{equation}
For families of operators of the form $ {\cal R}(\omega) : \vphi \mapsto {\cal R}(\vphi; \omega)$, $\omega \in \Omega_o \subset \R^\nu$, we define the norm 
\begin{equation}\label{norma decadimento lipschitz}
|{\cal R}|_s^{\Lipg} := |{\cal R}|_s^{\sup} + \gamma |{\cal R}|_s^{\lip}\,,
\end{equation}
$$
|{\cal R}|_s^{\sup} := \sup_{\omega \in \Omega_0} |{\cal R}(\omega)|_s\,, \quad |{\cal R}|_s^{\lip} := \sup_{\begin{subarray}{c}
\omega_1, \omega_2 \in \Omega_o \\
\omega_1 \neq \omega_2
\end{subarray}} \frac{|{\cal R}(\omega_1) - {\cal R}(\omega_2)|_s}{|\omega_1 - \omega_2|}\,.
$$
Moreover, if ${\cal R} : \T^\nu \to {\cal B}({\bf L}^2_0(\T^d))$, i.e. ${\cal R}$ has the form
\begin{equation}\label{operatore matriciale decadimento}
{\cal R}(\vphi) = \begin{pmatrix}
{\cal R}_1(\vphi) & {\cal R}_2(\vphi) \\
\overline{{\cal R}_2(\vphi)} & \overline{{\cal R}_1(\vphi)}
\end{pmatrix}\,, 
\end{equation}
we define 
\begin{equation}\label{norma decadimento operatore matriciale}
|{\cal R}|_s := |{\cal R}_1|_s + |{\cal R}_2|_s\,, \qquad |{\cal R}|_s^\Lipg := |{\cal R}_1|_s^\Lipg + |{\cal R}_2|_s^\Lipg\,.
\end{equation}
 
\noindent
In the following, we state some properties of this norm. We prove such properties for families of operators ${\cal R} : \T^\nu \to {\cal B}(L^2_0(\T^d))$. If ${\mathcal R}$ is an operator of the form \eqref{operatore matriciale decadimento} then the same statements hold with the obvious modifications. 
\begin{lemma}\label{elementarissimo decay}
$(i)$ The norm $| \cdot |_s$ is increasing, namely $|{\mathcal R}|_s \leq |{\mathcal R}|_{s'}$, for $s \leq s'$.
%

\noindent
$(ii)$ The operator ${\cal R}_{diag}$ defined by \eqref{media diagonale operatore cal R}, satisfies $|{\cal R}_{diag}|_s \leq |{\mathcal R}|_s$, implying that $\|[{\cal R}]_\alpha^\alpha \|_{HS} \leq \alpha^{- s} |{\cal R}|_s$ for any $\alpha \in \sigma_0(\sqrt{- \Delta})$

\noindent
$(iii)$ Items $(i),(ii)$ hold, replacing $|\cdot |_s$ by $|\cdot|_s^\Lipg$. 
\end{lemma}
\begin{proof}
The proof is elementary. It follows directly by the definitions \eqref{decadimento Kirchoff}, \eqref{norma decadimento lipschitz}, hence we omit it. 
\end{proof}

\begin{lemma}\label{interpolazione decadimento Kirchoff}
Let ${\cal R}$, ${\cal T}$ be operators of the form \eqref{operatore matriciale decadimento}. Then for any $s \geq s_0$ (recall \eqref{definition s0})
$$
|{\cal R} {\cal B} |_s \lesssim_s  |{\cal R}|_s |{\cal B}|_{2 s_0} + |{\cal R}|_{2 s_0} |{\cal B}|_s \,.
$$
If ${\cal R} = {\cal R}(\omega)$, ${\cal T}= {\cal T}(\omega)$ are Lipschitz with respect to the parameter $\omega \in \Omega_o \subseteq \Omega$, then the same estimate holds replacing $|\cdot |_s$ by $| \cdot |_s^\Lipg$. 
\end{lemma}

\begin{proof}
According to the notations \eqref{notazione a blocchi}, \eqref{definizione blocco operatore}, for any $\vphi \in \T^\nu$, the operator ${\cal R}(\vphi){\cal B}(\vphi)$ has the block representation 
$$
 {\cal R}(\vphi) {\cal T}(\vphi) = \Big(  [{\cal R}(\vphi) {\cal T}(\vphi)]_\alpha^\beta\Big)_{\begin{subarray}{c}
 \alpha, \beta \in \sigma_0(\sqrt{- \Delta}) \\
 \ell \in \Z^\nu
 \end{subarray}}\,, \quad   [{\cal R}(\vphi) {\cal T}(\vphi)]_\alpha^\beta = \sum_{\alpha_1 \in \sigma_0(\sqrt{- \Delta})} [{\cal R}(\vphi)]_\alpha^{\alpha_1} [{\cal T}(\vphi)]_{\alpha_1}^\beta
$$
and for all $\ell \in \Z^\nu$ 
$$
[\widehat{{\cal R}{\cal T}}(\ell)]_\alpha^\beta = \sum_{\alpha_1 \in \sigma_0(\sqrt{- \Delta}), \ell' \in \Z^\nu} [\widehat{\cal R}(\ell - \ell') ]_\alpha^{\alpha_1} [\widehat{\cal T}(\ell')]_{\alpha_1}^\beta\,.
$$
Then, using Lemma \ref{proprieta facili norma L2 matrici}-$(ii)$, we get that for any $\alpha, \beta \in \sigma_0(\sqrt{- \Delta})$
\begin{align}
\sum_{ \ell \in \Z^\nu} \langle \ell, \alpha,  \beta \rangle^{2 s} \| [\widehat{{\cal R}{\cal T}}(\ell)]_\alpha^\beta \|_{HS}^2 & \leq \sum_{\ell \in \Z^\nu} \Big( \sum_{\begin{subarray}{c}
\ell' \in \Z^\nu \\
\alpha_1 \in \sigma_0(\sqrt{- \Delta})
\end{subarray}} \langle \ell, \alpha, \beta \rangle^{s} \|  [\widehat{\cal R}(\ell - \ell') ]_\alpha^{\alpha_1}  \|_{HS} \| [\widehat{\cal T}(\ell')]_{\alpha_1}^\beta \|_{HS}\Big)^2\,. \label{paris 0} 
\end{align}
Using that for any $\alpha, \beta, \alpha_1 \in \sigma_0(\sqrt{- \Delta})$, $\ell, \ell' \in \Z^\nu$, $
\langle \ell, \alpha, \beta \rangle^s \lesssim_s \langle \ell - \ell', \alpha, \alpha_1 \rangle^s + \langle  \ell',  \alpha_1, \beta \rangle^s\,,$
we get 
\begin{align}
\eqref{paris 0} & \lesssim_s (I) + (II) \label{monkey kong}
\end{align}
where 
\begin{align}
(I) & := \sum_{\ell \in \Z^\nu} \Big( \sum_{\begin{subarray}{c}
 \ell' \in \Z^\nu \\
 \alpha_1 \in \sigma_0(\sqrt{- \Delta})
\end{subarray}} \langle \ell - \ell',  \alpha, \alpha_1 \rangle^s \|[\widehat{\cal R}(\ell -\ell')]_\alpha^{\alpha_1} \|_{HS} \| [\widehat{\cal T}(\ell')]_{\alpha_1}^\beta \|_{HS} \Big)^2 \label{arma letale 0} \\
(II) & := \sum_{\ell \in \Z^\nu} \Big( \sum_{\begin{subarray}{c}
 \ell' \in \Z^\nu \\
 \alpha_1 \in \sigma_0(\sqrt{- \Delta})
\end{subarray}} \langle  \ell', \alpha_1, \beta  \rangle^s \| [\widehat{\cal R}(\ell -\ell')]_\alpha^{\alpha_1} \|_{HS} \| [\widehat{\cal T}(\ell')]_{\alpha_1}^\beta \|_{HS} \Big)^2\,. \label{arma letale - 10}
\end{align}
Using that, by Lemma \ref{lemma spettro laplaciano}-$(i)$,  
$\sum_{\begin{subarray}{c}
\ell' \in \Z^\nu \\
\alpha_1 \in \sigma_0(\sqrt{- \Delta})
\end{subarray}} \langle \ell', \alpha_1\rangle^{- 2 s_0} \,,\, \sum_{\alpha_1 \in \sigma_0(\sqrt{- \Delta})} \alpha_1^{- 2 s_0} < + \infty
$ (recall that $s_0 > (\nu + d)/2$), applying the Cauchy Schwartz inequality, one gets 
\begin{align}
(I) & \lesssim \sum_{\ell \in \Z^\nu} \sum_{\begin{subarray}{c}
 \ell' \in \Z^\nu \\
 \alpha_1 \in \sigma_0(\sqrt{- \Delta})
\end{subarray}} \langle \ell - \ell', \alpha, \alpha_1 \rangle^{2 s} \| [\widehat{\cal R}(\ell -\ell')]_\alpha^{\alpha_1}\|_{HS}^2 \langle  \ell', \alpha_1  \rangle^{2 s_0} \| [\widehat{\cal T}(\ell')]_{\alpha_1}^\beta \|_{HS}^2  \nonumber\\
& \lesssim_s  \sum_{\begin{subarray}{c}
 \ell' \in \Z^\nu \\
 \alpha_1 \in \sigma_0(\sqrt{- \Delta})
\end{subarray}} \langle \ell', \alpha_1   \rangle^{2 s_0} \| [\widehat{\cal T}(\ell')]_{\alpha_1}^\beta \|_{HS}^2 \sum_{ \ell \in \Z^\nu} \langle\ell - \ell',  \alpha, \alpha_1 \rangle^{2 s} \| [\widehat{\cal R}(\ell -\ell')]_\alpha^{\alpha_1} \|_{HS}^2 \nonumber\\
& \lesssim_s  \sum_{\begin{subarray}{c}
 \ell' \in \Z^\nu \\
 \alpha_1 \in \sigma_0(\sqrt{- \Delta})
\end{subarray}} \frac{1}{\alpha_1^{2 s_0}}\langle   \ell',  \alpha_1  \rangle^{4 s_0} \| [\widehat{\cal T}(\ell')]_{\alpha_1}^\beta \|_{HS}^2 \sum_{ \ell \in \Z^\nu} \langle \ell - \ell', \alpha, \alpha_1 \rangle^{2 s} \| [\widehat{\cal R}(\ell -\ell')]_\alpha^{\alpha_1} \|_{HS}^2 \nonumber\\
& \lesssim_s \sum_{\alpha_1 \in \sigma_0(\sqrt{- \Delta})} \alpha_1^{- 2 s_0} \Big( \sup_{\alpha_1 \in \sigma_0(\sqrt{- \Delta})}  \sum_{\ell' \in \Z^\nu} \langle \ell', \alpha_1   \rangle^{4 s_0} \| [\widehat{\cal T}(\ell')]_{\alpha_1}^\beta \|_{HS}^2 \Big) \Big( \sup_{\alpha, \alpha_1 \in \sigma_0(\sqrt{- \Delta})} \sum_{k \in \Z^\nu} \langle k, \alpha, \alpha_1 \rangle^{2 s} \| [\widehat{\cal R}(k)]_{\alpha}^{\alpha_1} \|_{HS}^2 \Big) \nonumber\\
& \stackrel{\eqref{decadimento Kirchoff}}{\lesssim_s} | {\cal B}|_{2 s_0}^2 |{\cal R}|_s^2\,.
\end{align}
Similarly one proves that $(II) \lesssim_s  | {\cal T}|_{s}^2 |{\cal R}|_{2 s_0}^2$ and then, recalling \eqref{paris 0}, \eqref{monkey kong} one proves $|{\cal R} {\cal T}|_s \lesssim_s  | {\cal T}|_{2 s_0} |{\cal R}|_s + | {\cal T}|_{s} |{\cal R}|_{2 s_0}$. The estimate for the norm $| \cdot |_s^\Lipg$ follows easily by the previous one, by applying the triangular inequality.   
\end{proof}

For all $n \geq 1$, iterating the estimate of Lemma \ref{interpolazione decadimento Kirchoff} we get  
\be\label{Mnab}
| {\cal R}^n |_{2 s_0} \leq [C(s_0)]^{n-1} | {\cal R} |_{2 s_0}^n \qquad
\text{and}  \qquad   
| {\cal R}^n |_{s} \leq n C(s)^n  |{\cal R}|_{2 s_0}^{n-1} | {\cal R} |_{s} \, , \ \forall s \geq 2 s_0 \, ,
\ee
and the same bounds also hold for the norm $| \cdot |_s^\Lipg$ if ${\cal R}$ is Lipschitz continuous with respect to the parameter $\omega$. 
\begin{lemma}\label{lem:inverti}
Let $ \Phi ={\rm exp}(\Psi) $ with $\Psi := \Psi(\omega)$, depending in a Lipschitz way on the parameter $\omega \in \Omega_o \subset \R $, 
such that  $ |  \Psi  |_{2 s_0}^\Lipg \leq 1 $, $| \Psi|_s^\Lipg < + \infty$, with $ s \geq 2 s_0$. Then 
\be\label{PhINV}
| \Phi^{\pm 1} - {\rm Id} |_s \lesssim_s |  \Psi  |_s \, , \quad 
| \Phi^{\pm 1} - {\rm Id} |_{s}^\Lipg \lesssim_s  |  \Psi  |_{s}^\Lipg \, . 
\ee
\end{lemma}

\begin{proof} 
The claimed estimates can be proved by using the Taylor expansion of $\Phi^{\pm 1} - {\rm Id} = {\rm exp}(\pm \Psi) - {\rm Id} $, using the condition $|  \Psi  |_{2 s_0}^\Lipg \leq 1$ and by applying the estimates \eqref{Mnab}. 
\end{proof}
\begin{lemma}\label{lemma smoothing decay}
The operator $ \Pi_N^\bot {\cal R} := {\cal R} - \Pi_N {\cal R}$ (recall \eqref{SM-block-matrix}) satisfies 
\be\label{smoothingN}
| \Pi_N^\bot {\cal R} |_{s} \leq N^{- \mathtt b} |  {\cal R} |_{s+\mathtt b} \, , \quad 
| \Pi_N^\bot {\cal R} |_{s}^\Lipg \leq N^{- \mathtt b} |  {\cal R} |_{s+\mathtt b}^\Lipg \, , 
\quad \mathtt b \geq 0,
\ee
where in the second inequality ${\cal R}$ is Lipschitz with respect to the parameter $\omega \in \Omega_o \subseteq \Omega$. 
\end{lemma}
\begin{proof}
We have that for all $\mathtt b \in \N$, $\alpha, \beta \in \sigma_0(\sqrt{- \Delta})$ 
\begin{align*}
\sum_{\begin{subarray}{c}
 \ell \in \Z^\nu 
\end{subarray}} \langle \ell, \alpha, \beta \rangle^{2 s} \| [\widehat{\Pi_N^\bot {\cal R}}(\ell)]_\alpha^\beta \|_{HS}^2  & \stackrel{\eqref{SM-block-matrix}}{=}  \sum_{\big\{\ell : \langle \ell, \alpha, \beta \rangle > N \big\}} \langle  \ell,  \alpha, \beta \rangle^{2 s} \|  [\widehat{\cal R}(\ell)]_\alpha^\beta \|_{HS}^2 \nonumber\\
& \leq N^{- 2 \mathtt b} \sum_{\ell\in \Z^\nu} \langle  \ell, \alpha, \beta \rangle^{2(s + \mathtt b)} \| [\widehat{\cal R}(\ell)]_\alpha^\beta \|_{HS}^2 \stackrel{\eqref{decadimento Kirchoff}}{\leq} N^{- 2 \mathtt b} |{\cal R}|_{s + \mathtt b}^2\,,
\end{align*}
and the lemma follows.
\end{proof}


\begin{lemma}\label{decadimento operatori di proiezione}
Let us define the operator 
\begin{equation}\label{operatore forma buona resto}
{\cal R}(\vphi) [h ]:= q(\vphi, x) \int_{\T^d } g(\vphi, y) h (y)\, dy \,, \quad h \in L^2_0(\T^d) \quad q\,, g \in H^s_0(\T^{\nu + d})\,, \quad s \geq s_0\,.
\end{equation}
Then 
$$
|{\cal R}|_s \lesssim_s \| g \|_{s_0} \| q \|_{s} + \| g \|_{s + s_0} \| q \|_{0}\,.
$$
Moreover if the functions $g$ and $q$ are Lipschitz with respect to the parameter $\omega \in \Omega_o \subseteq \Omega$, then the same estimate holds replacing $| \cdot |_s$ by $|\cdot |_s^\Lipg$ and $\| \cdot \|_s$ by $\| \cdot \|_s^\Lipg$. 
\end{lemma}
\begin{proof}
A direct calculation shows that for all $\ell \in \Z^\nu$ and for all $j, j' \in \Z^d \setminus \{ 0 \}$
$$
\widehat{\cal R}_j^{j'}(\ell) = \sum_{\ell' \in \Z^\nu} \widehat q_{ j}(\ell - \ell') \widehat g_{ - j'}(\ell')\,.
$$
Using definition \eqref{norma L2 blocco}, the Cauchy Schwartz inequality (using that $\sum_{\ell' \in \Z^\nu} \langle \ell' \rangle^{- 2 s_0} < +\infty$) we get  
\begin{align}
\| [\widehat{\cal R}(\ell)]_\alpha^\beta\|_{HS}^2 & = \sum_{\begin{subarray}{c}
|j| = \alpha \\
|j'| = \beta
\end{subarray}} |\widehat{\cal R}_j^{j'}(\ell)|^2 \leq \sum_{\begin{subarray}{c}
|j| = \alpha \\
|j'| = \beta
\end{subarray}} \Big( \sum_{\ell'} |\widehat q_{ j}(\ell - \ell')| |\widehat g_{ - j'}(\ell')| \Big)^2  \leq \sum_{|j| = \alpha} \sum_{|j'| = \beta}\sum_{\ell'} |\widehat q_{ j}(\ell - \ell')|^2  \langle \ell' \rangle^{2 s_0} |\widehat g_{ - j'}(\ell')|^2 \nonumber\\
& \stackrel{\eqref{raggruppamento modi Fourier}, \eqref{bf h ell alpha}}{=} \sum_{\ell'} \| \widehat {\mathfrak q}_{\alpha}(\ell - \ell')\|_{L^2}^2 \langle \ell' \rangle^{2 s_0} \| \widehat{\mathfrak g}_{\beta}(\ell') \|_{L^2}^2\,. \label{norma L2 blocco proiettore}
\end{align}
Now for all $\alpha, \beta \in \sigma_0(\sqrt{- \Delta})$, 
\begin{align}
\sum_{ \ell \in \Z^\nu} \langle \ell, \alpha, \beta \rangle^{2 s} \| [\widehat{\cal R}(\ell)]_\alpha^\beta \|_{HS}^2 & \stackrel{\eqref{norma L2 blocco proiettore}}{\leq} \sum_{ \ell, \ell' \in \Z^\nu}  \langle \ell, \alpha, \beta \rangle^{2 s} \| \widehat{\mathfrak q}_{ \alpha}(\ell - \ell')\|_{L^2}^2 \langle \ell' \rangle^{2 s_0} \| \widehat{\mathfrak g}_{ \beta}(\ell') \|_{L^2}^2\,. \label{londra 0}
\end{align}
Using that $\langle  \ell, \alpha, \beta  \rangle^{2 s} \lesssim_s  \langle \ell - \ell', \alpha \rangle^{2 s} + \langle \ell', \beta \rangle^{2 s}$
we get 
\begin{align}
\eqref{londra 0} & \lesssim_s \sum_{ \ell} \sum_{\ell'} \langle \ell - \ell', \alpha \rangle^{2 s} \| \widehat{\mathfrak q}_{ \alpha}(\ell - \ell')\|_{L^2}^2 \langle \ell' \rangle^{2 s_0} \| \widehat{\mathfrak  g}_{ \beta}(\ell') \|_{L^2}^2 \nonumber\\
& \quad + \sum_{ \ell} \sum_{\ell'} \langle \ell', \beta \rangle^{2 s} \| \widehat{\mathfrak q}_{ \alpha}(\ell - \ell')\|_{L^2}^2 \langle \ell' \rangle^{2 s_0} \| \widehat{\mathfrak g}_{ \beta}(\ell') \|_{L^2}^2 \nonumber\\
& \lesssim_s \sum_{\ell'} \langle \ell' \rangle^{2 s_0} \| \widehat{\mathfrak g}_{ \beta}(\ell') \|_{L^2}^2 \sum_{\ell} \langle \ell - \ell', \alpha \rangle^{2 s} \| \widehat{\mathfrak q}_{ \alpha}(\ell - \ell')\|_{L^2}^2 \nonumber\\
& \quad + \sum_{\ell'} \langle \ell', \beta \rangle^{2 (s + s_0)} \| \widehat{\mathfrak g}_{ \beta}(\ell') \|_{L^2}^2 \sum_{\ell} \| \widehat{\mathfrak q}_{ \alpha}(\ell - \ell')\|_{L^2}^2 \nonumber\\
& \stackrel{\eqref{altro modo norma s vphi x}}{\lesssim_s} \| g \|_{s_0}^2 \| q \|_s^2 + \| g \|_{s + s_0}^2 \| q \|_{L^2}^2\,
\end{align}
and hence the lemma follows. 
\end{proof}
For a $\vphi$-independent linear operator ${\cal R} \in {\cal B}(L^2_0(\T^d))$ having the block-matrix representation \eqref{notazione a blocchi}, the block-decay norm \eqref{decadimento Kirchoff} becomes 
\begin{equation}\label{decadimento Kirchoff x}
|{\cal R}|_{s} = \sup_{\alpha, \beta \in \sigma_0(\sqrt{- \Delta})}   \langle \alpha, \beta \rangle^{ s} \| [{\cal R}]_\alpha^\beta \|_{HS}\,, \quad \langle \alpha, \beta \rangle := {\rm max}\{  \alpha, \beta \}\,. 
\end{equation}
The following Lemma holds: 
\begin{lemma}\label{lemma decadimento Kirchoff in x}

\noindent
$(i)$ Let ${\cal R}\in {\cal B}(L^2_0(\T^d))$ satsfy $|{\cal R}|_{s + 2 s_0} < + \infty$, for $s \geq 0$. Then ${\cal R} \in {\cal B}(L^2_0(\T^d), H^s_0(\T^d))$ and $\| {\cal R}\|_{{\cal B}(L^2_0, H^s_0)} \lesssim |{\cal R}|_{s + 2 s_0}$.  As a consequence ${\cal R} \in {\cal B}(H^s_0)$, with $\| {\cal R} \|_{{\cal B}(H^s_0)} \leq \| {\cal R}\|_{{\cal B}(L^2_0, H^s_0)} \lesssim |{\cal R}|_{s + 2 s_0}$.

\noindent
$(ii)$ Let $k \in \N$ and ${\cal R} : \T^\nu \to {\cal B}(L^2_0(\T^d))$ with $|{\cal R}|_{s + k + 2 s_0} < +\infty$. Then ${\cal R} \in W^{k, \infty}\Big(\T^\nu, {\cal B}(L^2_0, H^s_0) \Big)$ and for any $a \in \N^\nu$, $|a| \leq k$, one has 
$$
\| \partial_\vphi^a {\cal R} \|_{L^\infty(\T^\nu, {\cal B}(H^s_0))} \lesssim \sup_{\vphi \in \T^\nu} |\partial_\vphi^\alpha{\cal R}(\vphi)|_{s + 2s_0} \lesssim |{\cal R}|_{s + |a| + 2 s_0}\,.
$$
\end{lemma}
\begin{proof}
{\sc Proof of $(i)$.}
Let $u \in L^2_0(\T^d)$. 
By \eqref{definition block diagonal operators}, \eqref{altro modo norma s}, one has that 
\begin{align}
\|{\cal R}[u] \|_{H^s_x}^2 & = \sum_{\alpha \in \sigma_0(\sqrt{- \Delta})} \alpha^{2 s} \Big\|\sum_{\beta \in \sigma_0(\sqrt{- \Delta})} [{\cal R}]_\alpha^\beta[{\mathfrak u}_\beta] \Big\|_{L^2}^2 \lesssim  \sum_{\alpha \in \sigma_0(\sqrt{- \Delta})}  \Big( \sum_{\beta \in \sigma_0(\sqrt{- \Delta})} \alpha^{s} \| [{\cal R}]_\alpha^\beta[{\mathfrak u}_\beta] \|_{L^2} \Big)^2\,.
\end{align}
Using Lemma \ref{proprieta facili norma L2 matrici}-$(i)$ and recalling \eqref{decadimento Kirchoff x}, one gets 
\begin{align}
\|{\cal R}[u] \|_{H^s_x}^2 & \lesssim \sum_{\alpha \in \sigma_0(\sqrt{- \Delta})}  \Big( \sum_{\beta \in \sigma_0(\sqrt{- \Delta})} \frac{\alpha^{s + s_0} \beta^{s_0}}{\alpha^{s_0} \beta^{s_0}} \| [{\cal R}]_\alpha^\beta \|_{HS} \|{\mathfrak u}_\beta \|_{L^2} \Big)^2   \nonumber\\
& \lesssim \sum_{\alpha \in \sigma_0(\sqrt{- \Delta})} \frac{1}{\alpha^{2 s_0}}  \Big( \sum_{\beta \in \sigma_0(\sqrt{- \Delta})} \frac{\langle \alpha, \beta \rangle^{s + 2 s_0}}{ \beta^{s_0}} \| [{\cal R}]_\alpha^\beta \|_{HS} \|{\mathfrak u}_\beta \|_{L^2} \Big)^2 \nonumber\\
& \lesssim |{\cal R}|_{s + 2 s_0}^2 \sum_{\alpha \in \sigma_0(\sqrt{- \Delta})} \frac{1}{\alpha^{2 s_0}}  \Big( \sum_{\beta \in \sigma_0(\sqrt{- \Delta})} \frac{1}{ \beta^{ s_0}} \|{\mathfrak u}_\beta \|_{L^2} \Big)^2\,. \label{benzene 0}
\end{align}
By the Cauchy-Schwartz inequality 
\begin{align}
\eqref{benzene 0} & \lesssim |{\cal R}|_{s + 2 s_0}^2 \sum_{\alpha, \beta \in \sigma_0(\sqrt{- \Delta})} \frac{1}{\alpha^{2 s_0}\beta^{2 s_0}} \sum_{\beta \in \sigma_0(\sqrt{- \Delta})} \| {\mathfrak u}_\beta \|_{L^2}^2 \stackrel{\eqref{altro modo norma s}}{\lesssim} |{\cal R}|_{s + 2 s_0}^2 \| u \|_{L^2}  \, \label{benzene 1}
\end{align}
by applying Lemma \ref{lemma spettro laplaciano}-$(i)$ (note that $2 s_0 = 2 ([(\nu + d)/ 2] + 1) > \nu + d$) and then the claim follows. 

\bigskip

\noindent
{\sc Proof of $(ii)$.}For any $\alpha, \beta \in \sigma_0(\sqrt{- \Delta})$ and for any multi-index $a \in \N^\nu$, $|a| \leq k$ one has that the operator $\partial_\vphi^a {\cal R}(\vphi)$ admits the block-matrix representation
$$
\partial_\vphi^a{\cal R}(\vphi) = \Big( \partial_\vphi^a [{\cal R}(\vphi)]_\alpha^\beta\Big)_{\alpha, \beta \in \sigma_0(\sqrt{- \Delta})}\,.
$$
Expanding in Fourier series $\partial_\vphi^a [{\cal R}(\vphi)]_\alpha^\beta$, one has 
$$
\partial_\vphi^a [{\cal R}(\vphi)]_\alpha^\beta = \sum_{\ell \in \Z^\nu} \ii^{|a|} \ell^a [\widehat{\cal R}(\ell)]_\alpha^\beta e^{\ii \ell \cdot \vphi},
$$
and by the Cauchy-Schwartz inequality
\begin{align}
\| \partial_\vphi^a [{\cal R}(\vphi)]_\alpha^\beta \|_{HS} & \leq \sum_{\ell \in \Z^\nu} |\ell|^{|a|}  \|[\widehat{\cal R}(\ell)]_\alpha^\beta \|_{HS} \lesssim \Big(\sum_{\ell \in \Z^\nu} \langle \ell \rangle^{2 (|a| + s_0)} \|[\widehat{\cal R}(\ell)]_\alpha^\beta \|_{HS}^2 \Big)^{\frac12}\,. \label{o mar e bell0}
\end{align}
Thus by \eqref{o mar e bell0}, for any $\alpha, \beta \in \sigma_0(\sqrt{- \Delta})$, for any $\vphi \in \T^\nu$, one has
\begin{align*}
&  \langle \alpha,  \beta  \rangle^{2 s} \|  \partial_\vphi^a [{\cal R}(\vphi)]_\alpha^\beta\|_{HS}^2  \stackrel{\eqref{o mar e bell0}}{\lesssim} \sum_{\ell \in \Z^\nu } \langle \ell, \alpha,  \beta  \rangle^{2 (s + |a| + s_0)}  \|  [\widehat{\cal R}(\ell)]_\alpha^\beta \|_{HS}^2 \stackrel{\eqref{decadimento Kirchoff}}{\lesssim} |{\cal R}|_{s + |a| + s_0}^2
\end{align*}
and then the lemma follows by recalling \eqref{decadimento Kirchoff x} and by applying item $(i)$. 
\end{proof}
\subsection{A class of $\vphi$-dependent Fourier multipliers}\label{sezione pseudo diff}
For any $m \in \R$, we define the class $S^m$ of Fourier multipliers of order $m$ as  
$$
S^m := \big\{ r : \sigma_0(\sqrt{- \Delta}) \to \C : \sup_{\alpha \in \sigma_0(\sqrt{- \Delta})} |r(\alpha)|  \alpha^{- m} < + \infty \big\}
$$
where we recall that the set $\sigma_0(\sqrt{- \Delta})$ is defined in \eqref{definizione cal N}. 
To any symbol $r \in S^m$, we associate the linear operator ${\rm Op}( r )$ defined by
\begin{equation}
{\rm Op}( r ) u(x) := \sum_{j \in \Z^d \setminus \{ 0 \}} r(|j|) u_j e^{\ii j \cdot x}\,, \qquad \forall u \in H^m_0(\T^d)\,.
\end{equation}
We denote by $OPS^m$ the class of the operators associated to the symbols in $S^m$.

\noindent
In the following we deal with $\vphi$-dependent families of Fourier multipliers $ r : \T^\nu \times \sigma_0(\sqrt{- \Delta}) \to \C$, $r (\vphi, \cdot) \in S^m$. The action of the operator ${\rm Op}( r )={\rm Op}( r(\vphi, |j|) )$ on Sobolev functions $u \in H^s_0(\T^{\nu + d})$ is given by 
\begin{equation}\label{Fourier multiplier time dependent}
{\rm Op}( r ) u(\vphi, x) := \sum_{j \in \Z^d \setminus \{ 0 \}} r(\vphi, |j|) u_j(\vphi) e^{\ii j \cdot x} = \sum_{\begin{subarray}{c}
\ell , \ell' \in \Z^\nu \\
j \in \Z^d \setminus \{ 0 \}
\end{subarray}} \widehat r(\ell - \ell', |j|) \widehat u_j(\ell') e^{\ii (\ell \cdot \vphi + j \cdot x)}\,.
\end{equation}
Note that, using the representation \eqref{bf h ell alpha 0}, the action of the operator ${\rm Op} ( r )$ on a function $u(\vphi, x)$ can be written as
\begin{equation}\label{Fourier multiplier time dependent blocchi}
{\rm Op} ( r ) u(\vphi, x) = \sum_{\alpha \in \sigma_0(\sqrt{- \Delta})} r(\vphi, \alpha) {\mathfrak u}_\alpha(\vphi, x) = \sum_{\begin{subarray}{c}
\ell, \ell' \in \Z^\nu \\
\alpha \in \sigma_0(\sqrt{- \Delta})
\end{subarray}} \widehat r(\ell - \ell', \alpha) \widehat{\mathfrak u}_\alpha(\ell', x) e^{\ii \ell \cdot \vphi}\,.
\end{equation}
The following elementary properties hold:  
\begin{equation}\label{proprieta elementari simboli classe Sm}
\overline{{\rm Op}( r )} = {\rm Op}( \overline r ) = {\rm Op}( r )^*\,, \quad {\rm Op}( r )^T = {\rm Op}( r )
\end{equation}
(recall \eqref{definizione operatore coniugato}, \eqref{definizione operatore trasposto}, \eqref{operatore aggiunto cal R}). The above properties imply that
\begin{equation}\label{autoaggiuntezza pseudo diff}
{\rm Op} ( r ) = {\rm Op}( r )^* \qquad \text{if\,\,and\,\,only\,\,if} \qquad r(\vphi, \alpha) = \overline{r(\vphi, \alpha)}\,, \qquad \forall (\vphi, \alpha) \in \T^\nu \times \sigma_0(\sqrt{- \Delta})\,.
\end{equation}
 Let ${\cal R} = {\rm Op}( r ) \in OPS^m$, ${\cal B} = {\rm Op}( b ) \in OPS^{m'}$. Then the composition operator ${\cal R} \circ {\cal B}$ is given by 
\begin{equation}\label{operatore composizione fourier multiplier}
{\cal R} \circ {\cal B} = {\rm Op}( r ) \circ {\rm Op}( b ) = {\rm Op}(r b ) \in OPS^{m + m'}\,.
\end{equation}
Note that ${\cal R} \circ {\cal B} = {\cal B} \circ {\cal R}$. 

\noindent
For an operator ${\cal R} = {\rm Op}( r )\in OPS^m$, for any $s \geq 0$, $m \in \R$, we define the family of norms 
\begin{equation}\label{norma pseudo diff}
 \norma {\rm Op}( r ) \norma_{m, s}  := \sup_{ \alpha \in \sigma_0(\sqrt{- \Delta})} \| r(\cdot, \alpha) \|_s  \alpha^{- m}\,
\end{equation}
and if $ r = r(\vphi, \alpha; \omega)$, $\omega \in \Omega_o \subseteq \Omega$ is Lipschitz with respect to the parameter $\omega \in \Omega_o$ then we define 
\begin{equation}
\norma {\rm Op}( r ) \norma_{m, s}^\Lipg = \norma {\rm Op}( r ) \norma_{m, s}^{\rm sup} + \gamma \norma {\rm Op}( r ) \norma_{m, s}^{\rm lip}
\end{equation}
where 
$$
\norma {\rm Op}( r ) \norma_{m, s}^{\rm sup} := \sup_{\omega \in \Omega_o }\norma {\rm Op}( r )(\omega)\norma_{m, s}\,, \quad \norma{\rm Op}( r ) \norma_{m, s}^{\rm lip} := \sup_{\begin{subarray}{c}
\omega_1, \omega_2 \in \Omega_o \\
\omega_1 \neq \omega_2
\end{subarray}} \dfrac{ \norma {\rm Op}( r )(\omega_1) - {\rm Op}( r )(\omega_2)\norma_{m, s}}{|\omega_1 - \omega_2|}\,.
$$
We also deal with operators 
\begin{equation}\label{matrix fourier multiplier}
{\cal R} = \begin{pmatrix}
{\rm Op}( r_1 ) & {\rm Op}( r_2 )\\
{\rm Op}( \overline r_2) & {\rm Op}(\overline r_1 )
\end{pmatrix}\,, \qquad r_1, r_2 \in S^m\,.
\end{equation}
With a slight abuse of notations we still denote by $OPS^m$ the class of operators of the form \eqref{matrix fourier multiplier}. For such operators, we define the norms $\norma {\cal R } \norma_{m , s} : = \norma {\rm Op}( r_1 )\norma_{m, s} + \norma{\rm Op}( r_2 ) \norma_{m, s}$ and $\norma {\cal R} \norma_{m, s}^\Lipg := \norma {\rm Op}( r_1 ) \norma_{m, s}^\Lipg + \norma {\rm Op}( r_2 ) \norma_{m, s}^\Lipg$. In the following, we state some properties of the norm $\norma \cdot \norma_{m, s}$. We prove such properties for operators ${\cal R}(\vphi) = {\rm Op}(r (\vphi, \cdot))$. If ${\mathcal R}$ is an operator of the form \eqref{matrix fourier multiplier} then the same statements hold with the obvious modifications. 

\noindent
It is immediate to verify that
\begin{equation}\label{proprieta facili norma pseudo 0}
\norma \cdot \norma_{m, s} \leq \norma \cdot \norma_{m, s'}\,, \quad \forall s \leq s'\,, \quad \forall m \in \R\,,
\end{equation}
\begin{equation}\label{proprieta facili norma pseudo 1}
\norma \cdot \norma_{m, s} \leq \norma \cdot \norma_{m', s}\,, \quad \forall m \geq m' \,, \quad  \forall s \geq 0\,
\end{equation}
and the same inequality holds for the corresponding Lipschitz norms. 
\begin{lemma}\label{azione fourier multiplier}
Let ${\cal R} = {\rm Op}( r )$ with $\norma {\cal R} \norma_{0, s} < + \infty$, $s \geq s_0$. Then for any $u \in H^s_0(\T^{\nu + d})$  
$$
\| {\cal R} u \|_s \lesssim_s \norma {\cal R} \norma_{0, s} \| u\|_{s_0 } + \norma {\cal R} \norma_{0, s_0} \| u\|_s\,.
$$
The same statements hold, replacing $\| \cdot \|_s$ by $\| \cdot \|_s^\Lipg$ and $\norma \cdot \norma_{0, s}$ by $\norma \cdot \norma_{0, s}^\Lipg$. If ${\cal R}$ is an operator of the form \eqref{matrix fourier multiplier}, then a similar estimate holds.
\end{lemma}
\begin{proof}
The claimed estimate follows by the same arguments used to prove Lemma 2.13 in \cite{BertiMontalto}, hence the proof is omitted. Actually our case is even simpler since the symbol $r$ does not depend on the variable $x \in \T^d$. 
\end{proof}
\begin{lemma}\label{stima Hs x fourier multiplier}
Let ${\cal R} = {\rm Op}( r )$, with $\norma {\cal R}\norma_{0, s_0 + 1} < + \infty$. Then ${\cal R} \in {\cal C}^1(\T^\nu, {\cal B}(H^s_0))$ for any $s \geq 0$ and 
$\| {\cal R}\|_{{\cal C}^1(\T^\nu, {\cal B}(H^s_0))} \lesssim \norma {\cal R}\norma_{0, s_0 + 1}$. 
\end{lemma}
\begin{proof}
Let ${\cal R} = {\rm Op}( r ) \in OPS^0$. Since $\norma {\cal R}\norma_{0, s_0 + 1} < + \infty$, by the definition \eqref{norma pseudo diff}, the symbol $r(\cdot, \alpha)$ is in $H^{s_0 + 1}(\T^\nu)$ for any $\alpha \in \sigma_0(\sqrt{- \Delta})$. Hence, by the Sobolev embedding $r(\cdot, \alpha) \in {\cal C}^1(\T^\nu)$ with $\| r(\cdot, \alpha)\|_{{\cal C}^1(\T^\nu)} \lesssim \| r(\cdot, \alpha)\|_{s_0 + 1} \lesssim \norma {\cal R}\norma_{0, s_0 + 1}$ for any $\alpha \in \sigma_0(\sqrt{- \Delta})$. Since $\| {\cal R} \|_{{\cal C}^1(\T^\nu, {\cal B}(H^s_0))} \leq \sup_{\alpha \in \sigma_0(\sqrt{- \Delta})} \| r(\cdot, \alpha)\|_{{\cal C}^1(\T^\nu)}$ for any $s \geq 0$, the claimed statement follows.  
\end{proof}
\begin{lemma}\label{composizione fourier multiplier}
Let $m, m' \in \R$ and  ${\cal R} \in OPS^m$, ${\cal B} \in OPS^{m'}$ be two operators of the form \eqref{matrix fourier multiplier} with $\norma {\cal R} \norma_{m, s}\,,\, \norma {\cal B} \norma_{m', s} < \infty$, with $s \geq s_0$. Then the operator ${\cal R} {\cal B} \in OPS^{m + m'}$ has still the form \eqref{matrix fourier multiplier} and it satisfies the estimate 
$$
\norma {\cal R} {\cal B} \norma_{m + m', s} \lesssim_s \norma {\cal R} \norma_{m, s} \norma {\cal B} \norma_{m', s_0} + \norma {\cal R} \norma_{m, s_0} \norma {\cal B} \norma_{m', s}\,.
$$
The same estimate holds replacing the norm $\norma \cdot \norma_{m, s}$ by the norm $\norma \cdot \norma_{m, s}^\Lipg$, if ${\cal R}$ and ${\cal B}$ are Lipschitz with respect to the parameter $\omega \in \Omega_o$. 
\end{lemma}
\begin{proof}
 The claimed statement follows by using the property \eqref{operatore composizione fourier multiplier}, the definition \eqref{norma pseudo diff} and the interpolation Lemma \ref{interpolazione C1 gamma}. 
\end{proof}
Note that the above lemma implies that if ${\cal R} \in OPS^m$, then ${\cal R}^k \in OPS^{k m}$ for any $k \geq 1$ and 
\begin{equation}\label{composizione iterata norma pseudo diff}
\norma{\cal R}^k \norma_{k m, s_0} \leq C(s_0)^{k - 1} \norma{\cal R} \norma_{m, s_0}^k\,, \qquad \norma{\cal R}^k \norma_{km, s} \leq  k C(s)^k \norma {\cal R} \norma_{m, s_0}^{k - 1} \norma {\cal R} \norma_{m, s}\,, \quad s \geq s_0\,.
\end{equation}
The same estimate holds replacing $\norma \cdot \norma_{m, s}$ by $\norma \cdot \norma_{m, s}^\Lipg$.
\begin{lemma}\label{esponenziale norma pseudo diff}
Let $\Psi(\vphi) \in OPS^{- m}$, $\vphi \in \T^\nu$, $m \geq 0$, with 
\begin{equation}\label{piccolezza Psi m s0} 
\norma \Psi \norma_{- m, s_0} \leq 1.
\end{equation}
 Then the operator $\Phi(\vphi) : = {\rm exp}(\Psi(\vphi))$ satisfies $\Phi(\vphi) - {\rm Id} \in OPS^{- m}$, $\forall \vphi \in \T^\nu$, with
 \begin{equation}\label{stima norma pseudo diff Phi - I}
\norma \Phi - {\rm Id} \norma_{- m, s} \lesssim_s \norma \Psi \norma_{- m , s}\,.
\end{equation}
Moreover the operator 
\begin{equation}\label{definizione Phi geq 2}
\Phi_{\geq 2}(\vphi) := \sum_{k \geq 2} \frac{\Psi(\vphi)^k}{k !} \in OPS^{- 2m}\,, \quad \forall \vphi \in \T^\nu
\end{equation}
and it satisfies the estimate 
\begin{equation}\label{stima norma pseudo diff code di ordine 2}
\norma \Phi_{\geq 2} \norma_{- 2 m, s} \lesssim_s \norma \Psi \norma_{- m, s} \norma \Psi \norma_{- m, s_0}\,.
\end{equation}
If the operator $\Psi$ depends in a Lipschitz way on the parameter $\omega \in \Omega_o \subseteq \Omega$ and $\norma \Psi \norma_{- m, s_0}^\Lipg \leq 1$, then the estimates \eqref{stima norma pseudo diff Phi - I}, \eqref{stima norma pseudo diff code di ordine 2} hold replacing the norm $\norma \cdot \norma_{- m, s}$ by the norm $\norma \cdot \norma_{- m, s}^\Lipg$. 
\end{lemma}
\begin{proof}
The Lemma follows by using the Taylor expansion of the operator $\Phi - {\rm Id}$, the definition \eqref{definizione Phi geq 2}, the estimate \eqref{composizione iterata norma pseudo diff} and the condition \eqref{piccolezza Psi m s0}. 
\end{proof}
In the next lemma we compare the block-decay norm $|\cdot |_s$ defined in \eqref{decadimento Kirchoff} with the norm $\norma \cdot \norma_{m, s}$ defined in \eqref{norma pseudo diff}.
\begin{lemma}\label{lemma norma decadimento norma pseudo diff}
Let $s \geq 0$ and ${\cal R}(\vphi) \in OPS^{- s - \frac{d - 1}{2}}$, $\vphi \in \T^\nu$. Then 
$$
|{\cal R}|_s \lesssim \norma {\cal R} \norma_{- s - \frac{d - 1}{2}, s}\,.
$$ 
The same estimate holds replacing $|\cdot |_s$ by $|\cdot |_s^\Lipg$ and $\norma \cdot \norma_{- s - \frac{d - 1}{2}, s}$ by $\norma \cdot \norma_{- s - \frac{d - 1}{2}, s}^\Lipg$ if the operator ${\cal R}$ depends in a Lipschitz way on the parameter $\omega \in \Omega_o \subseteq \Omega$. 
\end{lemma}
\begin{proof}
Let ${\cal R} = {\rm Op}\big( r \big)$. By the representation \eqref{Fourier multiplier time dependent blocchi}, 
 for any $\vphi \in \T^\nu$, the operator ${\cal R}(\vphi)$ is block-diagonal (recall the definition \eqref{notazione operatore diagonale a blocchi}) and  it has the block representation 
$$
{\cal R}(\vphi) = {\rm diag}_{\alpha \in \sigma_0(\sqrt{- \Delta})}[{\cal R}(\vphi)]_\alpha^\alpha \,, \qquad [{\cal R}(\vphi)]_\alpha^\alpha = r(\vphi, \alpha) {\mathbb  I}_\alpha\,,\quad \forall \alpha \in \sigma_0(\sqrt{- \Delta})
$$ 
and for any $\ell \in \Z^\nu$
$$
[\widehat{\cal R}(\ell)]_\alpha^\alpha = \widehat r(\ell, \alpha) {\mathbb I}_\alpha\,, \quad \forall \alpha \in \sigma_0(\sqrt{- \Delta})\,, \quad \forall \ell \in \Z^\nu\,
$$
where we recall that ${\mathbb I}_\alpha : {\mathbb E}_\alpha \to {\mathbb E}_\alpha$ is the identity. Hence, using that $\| {\mathbb I}_\alpha \|_{HS} \lesssim \alpha^{\frac{d - 1}{2}}$ (see \eqref{norma L2 blocco}), recalling the definition \eqref{decadimento Kirchoff}, one gets 
\begin{align}
|{\cal R}|_s^2 & = \sup_{\alpha \in \sigma_0(\sqrt{- \Delta})} \sum_{\ell \in \Z^\nu} \langle \ell, \alpha \rangle^{2 s} \| [\widehat{\cal R}(\ell)]_\alpha^\alpha\|_{HS}^2   = \sup_{\alpha \in \sigma_0(\sqrt{- \Delta})} \sum_{\ell \in \Z^\nu} \langle \ell, \alpha \rangle^{2 s} |\widehat r(\ell, \alpha)|^2 \| {\mathbb I}_\alpha\|_{HS}^2 \nonumber\\
& \lesssim \sup_{\alpha \in \sigma_0(\sqrt{- \Delta})} \sum_{\ell \in \Z^\nu} \langle \ell, \alpha \rangle^{2 s} |\widehat r(\ell, \alpha)|^2 \alpha^{d - 1} \lesssim \sup_{\alpha \in \sigma_0(\sqrt{- \Delta})} \| r(\cdot, \alpha) \|_s^2 \alpha^{2 s + d - 1} \lesssim \norma {\cal R} \norma_{- s - \frac{d - 1}{2}, s}^2
\end{align}
which is the claimed estimate. 
\end{proof}
\subsection{Hamiltonian formalism}\label{formalismo Hamiltoniano} 
We define the symplectic form ${\mathcal W}$ as
\begin{equation}\label{forma simplettica reale}
{\mathcal W}[ z_1, z_2 ] := \langle z_1, J z_2 \rangle_{{ L}^2_x}\,, \quad J = \begin{pmatrix}
0 & 1 \\
- 1 & 0
\end{pmatrix}\,, \quad  \forall z_1, z_2 \in L^2_0(\T^d, \R) \times L^2_0(\T^d, \R)\,.
\end{equation}

\begin{definition}\label{campo Hamiltoniano reale}
A $\vphi$-dependent linear vector field $X(\vphi) : {L}^2_0(\T^d, \R) \times {L}^2_0(\T^d, \R)  \to {L}^2_0(\T^d, \R) \times {L}^2_0(\T^d, \R)$, $\vphi \in \T^\nu$, is Hamiltonian, if $X(\vphi) = J G(\vphi)$, where $J$ is given in \eqref{forma simplettica reale}
and the operator $G(\vphi)$ is symmetric for every $\vphi \in \T^\nu$. 
\end{definition}
\begin{definition}\label{trasformazione simplettica reale}
A $\vphi$-dependent map $\Phi(\vphi) : {L}^2_0(\T^d, \R) \times {L}^2_0(\T^d, \R) \to {L}^2_0(\T^d, \R) \times {L}^2_0(\T^d, \R)$, $\vphi \in \T^\nu$ is symplectic if for any $\vphi \in \T^\nu$, for any $z_1, z_2 \in {L}^2_0(\T^d, \R) \times {L}^2_0(\T^d, \R)$, 
$$
{\cal W}[\Phi(\vphi)[z_1]\,,\, \Phi(\vphi) [z_2]] = {\cal W}[ z_1,  z_2]\,,
$$
or equivalently $\Phi(\vphi)^T J \Phi(\vphi) = J$ for any $\vphi \in \T^\nu$. 
\end{definition}
Assume to have a differentiable map $\vphi \in \T^\nu \mapsto \Phi(\vphi ) \in {\cal B}\big(L^2_0(\T^d, \R) \times L^2_0(\T^d, \R) \big)$ and let us consider the quasi-periodically forced linear Hamiltonian PDE 
\begin{equation}\label{interpretazione dinamica 1}
\partial_t z = X(\omega t) z\,, \qquad X(\vphi) :=  J G(\vphi) \,, \quad \vphi \in \T^\nu\,, \quad z \in L^2_0(\T^d, \R) \times L^2_0(\T^d, \R)\,.
\end{equation}
Under the change of coordinates $ z = \Phi(\omega t) h $, the above PDE is transformed into the equation
 \begin{equation}\label{interpretazione dinamica 2}
 \partial_t h = X_+(\omega t) h\,, \quad  
  \end{equation}
  where $X_+(\omega t)$ is the transformed vector field under the action of the map $\Phi(\omega t)$ (push-forward), namely
  \begin{equation}\label{push forward}
  X_+(\vphi) = \Phi_{\omega*} X(\vphi) := \Phi(\vphi)^{- 1} X(\vphi) \Phi(\vphi) - \Phi(\vphi)^{- 1} \omega \cdot \partial_\vphi \Phi(\vphi), \quad \forall \vphi \in \T^\nu\,. 
  \end{equation}
It turns out that, since $X(\vphi)$ is a Hamiltonian vector field and $\Phi(\vphi)$ is symplectic, the transformed vector field $X_+(\vphi)$ is still Hamiltonian, namely it has the form given in Definition \eqref{campo Hamiltoniano reale}.  
  \subsubsection{Hamiltonian formalism in complex coordinates}\label{formalismo hamiltoniano complesso}
In this section we describe how the Hamiltonian structure described before, reads in the complex coordinates introduced in \eqref{prima volta variabili complesse 0}, \eqref{matrice coordinate complesse}.  
Let $J G(\vphi)$, $\vphi \in \T^\nu$ be a linear Hamiltonian vector field, with $G(\vphi) \in {\cal B}\Big(L^2_0(\T^d, \R) \times L^2_0(\T^d, \R) \Big)$ being a symmetric operator as in \eqref{operatore matriciale reale}. The conjugated vector field ${\cal R}(\vphi) := {\cal C}^{- 1} J G(\vphi) {\cal C} \in {\cal B}({\bf L}^2_0(\T^d))$ has the form 
\begin{equation}\label{operatore Hamiltoniano coordinate complesse}
{\cal R}(\vphi) = \ii \begin{pmatrix}
{ R}_1(\vphi) &  R_2(\vphi) \\
- \overline{R_2(\vphi)} & - \overline{R_1(\vphi)}
\end{pmatrix}\,,
\end{equation}
where 
\begin{equation}\label{R1 R2 operatore Hamiltoniano complesso}
R_1(\vphi) := - A(\vphi) - D(\vphi) + \ii B(\vphi) - \ii B(\vphi)^T\,, \quad R_2(\vphi) := - A(\vphi) + D(\vphi) - \ii B(\vphi) - \ii B(\vphi)^T\,
\end{equation}
(recall that the operator $\overline R$ is defined in \eqref{definizione operatore coniugato}). Note that the operators $R_1(\vphi)$, $R_2(\vphi)$ are linear operators acting on complex valued $L^2$ functions $L^2_0(\T^d)$. Furthermore, since $G(\vphi)$ is symmetric, i.e. $A(\vphi) = A(\vphi)^T$, $B(\vphi) = C(\vphi)^T$, $D(\vphi) = D(\vphi)^T$, it turns out that  
\begin{equation}\label{condizione R1 R2 campo hamiltoniano complesso}
R_1(\vphi) = R_1(\vphi)^*\,, \qquad R_2(\vphi) = R_2(\vphi)^T\,, \qquad \forall \vphi \in \T^\nu.
\end{equation}
We refer to an operator ${\cal R}$ of the form \eqref{operatore Hamiltoniano coordinate complesse}, with $R_1$ and $R_2$ satisfying \eqref{condizione R1 R2 campo hamiltoniano complesso}, as a Hamiltonian vector field in complex coordinates. The operator ${\cal R}(\vphi)$ in \eqref{operatore Hamiltoniano coordinate complesse} satisfies 
\begin{equation}\label{campo hamiltoniano complesso}
{\cal R}(\vphi)[{\bf u}] = \ii J \nabla_{\bf u} {\cal H}(\vphi, {\bf u})\,, \quad {\bf u} := (u , \bar u)\,, \quad \nabla_{\bf u} {\cal H} = (\nabla_u {\cal H}, \nabla_{\bar u} {\cal H})\,,
\end{equation}
where the real Hamiltonian ${\cal H}$ has the form 
\begin{equation}\label{generica hamiltoniana quadratica reale nel complesso esplicita}
{\cal H}(\vphi , {\bf u}) := \langle {\cal G}(\vphi)[{\bf u}]\,,\, {\bf u} \rangle \,, \quad {\cal G}(\vphi) := \begin{pmatrix}
\overline{R_2(\vphi)} & \overline{R_1(\vphi)} \\
R_1(\vphi) & R_2(\vphi)
\end{pmatrix} \,, 
\end{equation}
i.e.
\begin{equation}\label{generica hamiltoniana quadratica reale nel complesso}
{\cal H}(\vphi, u, \bar u) = \int_{\T^d} {R}_1(\vphi)[u] \bar u\, d x+ \frac12 \int_{\T^d} { R}_2(\vphi)[u]\,, u \, dx + \frac12 \int_{\T^d} \overline{{ R}_2(\vphi)}[\bar u]\,\bar u\,d x\,.
\end{equation}
and 
$$
\nabla_u {\cal H} = \frac{1}{\sqrt{2}}\big( \nabla_v {\cal H} - \ii \nabla_\psi {\cal H} \big)\,, \quad \nabla_{\overline u} {\cal H} = \frac{1}{\sqrt{2}}\big( \nabla_v {\cal H} + \ii \nabla_\psi {\cal H} \big)\,. 
$$
By \eqref{condizione R1 R2 campo hamiltoniano complesso} we deduce that 
$$
{\cal G}(\vphi) = {\cal G}(\vphi)^T\,, \qquad \forall\vphi \in \T^\nu\,.
$$
The symplectic form ${\cal W}$ defined in \eqref{forma simplettica reale} reads in the coordinates ${\bf u} =(u, \bar u)$ as. 

\begin{equation}\label{forma simplettica coordinate complesse}
{\Gamma}[{\bf u}_1, {\bf u}_2] =  \ii \int_{\T^d} (u_1 \bar u_2 -  \bar u_1 u_2)\, dx  = \ii \langle {\bf u}_1 \,,\, J {\bf u}_2\rangle_{{\bf L}^2_x}\,, \quad \forall {\bf u}_1, {\bf u}_2 \in {\bf L}^2_0(\T^d)\,
\end{equation}
where 
\begin{equation}\label{prodotto scalare u bar u}
\langle {\bf u}_1, {\bf u}_2 \rangle_{{\bf L}^2_x} := \int_{\T^d} u_1 u_2 + \overline u_1 \overline u_2\, d x\,, \qquad \forall {\bf u}_1 , {\bf u}_2 \in {\bf L}^2_0(\T^d)\,. 
\end{equation}
\begin{definition}\label{definizione mappa simplettica complessa}
A $\vphi$-dependent family of linear operators $\Phi(\vphi) : {\bf L}^2_0(\T^d) \to {\bf L}^2_0(\T^d)$, $\vphi \in \T^\nu$ is symplectic if 
$$
{ \Gamma}[\Phi(\vphi)[{\bf u}_1], \Phi(\vphi)[{\bf u}_2]] = { \Gamma}[{\bf u}_1, {\bf u}_2]\,, \quad \forall {\bf u}_1, {\bf u}_2 \in {\bf L}^2_0(\T^d)\,, \quad \forall \vphi \in \T^\nu\,.
$$
\end{definition}
It is well known that if ${\cal R}(\vphi)$ is an operator of the form \eqref{operatore Hamiltoniano coordinate complesse}, \eqref{condizione R1 R2 campo hamiltoniano complesso},namely by \eqref{campo hamiltoniano complesso}, it is a linear Hamiltonian vector field associated to the real quadratic Hamiltonian ${\cal H}$ in \eqref{generica hamiltoniana quadratica reale nel complesso}, the operator $\Phi(\vphi) = {\rm exp}({\cal R}(\vphi))$ is a symplectic. 
Assume that the map $\vphi \in \T^\nu \mapsto \Phi(\vphi) \in {\cal B}({\bf L}^2_0(\T^d))$ is a differentiable family of maps and let $\vphi \in \T^\nu \mapsto {\cal X}(\vphi) \in {\cal B}({\bf L}^2_0(\T^d))$ be a differentiable families of Hamiltonian vector fields, i.e. ${\cal  X}(\vphi) = \ii J {\cal G}(\vphi)$, ${\cal G}(\vphi) = {\cal G}(\vphi)^T$ for any $\vphi \in \T^\nu$. Arguing as in \eqref{interpretazione dinamica 1}, \eqref{interpretazione dinamica 2}, under the transformation ${\bf u} = \Phi(\omega t){\bf h}$, the PDE 
\begin{equation}\label{interpretazione dinamica 5}
\partial_t {\bf u} = {\cal X}(\omega t)  {\bf u}\,, \qquad \omega \in \R^\nu\,, \quad t \in \R\,,
\end{equation}
transforms into the PDE 
\begin{equation}\label{interpretazione dinamica 6}
\partial_t {\bf h} = {\cal X}_+(\omega t){\bf h} \,, \quad {\cal X}_+(\vphi) := \Phi_{\omega*} {\cal X}(\vphi) = \Phi(\vphi)^{- 1} {\cal X}(\vphi) \Phi(\vphi) - \Phi(\vphi)^{- 1} \omega \cdot \partial_\vphi \Phi(\vphi)\,, \quad \forall \vphi \in \T^\nu \,.
\end{equation}
If $\Phi(\vphi)$ is symplectic then the vector field ${\cal X}_+(\vphi)$ is Hamiltonian, i.e. it satisfies \eqref{operatore Hamiltoniano coordinate complesse}, \eqref{condizione R1 R2 campo hamiltoniano complesso}.  
In the following, we will consider also reparametrizations of time of the form
 $$
 \tau = t + \alpha(\omega t)\,, 
 $$
 where $\alpha : \T^\nu \to \R$ is a sufficiently smooth function with $\| \alpha\|_{{\cal C}^1}$ small enough. Then the function $t \mapsto t + \alpha(\omega t)$ is invertible and its inverse is given by 
 $$
 t = \tau + \widetilde \alpha(\omega \tau)\,. 
 $$
  by setting ${\bf v} (t) := {\cal A}(\omega t) {\bf u} := {\bf u}(t + \alpha(\omega t))$, the PDE \eqref{interpretazione dinamica 5} is transformed into 
  \begin{equation}\label{interpretazione dinamica 4}
  \partial_\tau {\bf v} = J {\cal G}_+(\omega \tau) {\bf v}\,, \qquad {\cal G}_+(\vartheta) := \frac{1}{\rho(\vartheta)} {\cal G}(\vartheta + \omega \tilde \alpha(\vartheta))\,, \quad \rho(\vartheta) := 1 + \omega \cdot \partial_\vphi \alpha\Big(\vartheta + \omega \widetilde \alpha(\vartheta) \Big)
  \end{equation}
 which is still a Hamiltonian equation.

\section{Regularization procedure of the vector field ${\cal L}(\vphi)$.}\label{riduzione linearizzato}
As described in the introduction, in this section we carry out the first part of the reduction procedure of the vector field ${\cal L}(\vphi)$, defined in \eqref{campo vettoriale main equation}, to a block-diagonal operator with constant coefficients. Our purpose is to transform the vector field ${\cal L}(\vphi)$ into the vector field ${\cal L}_4(\vphi)$ which is a regularizing perturbation of a time-independent diagonal operator, see \eqref{cal L5}. The regularizing perturbation ${\cal R}_4$ defined in \eqref{cal RM pre KAM} is the sum of a finite rank operator and a $\vphi$-dependent Fourier multiplier of order $-M$ where the constant $M$ is fixed in \eqref{scelta di M nuove norme}. In the following subsections, we describe in details all the steps needed to transform the vector field ${\cal L}(\vphi)$ into the vector field ${\cal L}_4(\vphi)$. 
\subsection{Symplectic symmetrization of the highest order}
We start by symmetryzing the highest order of the vector field
$$
{\cal L}(\vphi) = \begin{pmatrix}
0 & 1 \\
  (1 + \e a(\vphi)) \Delta + \e {\cal R}(\vphi)& 0
\end{pmatrix}\,, \quad \vphi \in \T^\nu
$$
where we recall the definitions given in \eqref{campo vettoriale main equation}, \eqref{definizione perturbazione rango finito}. For any $\vphi \in \T^\nu$, let us consider the transformation  
\begin{equation}\label{cal S1}
{\cal S}(\vphi) : H^{s }_0(\T^d, \R)\times H^{s }_0(\T^d, \R) \to H^{s + \frac12}_0(\T^d, \R) \times H^{s - \frac12}_0(\T^d, \R)\,, \quad  \begin{pmatrix}
 u \\
 \psi
\end{pmatrix}
\mapsto 
\begin{pmatrix}
\beta(\vphi)   |D|^{- \frac12} u \\
\dfrac{1}{\beta(\vphi)} |D|^{\frac12} \psi
\end{pmatrix}
\end{equation}
where $\beta : \T^\nu \to \R$ is a function close to $1$ to be determined and for all $m \in \R$, the operator $|D|^m$ is defined by
\begin{equation}\label{definizione |D| m}
 \qquad |D|^m (e^{\ii j \cdot x}) = |j|^m e^{\ii j \cdot x}\quad \forall j \neq 0\,. 
\end{equation}
For any $\vphi \in \T^\nu$, the inverse of the operator ${\cal S}(\vphi)$ is given by 
\begin{equation}\label{cal S1 inverso}
{\cal S}(\vphi)^{- 1} : H^s_0(\T^d, \R) \times H^{s - 1}_0(\T^d, \R) \to H^{s - \frac12}_0(\T^d, \R)\times H^{s - \frac12}_0(\T^d, \R)\,,\quad    \begin{pmatrix}
 u \\
 \psi
\end{pmatrix} \mapsto  \begin{pmatrix}
\dfrac{1}{\beta(\vphi)} |D|^{\frac12} u \\
\beta(\vphi) |D|^{- \frac12} \psi 
\end{pmatrix}\,.
\end{equation} 
By \eqref{push forward}, the push-forward of the vector field ${\cal L}(\vphi)$ by means of the transformation ${\cal S}(\vphi)$ is given by 
\begin{align}
{\cal L}_1(\vphi) & := {\cal S}_{\omega*} {\cal L}(\vphi) = {\cal S}(\vphi)^{- 1} {\cal L}(\vphi) {\cal S}(\vphi) - {\cal S}(\vphi)^{- 1} \omega \cdot \partial_\vphi {\cal S}(\vphi) \nonumber\\
& = \begin{pmatrix}
- \beta^{- 1}(\vphi)(\omega \cdot \partial_\vphi \beta(\vphi)) &  \beta^{- 2}(\vphi) |D|\\
 ( 1 + \e a(\vphi)) \beta^2(\vphi) |D|^{- 1} \Delta + \e  \beta^2(\vphi)|D|^{- \frac12} {\cal R}(\vphi)|D|^{- \frac12}  &-  \beta(\vphi) (\omega \cdot \partial_\vphi \beta^{- 1}(\vphi))   
\end{pmatrix}
\end{align}
and we look for $\beta: \T^\nu \to \R$ such that  
\begin{equation}
\beta^{- 2}(\vphi) = (1 + \e a(\vphi)) \beta^2(\vphi)\,,
\end{equation}
namely we choose 
\begin{equation}\label{definition beta}
\beta(\vphi) := \frac{1}{[1 + \e a(\vphi)]^{\frac14}}\,.
\end{equation}
Since 
$$
\beta(\vphi) \omega \cdot \partial_\vphi \beta^{- 1}(\vphi) = -  \frac{\omega \cdot \partial_\vphi \beta(\vphi)}{\beta(\vphi)} \quad \text{and} \quad - \Delta = |D|^2
$$
we get that 
\begin{equation}\label{cal L1}
{\cal L}_1(\vphi) = \begin{pmatrix}
 - a_0(\vphi) &  a_1(\vphi) |D| \\
 - a_1(\vphi)|D| + \e {\cal R}^{(1)}(\vphi) &   a_0(\vphi)
\end{pmatrix}\,,
\end{equation}
where 
\begin{equation}\label{definitions a0 a1}
a_0(\vphi) := \frac{\omega \cdot \partial_\vphi \beta(\vphi)}{\beta(\vphi)} \,, \quad a_1(\vphi) := \sqrt{1 + \e a(\vphi)}\,, \quad {\cal R}^{(1)}(\vphi) := \beta^2(\vphi)|D|^{- \frac12} {\cal R}(\vphi) |D|^{- \frac12}\,.
\end{equation}
Since $\beta$ is a real-valued function, the operator ${\cal S}(\vphi)$ is real for any $\vphi \in \T^\nu$ and a direct verification shows that it is also symplectic. Hence the transformed vector field ${\cal L}_1(\vphi)$ is still real and Hamiltonian. Note that by \eqref{definition beta}, \eqref{definitions a0 a1}, the functions $\beta$, $a_1$ and the operator ${\cal R}^{(1)}$ does not depend on the parameter $\omega \in \Omega$, whereas the function $a_0(\vphi) = a_0(\vphi ;\omega)$ depends on $\omega \in \Omega$. 

\noindent
Now we give some estimates on the coefficients of the vector field ${\cal L}_1(\vphi)$. 
\begin{lemma}\label{Lemma dopo simmetrizzazione}
Let $q > s_0 + 1$. Then there exists $\delta_q \in (0, 1)$ small enough such that for any $\e \in (0, \delta_q)$, for any $s_0 \leq s \leq q - 1$, the following holds: 
 the functions $\beta$, $a_0$, $a_1$ defined in \eqref{definition beta}, \eqref{definitions a0 a1} satisfy the estimates 
\begin{equation}\label{estimates a0 a1}
\|\beta^{\pm 1} - 1 \|_{s}, \| a_1 - 1\|_{s}\,, \, \| a_0\|_{s}^\Lipg \lesssim_q \e \,.
\end{equation}
The remainder ${\cal R}^{(1)}(\vphi)$ in \eqref{definitions a0 a1} has the form 
\begin{equation}\label{forma cal R (1) (vphi)}
{\cal R}^{(1)}(\vphi)[v] =\sum_{k = 1}^N b_k^{(1)}(\vphi, x) \int_{\T^d} c_k^{(1)}(\vphi, y) v(y)\, d y + c_k^{(1)}(\vphi, x) \int_{\T^d} b_k^{(1)}(\vphi, y) v(y)\, d y\,, 
\end{equation}
$\vphi \in \T^\nu, v \in L^2_0(\T^d, \R)$  (then it is symmetric ${\cal R}^{(1)}(\vphi) = {\cal R}^{(1)}(\vphi)^T$, for all $\vphi \in \T^\nu$) with
\begin{equation}\label{estimate cal R2 1}
\| b_k^{(1)} \|_s, \| c_k^{(1)} \|_s \lesssim_q  1\,, \qquad \forall k = 1, \ldots, N\,. 
\end{equation}
Furthermore, for any $s \geq 1/2$, the maps 
$$
\vphi \mapsto {\cal S}(\vphi)\,, \quad \T^\nu \to  {\cal B}\Big( H_0^{s}(\T^d, \R) \times H_0^{s}(\T^d, \R)  , H^{s + \frac12}_0(\T^d, \R) \times H^{s - \frac12}_0(\T^d, \R) \Big), 
$$
$$
\vphi \mapsto {\cal S}(\vphi)^{- 1}\,,\quad \T^\nu  \to {\cal B}\Big(H^{s + \frac12}_0(\T^d, \R) \times H^{s - \frac12}_0(\T^d, \R) , H_0^{s}(\T^d, \R) \times H_0^{s}(\T^d, \R) \Big)
$$
are ${\cal C}^1$ maps. 
\end{lemma}
\begin{proof}
The estimates \eqref{estimates a0 a1} follows by the definitions \eqref{definition beta}, \eqref{definitions a0 a1} and by Lemmata \ref{interpolazione C1 gamma}, \ref{Moser norme pesate}. Let us prove the estimates \eqref{estimate cal R2 1}. 
By \eqref{definitions a0 a1}, recalling the definition of ${\cal R}(\vphi)$ given in \eqref{definizione perturbazione rango finito}, using that $|D|^{- \frac12}$ is symmetric, one has that the operator ${\cal R}^{(1)}(\vphi)$ has the form \eqref{forma cal R (1) (vphi)} with 
$$
b_k^{(1)} (\vphi, x):= \beta(\vphi) |D|^{- \frac12} b_k(\vphi, x) \,, \quad c_k^{(1)}(\vphi, x) := \beta(\vphi) |D|^{- \frac12} c_k(\vphi, x) \,, \quad k = 1, \ldots, N\,.
$$
Then the claimed estimates follow by applying the estimate \eqref{estimates a0 a1} and applying the interpolation Lemma \ref{interpolazione C1 gamma}. A direct verification shows that ${\cal R}^{(1)}(\vphi) = {\cal R}^{(1)}(\vphi)^T$ for any $\vphi \in \T^\nu$. 
\end{proof}

\subsection{Complex variables}
Now we write the vector field ${\mathcal L}_1(\vphi)$ defined in \eqref{cal L1} in the complex coordinates introduced in \eqref{prima volta variabili complesse 0}, \eqref{matrice coordinate complesse}.
 More precisely, we conjugate the vector field ${\cal L}_1(\vphi)$ by means of the transformation ${\cal C}$ defined in \eqref{matrice coordinate complesse}. Since ${\cal C}$ is $\vphi$-independent, we get that by \eqref{push forward}, the push-forward ${\mathcal L}_2(\vphi) := {\mathcal C}_{\omega*}{\mathcal L}_1(\vphi) = {\cal C}^{- 1} {\cal L}_1(\vphi) {\cal C}$ is given by
\begin{equation}\label{cal L2 complex coordinates}
{\mathcal L}_2(\vphi)  = \begin{pmatrix}
 - \ii a_1(\vphi) |D| + \ii \e{\mathcal R}^{(2)}(\vphi) & - a_0(\vphi)+ \ii \e {\mathcal R}^{(2)}(\vphi) \\
- a_0(\vphi) - \ii \e {\mathcal R}^{(2)}(\vphi) &  \ii a_1(\vphi) |D| - \ii \e {\mathcal R}^{(2)}(\vphi)
\end{pmatrix}\,, \quad {\cal R}^{(2)}(\vphi) := \frac{{\cal R}^{(1)}(\vphi)}{\sqrt{2}}\,.
\end{equation}
 Since $a_1$ and $a_0$ are real valued functions and ${\mathcal R}^{(1)}(\vphi)$ (and then ${\mathcal R}^{(2)}(\vphi)$) is symmetric and real, the operator ${\mathcal L}_2(\vphi)$ is a Hamiltonian vector field in complex coordinates, in the sense of the Definition \eqref{operatore Hamiltoniano coordinate complesse}. We recall that the transformations ${\mathcal C}, {\cal C}^{- 1}$ satisfy the property \eqref{proprieta coordinate complesse}. 
 

\subsection{Quasi-periodic reparametrization of time}\label{subsection rimparametrizzazione tempo}
The aim of this Section is to reduce to constant coefficients the term $a_1(\vphi) |D|$ in the operator ${\cal L}_2(\vphi)$ defined in \eqref{cal L2 complex coordinates}.   In order to do this, let us consider a function $\alpha : \T^\nu \to \R$ (to be determined) and define a reparametrization of time of the form
\begin{equation}\label{riparametrizzazione tempo}
\R \to \R\,, \quad t \mapsto t +  \alpha(\omega t)\,,\quad \omega \in \Omega\,. 
\end{equation}
It is easy to verify that if $\| \alpha\|_{{\cal C}^1}$ is small enough, the above function is invertible and its inverse has the form 
\begin{equation}\label{inverso riparametrizzazione tempo}
\tau \mapsto \tau + \widetilde \alpha(\omega \tau)\,.
\end{equation}
Note that the reparametrization of time \eqref{riparametrizzazione tempo} induces also a diffeomorphism of the torus $\T^\nu$
\begin{equation}\label{diffeo toro T nu}
\T^\nu \to \T^\nu\,, \quad \vphi \mapsto \vphi + \alpha(\vphi)
\end{equation}
whose inverse is given by 
\begin{equation}\label{inverso diffeo toro T nu}
\T^\nu \mapsto \T^\nu\,, \quad \vartheta \mapsto \vartheta + \widetilde \alpha(\vartheta)\,. 
\end{equation}
The corresponding composition operators $A, A^{- 1}$ acting on the periodic functions $h : \T^\nu \times \T^d \to \C$ are given by 
\begin{equation}\label{operatori indotti da diffeo del toro}
A h(\vphi, x) := h(\vphi + \omega \alpha(\vphi), x)\,, \quad A^{- 1} h(\vartheta, x) := h(\vartheta + \omega \widetilde \alpha(\vartheta), x)\,. 
\end{equation}
According to \eqref{interpretazione dinamica 4}, under the reparametrization of time defined by 
\begin{equation}\label{reparametrizzazione tempo dinamica}
 {\cal A}(\omega t){\bf v}(t, x) := {\bf v}(t + \alpha(\omega t), x)\,, \quad {\cal A}(\omega t)^{- 1} {\bf v}(\tau, x) := {\bf v}(\tau + \widetilde \alpha(\omega \tau), x)\,,
\end{equation} 
the vector field ${\cal L}_2(\vphi)$ transforms into the vector field 
\begin{align}
{\cal L}_3(\vartheta) & := \frac{1}{\rho(\vartheta)} {\cal L}_2(\vartheta + \omega \tilde \alpha(\vartheta)) \nonumber\\
& =
\frac{1}{\rho(\vartheta)} \begin{pmatrix}
 - \ii (A^{- 1}a_1)(\vartheta) |D| + \ii \e {\mathcal R}^{(2)}(\vartheta + \omega \widetilde \alpha(\vartheta)) & - (A^{- 1}a_0)(\vartheta)+ \ii \e {\mathcal R}^{(2)}(\vartheta + \omega \widetilde \alpha(\vartheta)) \\
- (A^{- 1}a_0)(\vartheta) - \ii \e {\mathcal R}^{(2)}(\vartheta + \omega \widetilde \alpha(\vartheta)) & + \ii (A^{- 1}a_1)(\vartheta) |D| - \ii \e {\mathcal R}^{(2)}(\vartheta + \omega \widetilde \alpha(\vartheta))
\end{pmatrix}\, \label{forma preliminare cal L3}
\end{align}
where 
\begin{equation}\label{definition rho}
\rho(\vartheta) := 1 + \omega \cdot \partial_\vphi \alpha(\vartheta + \omega \widetilde \alpha(\vartheta)) = { A}^{- 1}[1 + \omega \cdot \partial_\vphi \alpha](\vartheta)\,.
\end{equation}
We want to choose the function $\alpha(\vphi)$ so that 
\begin{equation}\label{equation for alpha vartheta}
\frac{(A^{- 1}a_1)(\vartheta)}{\rho(\vartheta)} = m\,, \quad \forall \vartheta \in \T^\nu,
\end{equation}
for some constant $m \in \R$ to be determined. The above equation leads to 
\begin{equation}\label{equation for alpha}
m \big(1 + \omega \cdot \partial_\vphi \alpha(\vphi) \big) = a_1(\vphi)\, \quad \forall \vphi \in \T^\nu\,.
\end{equation}
Integrating on $\T^\nu$ we fix the value of $m$ as 
\begin{equation}\label{definition m}
m := \frac{1}{(2 \pi)^\nu} \int_{\T^\nu} a_1(\vphi)\, d \vphi
\end{equation}
and then, assuming that $\omega \in DC(\gamma, \tau)$, for some $\gamma, \tau > 0$ (see the definition \eqref{diofantei Kn}), we get
\begin{equation}\label{definition alpha}
\alpha(\vphi) = (\omega \cdot \partial_\vphi)^{- 1} \Big[ \frac{a_1}{m} - 1 \Big](\vphi)\,
\end{equation}
where the operator $(\omega \cdot \partial_\vphi)^{- 1}$ is defined by \eqref{om d vphi inverso}. Note that, since the function $a_1$ is real valued, then $m$ is real and $\alpha$ is a real valued function.

\noindent
By \eqref{forma preliminare cal L3}-\eqref{definition alpha}, the vector field ${\cal L}_3(\vartheta)$ has then the form 
\begin{equation}\label{cal L3}
{\cal L}_3(\vartheta) := \begin{pmatrix}
- \ii m |D| + \ii \e {\cal R}^{(3)}(\vartheta) & a_2(\vartheta) + \ii \e {\cal R}^{(3)}(\vartheta) \\
a_2(\vartheta) - \ii \e {\cal R}^{(3)}(\vartheta) &  \ii m |D| - \ii \e {\cal R}^{(3)}(\vartheta)
\end{pmatrix}\,
\end{equation}
where 
\begin{equation}\label{definitions a2 cal R3}
a_2 (\vartheta) := \rho^{- 1}(\vartheta) { A}^{- 1}[a_0](\vartheta) \,, \qquad {\cal R}^{(3)}(\vartheta) := \rho(\vartheta)^{- 1}  {\cal R}^{(2)}(\vartheta + \omega \widetilde \alpha(\vartheta))\,.
\end{equation}
The operator ${\cal L}_3(\vartheta)$ is still a Hamiltonian vector field in complex coordinates, since ${\cal L}_2(\vartheta)$ is Hamiltonian and the reparametrization of time ${\cal A}$ preserves the Hamiltonian structure (see Section \ref{formalismo hamiltoniano complesso}).  We point out that by \eqref{definition m}, \eqref{definitions a0 a1}, the constant $m$ is independent of the parameter $\omega \in \Omega$, whereas by \eqref{definition alpha}, \eqref{operatori indotti da diffeo del toro}, \eqref{definition rho}, \eqref{definitions a2 cal R3}, the functions $\alpha, \widetilde \alpha , \rho, a_2$ and the operator ${\cal R}^{(3)}$ depends in a Lipschitz way with respect to the parameter $\omega \in DC(\gamma, \tau)$.
\begin{lemma}\label{stime step 1}
Let $ \tau > 0$, $\gamma \in (0, 1)$ and $\omega \in  DC(\gamma, \tau)$ (recall \eqref{diofantei Kn}). Then there exists a constant $\sigma = \sigma(\tau) >0$ such that if $q > s_0 + \sigma$, there exists $\delta_q \in (0, 1)$ such that if $\e \gamma^{- 1} \leq \delta_q$, for all $s_0 \leq s \leq q - \sigma$ the following estimates hold:
\begin{equation}\label{stime m}
|m - 1|\,,  \| a_2\|_{s}^\Lipg, \| \rho^{\pm 1} - 1\|_{s}^\Lipg \lesssim_q \e, \quad \| \alpha\|_{s}^\Lipg, \| \widetilde \alpha\|_{s}^\Lipg \lesssim_q \e \gamma^{- 1}
\end{equation}
The symmetric operator ${\cal R}^{(3)}(\vartheta)$ defined in \eqref{definitions a2 cal R3} has the form
\begin{equation}\label{forma buona resto cal R (3) (vartheta)}
{\cal R}^{(3)}(\vartheta)[u] = \sum_{k = 1}^N b_k^{(3)}(\vartheta, x) \int_{\T^d} c_k^{(3)}(\vartheta, y) v(y)\, d y + c_k^{(3)}(\vartheta, x) \int_{\T^d} b_k^{(3)}(\vartheta, y) v(y)\, d y\,, 
\end{equation}
$\vphi \in \T^\nu$, $v \in L^2_0(\T^d)$, with  
\begin{equation}\label{stime cal R3 1}
\| b_k^{(3)} \|_s^\Lipg, \| c_k^{(3)} \|_s^\Lipg \lesssim_q 1\,, \qquad k = 1, \ldots, N\,. 
\end{equation}
\end{lemma}
\begin{proof}
The estimates \eqref{stime m} follow by \eqref{definition m}, \eqref{definition alpha}, \eqref{definitions a2 cal R3} and by the estimates \eqref{estimates a0 a1} by applying Lemmata \ref{interpolazione C1 gamma}, \ref{Moser norme pesate}, \ref{lemma:utile}. The formula \eqref{forma buona resto cal R (3) (vartheta)} follows by \eqref{forma cal R (1) (vphi)}, \eqref{cal L2 complex coordinates}, \eqref{definitions a2 cal R3}, by defining $b_k^{(3)} := 2^{- \frac14} \rho^{- \frac12} b_k^{(1)}$, $c_k^{(3)} := 2^{- \frac14} \rho^{- \frac12} c_k^{(1)}$, $k = 1, \ldots, N$ and the estimates \eqref{stime cal R3 1} follow by \eqref{estimate cal R2 1}, \eqref{stime m} and Lemmata \ref{interpolazione C1 gamma}, \ref{lemma:utile}.
\end{proof}

\subsection{Symplectic reduction up to order $|D|^{- M}$.}
Introducing the notation 
\begin{equation}\label{notazione T Id0}
 \quad T := \begin{pmatrix}
- {\rm I d} & 0 \\
0 &  {\rm Id}
\end{pmatrix}\,, \quad {\rm Id} : L^2_0(\T^d) \to L^2_0(\T^d) \quad \text{is\,the \,identity}
\end{equation}
and renaming the variable $\vartheta = \vphi$, we can write the vector field in \eqref{cal L3} as 
\begin{equation}\label{cal L3 notazioni compatte}
{\cal L}_3(\vphi) = \ii m T |D| + A_2(\vphi) + \e {\cal R}_3(\vphi)\,,
\end{equation}
where 
\begin{equation}\label{A0 bf R0}
 A_2(\vphi) := \begin{pmatrix}
0 & a_2(\vphi) \\
a_2 (\vphi) & 0
\end{pmatrix}\,,\quad {\cal R}_3(\vphi) := \ii \begin{pmatrix}
{\cal R}^{(3)}(\vphi) & {\cal R}^{(3)}(\vphi) \\
- {\cal R}^{(3)}(\vphi) & - {\cal R}^{(3)}(\vphi)
\end{pmatrix}\,, \quad \vphi \in \T^\nu
\end{equation}
and the operator ${\cal R}^{(3)}(\vphi)$, defined in \eqref{definitions a2 cal R3}, has the form \eqref{forma buona resto cal R (3) (vartheta)}. The aim of this section is to conjugate ${\cal L}_3(\vphi)$ to the vector field ${\cal L}_4(\vphi)$ defined in \eqref{cal L5} which is the sum of a diagonal operator and a regularizing remainder. Since the operator ${\cal R}^{(3)}(\vphi)$ is finite rank operator of the form \eqref{forma buona resto cal R (3) (vartheta)}, it is already regularizing. Hence in the following two Sections \ref{sezione block decoupling}, \ref{riduzione diagonale dopo decoupling}, we neglect the operator ${\cal R}_3(\vphi)$ in \eqref{cal L3 notazioni compatte} and we work with the vector field  
\begin{equation}\label{cal L 3 0}
L_3^{(0)}(\vphi) :=  \ii m T |D| + A_2(\vphi)\,, \quad \vphi \in \T^\nu\,.
\end{equation}
We compute the complete conjugation of ${\cal L}_3$ in Section \ref{cal L3 notazioni compatte}. 
\subsubsection{Block-decoupling up to order $|D|^{- M}$.}\label{sezione block decoupling}
Given a positive integer $M$, our goal is to conjugate the operator $L_3^{(0)}$ in \eqref{cal L 3 0} to the 
operator $ {L}_3^{(M)} $ in \eqref{cal LN decoupling} whose off-diagonal part 
  $ {Q}_M$ is an operator of order $- M$. This is achieved by applying iteratively $ M $-times  
a conjugation map which transforms the off-diagonal block operator into a $ 1 $-smoother ones. For such a procedure we will use the class of $\vphi$-dependent Fourier multipliers introduced in Section \ref{sezione pseudo diff}. 

\noindent
We describe the inductive step of such a procedure. We assume that $q > s_0 + \sigma + M$, where the constant $\sigma =\sigma(\tau)$ is given in Lemma \ref{stime step 1} and $M \in \N$ is the number of the steps of this regularization procedure. In this section we use the following notation: If $n \in \{ 1, \ldots, M \}$, $s \geq 0$, we write
$$
a \lesssim_{n, s} b \quad \ \Longleftrightarrow \quad a \leq C(n, s) b
$$  
for some constant $C(n, s) > 0$ (that may depend also on $d, \tau, \nu$).

\noindent
At the $n$-th step, we have a Hamiltonian vector field  
\begin{equation}\label{cal L 3 n}
{ L}_3^{(n)}(\vphi) = \ii m T |D| + R_n(\vphi) + Q_n(\vphi)\,,
\end{equation}
where $R_n(\vphi) = R_n(\vphi; \omega)$, $Q_n(\vphi) = Q_n(\vphi; \omega)$, $\omega \in DC(\gamma, \tau)$ are Hamiltonian vector fields of the form 
\begin{equation}\label{Rn}
R_n := \ii \begin{pmatrix}
{\rm Op}(r_n) & 0 \\
0 & - \overline{{\rm Op}(r_n)}
\end{pmatrix}\,, \qquad Q_n := \ii \begin{pmatrix}
0 & {\rm Op}(q_n) \\
- \overline{{\rm Op}(q_n)}  & 0
\end{pmatrix}
\end{equation}
and $r_n(\vphi, \cdot) \in S^{- 1}$, $q_n(\vphi, \cdot) \in S^{- n}$. Moreover they satisfy the estimates 
\begin{equation}\label{stima Rn}
\norma R_n \norma_{- 1, s}^\Lipg, \norma Q_n \norma_{ -n, s}^\Lipg \lesssim_{ n, q} \e \,, \qquad \forall s_0 \leq s \leq q - n - \sigma\,
\end{equation}
where $\sigma = \sigma (\tau) > 0$ is given in Lemma \ref{stime step 1}. Recall that the definition of the norm $\norma \cdot \norma_{m, s}$ is given in \eqref{norma pseudo diff}. 

\noindent
{\sc Initialization.} The Hamiltonian vector field $ {L}_3^{(0)} (\vphi)$ in \eqref{cal L 3 0} 
satisfies the assumptions
\eqref{cal L 3 n}-\eqref{stima Rn}, with $R_0(\vphi) = 0$ and $Q_0(\vphi) = A_2(\vphi) \in OPS^0$, by Lemma \ref{stime step 1}. 

\medskip

\noindent
{\sc inductive step.}
We consider a symplectic transformation of the form 
\begin{equation}\label{cal Vn definizione}
{\cal V}_n := {\rm exp}(\ii V_n) 
\end{equation}
where the operator $V_n$ has the form 
\begin{equation}\label{Vn definizione}
V_n := \begin{pmatrix}
0 & {\rm Op}(v_n) \\
- \overline{{\rm Op}(v_n)} & 0
\end{pmatrix}\,, \qquad v_n \in S^{- n - 1}\,.
\end{equation}
We write 
\begin{equation}\label{espansione cal Vn}
{\cal V}_n = {\rm Id} + \ii V_n + {\cal V}_{n, \geq 2}\,, \qquad {\cal V}_{n, \geq 2} := \sum_{k \geq 2} \frac{\ii^k}{k !} V_n^k\,. 
\end{equation}
In the above formula, with a slight abuse of notations we denote by ${\rm Id}: {\bf L}^2_0(\T^d) \to {\bf L}^2_0(\T^d)$ the identity on the space ${\bf L}^2_0(\T^d)$. Note that, by Lemma \ref{esponenziale norma pseudo diff}, one gets ${\cal V}_{n , \geq 2} \in OPS^{- 2(n + 1)}$. We now compute the push-forward $({\cal V}_n)_{\omega*} L_3^{(n)}(\vphi)$. By \eqref{interpretazione dinamica 6} one has 
\begin{align}
({\cal V}_n)_{\omega*} L_3^{(n)}(\vphi) & = {\cal V}_n(\vphi)^{- 1} \Big( L_3^{(n)}(\vphi) {\cal V}_n(\vphi) -  \omega \cdot \partial_\vphi {\cal V}_n(\vphi) \Big)\,. \label{joao mario0} 
\end{align}
Since $\omega \cdot \partial_\vphi{\cal V}_n(\vphi) = \omega \cdot \partial_\vphi \Big({\cal V}_n(\vphi) - {\rm Id} \Big)$, by Lemmata \ref{composizione fourier multiplier}, \ref{esponenziale norma pseudo diff}, one has 
\begin{align}
 - {\cal V}_n(\vphi)^{- 1} \omega \cdot \partial_\vphi{\cal V}_n(\vphi) =  - {\cal V}_n(\vphi)^{- 1} \omega \cdot \partial_\vphi \Big({\cal V}_n(\vphi) - {\rm Id} \Big) \in OPS^{- n - 1}\,. \label{proimo pezzo push forward decoupling}
\end{align}
Moreover
\begin{align}
L_3^{(n)}(\vphi) {\cal V}_n(\vphi)&  \stackrel{\eqref{cal L 3 n}}{=}   \ii{\cal V}_n(\vphi) m T |D|  + [\ii m T |D|, \ii V_n(\vphi)] + Q_n(\vphi)  + R_n(\vphi) \nonumber\\
& \quad + [\ii m T |D|, {\cal V}_{n, \geq 2}(\vphi)]  + (R_n(\vphi) + Q_n(\vphi))({\cal V}_n(\vphi) - {\rm Id})\,. \label{semiconiugio cal Ln cal Vn} 
\end{align}
Note that $[\ii m T |D|, {\cal V}_{n, \geq 2}(\vphi)] \in OPS^{- 2n - 1} \subset OPS^{- n - 1}$, $(R_n(\vphi) + Q_n(\vphi))({\cal V}_n(\vphi) - {\rm Id}) \in OPS^{- n - 2} \subset OPS^{- n - 1}$, therefore the only off-diagonal term of order $-n$ (which we want to eliminate) is given by  $[\ii m T |D|, \ii V_n(\vphi)] + Q_n(\vphi)$. We want to choose $V_n(\vphi)$ so that 
\begin{equation}\label{equazione omologica decoupling}
[\ii m T |D|, \ii V_n(\vphi)] + Q_n(\vphi) = 0\,.
\end{equation}
By a direct calculation, one has 
\begin{align}
[\ii m T |D|, \ii V_n(\vphi)] + Q_n(\vphi) 
& = \begin{pmatrix}
0 &{\rm Op}\Big(  2 m |j| v_n(\vphi, |j|) + \ii q_n(\vphi, |j|) \Big)   \\
 \overline{{\rm Op}\Big(  2 m |j| v_n(\vphi, |j|) + \ii q_n(\vphi, |j|) \Big)} & 0 
\end{pmatrix}\,.
\end{align}
Then $[\ii m T |D|, \ii V_n] + Q_n = 0$ if we choose the symbol $v_n$ so that 
\begin{equation}\label{scelta simbolo vn}
v_n(\vphi, \alpha) : = - \dfrac{\ii q_n(\vphi, \alpha)}{2 m \alpha}\,, \qquad \forall \vphi \in \T^\nu\,, \qquad \forall \alpha \in \sigma_0(\sqrt{- \Delta})\,.
\end{equation}
Note that since $q_n(\vphi, \cdot) \in S^{- n}$, the symbol $v_n(\vphi, \cdot) \in S^{- n - 1}$ for any $\vphi \in \T^\nu$. 
\begin{lemma}\label{lemma:trasformazione cal Vn decoupling}
For any $s_0 \leq s \leq q - n - \sigma$, the operators $V_n(\vphi), {\cal V}_n(\vphi) - {\rm Id} \in S^{- n - 1}$ and ${\cal V}_{n, \geq 2}(\vphi) \in OPS^{- 2(n + 1)}$, see \eqref{Vn definizione}, \eqref{espansione cal Vn} (which depend on the parameter $\omega \in DC(\gamma, \tau)$) satisfy the estimates 
\begin{equation}\label{stime Vn}
\norma V_n \norma_{- n- 1, s}^\Lipg\,,\, \norma {\cal V}_n^{\pm 1} - {\rm Id} \norma_{- n - 1, s}\,,\, \norma {\cal V}_{n, \geq 2} \norma_{-2(n + 1), s} \lesssim_{n, q} \e\,.
\end{equation}
\end{lemma}
\begin{proof}
The estimate for the operator $V_n$ follows by the definitions \eqref{Vn definizione}, \eqref{scelta simbolo vn} and by the estimates \eqref{stime m}, \eqref{stima Rn}. The estimates for ${\cal V}_n(\vphi) - {\rm Id}$ and ${\cal V}_{n, \geq 2}(\vphi)$ follow by applying Lemma \ref{esponenziale norma pseudo diff}, using the estimate on $V_n(\vphi)$. 
\end{proof}

\noindent
By \eqref{joao mario0}-\eqref{equazione omologica decoupling}, one gets 
\begin{equation}\label{cal L n+1 decoupling}
{L}_3^{(n + 1)}(\vphi)  =  \ii m T |D| + R_n(\vphi) + P_n(\vphi)\,,
\end{equation} 
where 
\begin{equation}\label{definizione Pn}
P_n := ({\cal V}_n^{- 1} - {\rm Id}) R_n + {\cal V}_n^{- 1} \Big(  [\ii m T |D|, {\cal V}_{n , \geq 2}] + (R_n + Q_n)({\cal V}_n - {\rm Id})  - \omega \cdot \partial_\vphi ({\cal V}_n - {\rm Id}) \Big)\,.
\end{equation}
Note that $P_n$ is the only operator which contains off-diagonal terms. In the next lemma we provide some estimates on the remainder $P_{n}$. 
\begin{lemma}\label{lemma:stime nuovo resto decoupling}
For any $s_0 \leq s \leq q - \sigma - n - 1$, the operator $P_{n}(\vphi) = P_n(\vphi; \omega) \in OPS^{- n - 1}$, $\omega \in DC(\gamma, \tau)$ satisfies the estimates 
\begin{equation}\label{stime R n*}
\norma P_n \norma_{- n - 1, s}^\Lipg \lesssim_{n, q} \e\,.
\end{equation}
\end{lemma}
\begin{proof}
The Lemma follows by Lemma \ref{lemma:trasformazione cal Vn decoupling}, the estimates \eqref{stima Rn}, by applying the property \eqref{proprieta facili norma pseudo 1} and Lemma \ref{composizione fourier multiplier} to estimate all the terms in \eqref{definizione Pn}.
\end{proof}
By \eqref{cal L n+1 decoupling} and \eqref{stime R n*}  
the vector field $ {L}_3^{(n+1)}(\vphi) $ has the same form \eqref{cal L 3 n}-\eqref{Rn} 
with $ { R}_{n+1}(\vphi) $, $ {Q}_{n+1} (\vphi)$ 
that satisfy the estimates \eqref{stima Rn}
at the step $ n+  1$. Since $L_3^{(n)}$ is a Hamiltonian vector field and ${\cal V}_n$ is symplectic, the vector field $L_3^{(n + 1)}$ is still Hamiltonian. We can repeat iteratively the procedure of 
Lemmata \ref{lemma:trasformazione cal Vn decoupling} and \ref{lemma:stime nuovo resto decoupling}. Applying it $ M $-times, we derive the following proposition. 

\begin{proposition}\label{Lemma finale decoupling}
Let $\gamma \in (0, 1)$, $\tau > 0$, $M \in \N$, $q > s_0 + \sigma + M$. Then there exists a constant $\delta_q \in (0, 1)$ (possibly smaller than the one appearing in Lemma \ref{stime step 1}) such that for $\e \gamma^{- 1} \leq \delta_q$, for any $s_0 \leq s \leq q - \sigma - M$, for any $\omega \in DC(\gamma, \tau)$, the following holds: 
the symplectic invertible map $ \widetilde{\cal V}_M (\vphi):= {\cal V}_0(\vphi) \circ \ldots \circ {\cal V}_{M - 1}(\vphi) \in OPS^0$  satisfies  the estimate 
\begin{equation}\label{stime tilde cal V M}
\norma \widetilde{\cal V}_M^{\pm 1} \norma_{0,s}^\Lipg\,,\, \norma \widetilde{\cal V}_M^T \norma_{0,s}^\Lipg \lesssim_{M, q } 
1  \, , 
\end{equation}
and the push forward $ L_3^{(M)}(\vphi) := (\widetilde{\cal V}_M)_{\omega*} {L}_3^{(0)}(\vphi) $ of the Hamiltonian vector field $L_3^{(0)}(\vphi)$ in \eqref{cal L 3 0}  is the Hamiltonian vector field
\begin{equation}\label{cal LN decoupling}
{L}_3^{(M)}(\vphi) =   \ii m T |D| + R_M(\vphi) + Q_M(\vphi)
\end{equation}
where $R_M(\vphi) = R_M(\vphi; \omega), Q_M(\vphi) = Q_M(\vphi; \omega)$, $\omega \in DC(\gamma, \tau)$ have the form  
\be\label{RM decoupling}
R_M := \ii \begin{pmatrix}
{\rm Op}(r_M) & 0 \\
0 & - \overline{{\rm Op}(r_M)}
\end{pmatrix} \,, \qquad r_M(\vphi, \cdot ) \in S^{- 1} \,, \quad 
\ee
\begin{equation}\label{QM decoupling}
Q_M := \ii \begin{pmatrix}
0 & {\rm Op}(q_M) \\
- \overline{{\rm Op}(q_M)} & 0
\end{pmatrix}\,, \qquad q_M(\vphi, \cdot) \in S^{- M}
\end{equation}
and satisfy the estimates 
\begin{equation}\label{stime RM QM}
\norma R_M \norma_{-1, s}^\Lipg\,,\, \norma Q_M \norma_{- M, q}^\Lipg \lesssim_{M, s}  
\e \,, \quad \forall s_0 \leq s \leq q - \sigma - M\,. 
\end{equation}
\end{proposition}

\begin{proof}
We need only to prove the estimates \eqref{stime tilde cal V M}. For any $n = 1, \ldots, M - 1$ one has 
$$
\norma {\cal V}_n \norma_{0, s} \leq 1 + \norma {\cal V}_n - {\rm Id} \norma_{0, s} \stackrel{\eqref{proprieta facili norma pseudo 1}}{\leq} 1 + \norma {\cal V}_n - {\rm Id} \norma_{- n - 1, s} \stackrel{\eqref{stime Vn}}{\lesssim_{n, s}} 1\,,
$$
for any $s_0 \leq s \leq q - n - \sigma$. Since $n \leq M$, one has that the above estimate holds for any $s_0 \leq s \leq q - \sigma - M$. 
Applying Lemma \ref{composizione fourier multiplier} and using the above estimate one gets the estimate \eqref{stime tilde cal V M} for $\widetilde{\cal V}_M$. 
The estimates for $\widetilde{\cal V}_M^{- 1}$ follow by similar arguments and the estimates for $\widetilde{\cal V}_M^T$ follow since $\norma \widetilde{\cal V}_M^T \norma_{0, s} \leq \norma \widetilde{\cal V}_M \norma_{0, s} $ and then the lemma is proved. 
\end{proof}

\noindent
The operator $ L_3^{(M)}(\vphi) $ in \eqref{cal LN decoupling} is a space-diagonal operator up to the smoothing remainder $Q_M(\vphi)  \in OPS^{- M} $. The prize which has been paid is that there is a {\it loss of regularity} of $M$ derivatives with respect to the variable $\vphi$. In any case, the number of regularizing steps $ M $ will be fixed in \eqref{scelta di M nuove norme}.  
 \subsubsection{Reduction to constant coefficients of the diagonal reminder $R_M$}\label{riduzione diagonale dopo decoupling}
Our next aim is to eliminate the $\vphi$ dependence from the diagonal remainder $R_M(\vphi)$ of the Hamiltonian vector field ${L}_3^{(M)}(\vphi)$ defined in \eqref{cal LN decoupling}. In order to achieve this purpose, we look for a transformation of the form 
 \begin{equation}\label{operatore cal E}
 {\cal E}(\vphi) := {\rm exp}(\ii E(\vphi))\,, \qquad E(\vphi) := \begin{pmatrix}
 {\rm Op}(e(\vphi, |j|)) & 0 \\
 0 & - \overline{{\rm Op}(e(\vphi, |j|)) }
 \end{pmatrix}\,, \qquad e (\vphi, \cdot) \in S^{- 1}\,.
 \end{equation}
 Note that for any $\vphi \in \T^\nu$,
 \begin{equation}\label{cal E ed inverso esplicito}
 {\cal E}(\vphi)^{\pm 1} = \begin{pmatrix}
 {\rm Op}\big({\rm exp}( \pm \ii e(\vphi, |j|)) \big) & 0 \\
  0 & \overline{{\rm Op}\big(  {\rm exp}(\pm \ii e(\vphi, |j|)) \big)}
 \end{pmatrix}
 \end{equation}
 and 
 \begin{align}
{\cal E}(\vphi)^{- 1} \omega \cdot \partial_\vphi {\cal E}(\vphi) & = \begin{pmatrix}
{\rm Op}\Big( \ii \omega \cdot \partial_\vphi e(\vphi, \cdot)\Big) & 0 \\
0 & \overline{{\rm Op}\Big( \ii \omega \cdot \partial_\vphi e(\vphi, \cdot)\Big)}
\end{pmatrix}\,.  \label{omega d vphi E E inverse}
 \end{align}
 Therefore by \eqref{push forward}, \eqref{cal E ed inverso esplicito}, \eqref{omega d vphi E E inverse} and recalling the properties stated in \eqref{proprieta elementari simboli classe Sm}, the vector field $L_4^{(M)}(\vphi) := {\cal E}_{\omega*} L_3^{(M)}(\vphi)$ is given by
 \begin{align}
 & {L}_4^{(M)} := {\cal E}^{- 1 } {L}_3^{(M)} {\cal E} - {\cal E}(\vphi)^{- 1} \omega \cdot \partial_\vphi {\cal E} \nonumber\\
 &  = \begin{pmatrix} 
  {\rm Op}\big({\rm exp}(- \ii e) \big) \Big( \ii m |D| + \ii {\rm Op}(r_M) \Big) {\rm Op}\big({\rm exp}(\ii e) \big) & 0 \\
 0 & \overline{ {\rm Op}\big({\rm exp}(- \ii e) \big) \Big(  \ii m |D| + \ii {\rm Op}(r_M) \Big)  {\rm Op}\big({\rm exp}(\ii e) \big)}
 \end{pmatrix} \nonumber\\
 & \qquad + {\cal E}^{- 1 } Q_M {\cal E} - \begin{pmatrix}
{\rm Op}\Big( \ii \omega \cdot \partial_\vphi e\Big) & 0 \\
0 & \overline{{\rm Op}\Big( \ii \omega \cdot \partial_\vphi e\Big)}
\end{pmatrix} \, \nonumber\\
& = \begin{pmatrix} 
 \ii m |D| +  {\rm Op}\big(   \ii r_M - \ii \omega \cdot \partial_\vphi e  \big) & 0 \\
 0 &- \ii m |D| + \overline{ {\rm Op}\Big(  \ii r_M - \ii \omega \cdot \partial_\vphi e  \Big)}
 \end{pmatrix} +  {\cal E}^{- 1 } Q_M {\cal E}\,. \label{cal L4 M}
 \end{align}
 Note that to shorten notations, in the above chain of equalities, we avoided to write the dependence on $\vphi$. In order to eliminate the $\vphi$-dependence from the symbol $r_M(\vphi, |j|)$, we need to solve the equation 
$$
- \omega \cdot \partial_\vphi e (\vphi, |j|) + r_M(\vphi, |j|) = c(|j|) \in \R \,, \qquad \forall j \in \Z^d \setminus \{ 0 \}\,, \quad \forall \vphi \in \T^\nu
$$
 or equivalently
 \begin{equation}\label{equazione omologica diagonalizzazione decoupling}
 - \omega \cdot \partial_\vphi e(\vphi, \alpha) + r_M (\vphi, \alpha) = c (\alpha)\,, \qquad \forall (\vphi, \alpha) \in \T^\nu \times \sigma_0(\sqrt{- \Delta})\,, \qquad c (\alpha) \in \R\,.
 \end{equation}
 Integrating with respect to $\vphi$ the above equation, we determine the value of the constant $c( \alpha)$, namely 
 \begin{equation}\label{definizione costante c alpha}
 c(\alpha) := \frac{1}{(2 \pi)^\nu} \int_{\T^\nu} r_M(\vphi, \alpha)\, d \vphi\,, \qquad \forall \alpha \in \sigma_0(\sqrt{- \Delta})
 \end{equation}
 and then we choose 
 \begin{equation}\label{definizione simbolo e vphi alpha}
 e(\vphi, \alpha) := (\omega \cdot \partial_\vphi)^{- 1} \Big(r_M(\vphi, \alpha) - c(\alpha) \Big)\,, \qquad \forall (\vphi, \alpha) \in \T^\nu \times \sigma_0(\sqrt{- \Delta})\,,
 \end{equation}
 (note that $\omega \in DC(\gamma, \tau)$ and recall the definition \eqref{om d vphi inverso}). By \eqref{cal L4 M}, \eqref{equazione omologica diagonalizzazione decoupling},  \eqref{QM decoupling}, \eqref{cal E ed inverso esplicito} one gets  
\begin{align}
&  L_4^{(M)}(\vphi) = {\cal E}_{\omega *} L_3^{(M)}(\vphi) = \ii D_M T + Q_{M, 4}(\vphi)\,, \label{forma finale cal L4 M} 
\end{align}
where the diagonal operator $D_M$ is defined as 
\begin{equation}\label{operatore diagonal pre KAM}
D_M := m |D| + {\rm Op}(c(|j|)) = {\rm diag}_{j \in \Z^d \setminus \{ 0 \}}\big( m |j| + c(|j|) \big)
\end{equation}
and
\begin{align}\label{Q M 4}
Q_{M, 4}(\vphi) := {\cal E}(\vphi)^{- 1} Q_M(\vphi) {\cal E}(\vphi) & = \ii \begin{pmatrix}
0 & {\rm Op}(q_{M, 4})  \\
- \overline{{\rm Op}(q_{M, 4})} & 0 
\end{pmatrix}\,, \quad q_{M, 4}:= q_M {\rm exp}(- 2 \ii e) \,.
\end{align}
\begin{lemma}\label{maradona lemma}
Let $\gamma \in (0, 1)$, $\tau > 0$, $M \in \N$, $q > s_0 + \sigma + 2 \tau + M + 1$. Then there exists a constant $\delta_q \in (0, 1)$ (possibly smaller than the one appearing in Proposition \ref{Lemma finale decoupling}) such that for $\e \gamma^{- 1} \leq \delta_q$, for any $s_0 \leq s \leq q - M- \sigma - 2 \tau - 1$,  the following holds:
 for any $\alpha \in \sigma_0(\sqrt{- \Delta})$, the constant $c(\alpha) = c(\alpha; \omega)$, given in \eqref{definizione costante c alpha}, is real and defined for all the parameters $\omega \in DC(\gamma, \tau)$. Furthermore it satisfies the Lipschitz estimate   
\begin{equation}\label{stime c(alpha)}
\sup_{\alpha \in \sigma_0(\sqrt{- \Delta})} |c(\alpha)|^\Lipg\alpha \lesssim_{M, q} \e\,, 
\end{equation} 
The symplectic invertible operator ${\cal E}(\vphi) = {\cal E}(\vphi; \omega) \in OPS^0$, $\omega \in DC(\gamma, \tau)$, defined in \eqref{operatore cal E} satisfies the estimates 
\begin{equation}\label{stime cal E}
\norma {\cal E}^{\pm 1} \norma_{0, s}^\Lipg\,,\, \norma {\cal E}^T \norma_{0, s}^\Lipg \lesssim_{M, q} 1 
\end{equation}
The Hamiltonian vector field $Q_{M, 4}(\vphi) = Q_{M, 4}(\vphi; \omega) \in OPS^{- M}$, $\omega \in DC(\gamma, \tau)$ defined in \eqref{Q M 4} satisfies the estimates 
\begin{equation}\label{stime QM 4}
\norma Q_{M, 4}\norma_{- M, s}^\Lipg \lesssim_{M, q} \e  
\end{equation}
\end{lemma}
\begin{proof}
Since the remainder $R_M$ in \eqref{RM decoupling} is a Hamiltonian vector field, then ${\rm Op}(r_M)$ is self-adjoint, hence by \eqref{autoaggiuntezza pseudo diff} the symbol $r_M(\vphi, \alpha)$ is real, implying that, by \eqref{definizione costante c alpha}, $c(\alpha)$ is real for any $\alpha \in \sigma_0(\sqrt{- \Delta})$. The estimate \eqref{stime c(alpha)} follows by \eqref{definizione costante c alpha}, \eqref{stime RM QM}.
The estimates \eqref{stime cal E} follow by \eqref{cal E ed inverso esplicito}, \eqref{definizione simbolo e vphi alpha}, \eqref{stime RM QM}, \eqref{stime c(alpha)} (using also Lemma \ref{Moser norme pesate} to estimate $\| {\rm exp}(\ii e) \|_s$.)

\noindent
The estimate \eqref{stime QM 4} follows by Lemma \ref{composizione fourier multiplier} and by the estimates \eqref{stime RM QM},  \eqref{stime cal E}. 
\end{proof}

\subsubsection{Conjugation of the operator ${\cal L}_3$ in \eqref{cal L3 notazioni compatte}} 
Now we compute the conjugation of the vector field ${\cal L}_3 = { L}_3^{(0)} + {\cal R}_3$ in \eqref{cal L3 notazioni compatte} (see \eqref{A0 bf R0}, \eqref{cal L 3 0}). First, we link the number of regularization steps with the regularity $q$ of the functions $a(\vphi), b_k(\vphi, x), c_k(\vphi, x)$, $k = 1, \ldots, N$ (recall \eqref{perturbazione P omega t}, \eqref{definizione perturbazione rango finito}). We define 
\begin{equation}\label{scelta di M nuove norme}
M  = M(q) := [q/2]\,, \qquad \overline \mu = \overline \mu (\tau, d) := \frac{d - 1}{2} + \sigma + 2 \tau + 1
\end{equation}
and we define the map  
\begin{equation}\label{definizione cal TM}
{\cal T} :=  \widetilde{\cal V}_M \circ {\cal E}\,.
\end{equation}
By \eqref{cal LN decoupling}, \eqref{forma finale cal L4 M} one gets that 
\begin{equation}\label{cal L5}
{\cal L}_4(\vphi) := ({\cal T})_{\omega*} {\cal L}_3(\vphi) = \ii D_M T + {\cal R}_4(\vphi)
\end{equation}
where the diagonal operator $D_M$ is defined in \eqref{operatore diagonal pre KAM}, $T$ is defined in \eqref{notazione T Id0} and the operator ${\cal R}_4$ is defined by
\begin{equation}\label{cal RM pre KAM}
{\cal R}_4(\vphi) := Q_{M, 4}(\vphi) + \e {\cal T}(\vphi)^{- 1} {\cal R}_3(\vphi) {\cal T}(\vphi)\,, \quad \vphi \in \T^\nu\,. 
\end{equation}
\begin{lemma}\label{ultimo lemma pre riducibilita}
Let $\gamma \in (0, 1)$, $\tau > 0$, $q > 2 (s_0 + \overline \mu)$, where $\overline \mu$ is defined in \eqref{scelta di M nuove norme}. Then there exists $\delta_q \in (0, 1)$ (possibly smaller than the one appearing in Lemma \ref{maradona lemma}) such that if $\e \gamma^{- 1} \leq \delta_q$, for all $s_0 \leq s \leq [q / 2] - \overline \mu$, the following holds: the symplectic invertible operator ${\cal T}(\vphi) = {\cal T}(\vphi; \omega) \in OPS^0$, $\omega \in DC(\gamma, \tau)$ defined in \eqref{operatore cal E} satisfies the estimates 
\begin{equation}\label{stime cal TM}
\norma {\cal T}^{\pm 1} \norma_{0, s}^\Lipg\,,\, \norma {\cal T}^T \norma_{0, s}^\Lipg \lesssim_q  1\,. 
\end{equation}
As a consequence one has ${\cal T}^{\pm 1} \in {\cal C}^1 \big( \T^\nu, {\cal B}({\bf H}^s_0(\T^d)) \big)$. 

\noindent
The remainder ${\cal R}_4(\vphi) = {\cal R}_4(\vphi; \omega)$, $\omega \in DC(\gamma, \tau)$ defined in \eqref{cal RM pre KAM} satisfies the estimates 
\begin{equation}\label{stime cal RM pre KAM}
| {\cal R}_4  |_s^\Lipg \lesssim_q \e
\end{equation}
where the block-decay norm $| \cdot |_s^\Lipg$ is defined in \eqref{decadimento Kirchoff}-\eqref{norma decadimento operatore matriciale}. 
\end{lemma}
\begin{proof}
By the choices of the constants in \eqref{scelta di M nuove norme}, one has that if $s_0 \leq s \leq [q / 2] - \overline \mu$, then 
$$
s + \frac{d - 1}{2} \leq M \qquad \text{and} \qquad  s_0 \leq s \leq q - M- \sigma - 2 \tau - 1\,.
$$
The estimates \eqref{stime cal TM} follow by Lemma \ref{composizione fourier multiplier} and by the estimates \eqref{stime tilde cal V M}, \eqref{stime cal E}. The fact that ${\cal T}^{\pm 1} \in {\cal C}^1 \big( \T^\nu, {\cal B}({\bf H}^s_0(\T^d)) \big)$ follows by applying Lemma \ref{stima Hs x fourier multiplier}. 

\noindent
Now we prove the estimate \eqref{stime cal RM pre KAM}. We estimate separately the two terms in \eqref{cal RM pre KAM}. 

\noindent
{\sc Estimate of $Q_{M, 4}$.}  By Lemma \ref{lemma norma decadimento norma pseudo diff} one gets 
$$
|Q_{M, 4} |_s^\Lipg  \lesssim \norma Q_{M, 4} \norma_{- s - \frac{d - 1}{2} , s}^\Lipg\,.
$$
hence we can apply the estimate \eqref{stime QM 4}, obtaining that $\norma Q_{M, 4} \norma_{- s - \frac{d - 1}{2} , s}^\Lipg \lesssim \norma Q_{M, 4} \norma_{-M , s}^\Lipg \lesssim_{M, q} \e \lesssim_q \e$, since the constant $M = M(q) = [q/2]$.  

\noindent
{\sc Estimate of ${\cal T}^{- 1} {\cal R}_3 {\cal T}$.} Recalling the definition of ${\cal R}_3$ given in \eqref{A0 bf R0} and using that the operator ${\cal R}^{(3)}$ has the form \eqref{forma buona resto cal R (3) (vartheta)}, defining 
$$
B_{1, k} := (\ii b_k^{(3)}, - \ii b_k^{(3)})\,, \quad B_{2, k} := (b_k^{(3)}, b_k^{(3)})\,, \quad C_{1, k} := (\ii c_k^{(3)}, -\ii c_k^{(3)})\,, \quad C_{2, k} := (c_k^{(3)}, c_k^{(3)})\,, \quad k = 1, \ldots, N
$$ we have that for ${\bf u} = (u, \overline u) \in {\bf L}^2_0(\T^d)$,  
$$
{\cal R}_3 [{\bf u}] = \sum_{k = 1}^N B_{1, k} \langle C_{2, k}\,,\, {\bf u} \rangle_{{\bf L}^2_x} + C_{1, k} \langle B_{2, k}\,,\, {\bf u} \rangle_{{\bf L}^2_x} 
$$
where we recall that the bilinear form $\langle \cdot \,,\, \cdot \rangle_{{\bf L}^2_x}$ is defined in \eqref{prodotto scalare u bar u}. Thus 
$$
({\cal T}^{- 1} {\cal R}_3 {\cal T} )[{\bf u}] = \sum_{k = 1}^N \widetilde B_{1, k} \langle \widetilde C_{2, k}\,,\, {\bf u} \rangle_{{\bf L}^2_x} + \widetilde C_{1, k} \langle \widetilde B_{2, k}\,,\, {\bf u} \rangle_{{\bf L}^2_x}\,, 
$$
$$
\widetilde B_{1, k} := {\cal T}^{- 1} B_{1, k}\,, \quad \widetilde B_{2, k} := {\cal T}^T B_{2, k}\,, \quad \widetilde C_{1, k} := {\cal T}^{- 1} C_{1, k}\,, \quad \widetilde C_{2, k} := {\cal T}^T C_{2, k}\,, \quad k = 1, \ldots, N\,. 
$$
The operator $\e {\cal T}^{- 1} {\cal R}_3 {\cal T}$ satisfies the claimed inequality, by applying the estimates \eqref{stime cal R3 1}, \eqref{stime cal TM} and Lemmata \ref{azione fourier multiplier}, \ref{decadimento operatori di proiezione}. 
\end{proof}

\section{Block-diagonal reducibility}\label{sec:redu}
In this section we carry out the second part of the reduction of ${\cal L}(\vphi)$ to a block-diagonal operator with constant coefficients. Our goal is to block-diagonalize the linear Hamiltonian vector field ${\cal L}_4(\vphi) $  obtained in \eqref{cal L5}. We are going to perform an iterative Nash-Moser reducibility scheme for the linear Hamiltonian vector field  
\be\label{L0}
{\cal L}_{0}(\vphi) := {\cal L}_4(\vphi) = {\cal D}_{0} + {\cal R}_{0}(\vphi) \, , 
\ee
 where 
 \begin{equation}\label{cal D0 riducibilita}
 {\cal D}_0 = \ii \begin{pmatrix}
 - {\cal D}_0^{(1)} & 0 \\
 0 &  {\cal D}_0^{(1)}
 \end{pmatrix}\,, \quad {\cal D}_0^{(1)} := D_M = {\rm diag}_{j \in \Z^d \setminus \{ 0 \}} \big( m |j| + c(|j|) \big)
 \end{equation}
 (see \eqref{operatore diagonal pre KAM}) and ${\cal R}_0(\vphi) := {\cal R}_4(\vphi)$, $\vphi \in \T^\nu$, is a Hamiltonian vector field of the form 
  \begin{equation}\label{cal R0 riducibilita}
 {\cal R}_0(\vphi) = \ii \begin{pmatrix}
 {\cal R}_0^{(1)}(\vphi) & {\cal R}_0^{(2)}(\vphi) \\
 - \overline{\cal R}_0^{(2)}(\vphi) & - \overline{\cal R}_0^{(1)}(\vphi)
 \end{pmatrix}\,, \quad {\cal R}_0^{(1)}(\vphi) =  {\cal R}_0^{(1)}(\vphi)^*\,, \quad {\cal R}_0^{(2)}(\vphi) =  {\cal R}_0^{(2)}(\vphi)^T
 \end{equation}
satisfying, by \eqref{stime cal RM pre KAM}, the estimate
 \begin{equation}\label{stime primo resto KAM}
 |{\cal R}_0|_s^\Lipg \lesssim_q \e\,, \qquad \forall s_0 \leq s \leq [q/2 ]- \overline \mu
 \end{equation}
 where the constant $\overline \mu$ is defined in  \eqref{scelta di M nuove norme}. Note that, according to the block representation \eqref{notazione a blocchi}, the operator ${\cal D}_0^{(1)}$ can be written as 
 \begin{equation}\label{rappresentazione a blocchi cal D0 (1)}
 {\cal D}_0^{(1)} = {\rm diag}_{\alpha \in \sigma_0(\sqrt{- \Delta})} \mu_\alpha^{0}{\mathbb I}_\alpha\,, \qquad \mu_\alpha^0 := m\,  \alpha + c(\alpha)\,, \quad \forall \alpha \in \sigma_0(\sqrt{- \Delta})
 \end{equation}
   where ${\mathbb I}_\alpha : {\mathbb E}_\alpha \to {\mathbb E}_\alpha$ is the identity (recall \eqref{bf E alpha}, \eqref{notazione operatore diagonale a blocchi}) and the real constants $m$ and $c(\alpha)$ satisfy the estimates \eqref{stime m}, \eqref{stime c(alpha)}. We define
\be\label{defN}
N_{-1} := 1 \, ,  \quad 
N_{k} := N_{0}^{\chi^{k}} \  \forall k \geq 0 \, ,  \quad  
\chi := 3 /2  
\ee
(then $ N_{k +1} = N_{k}^\chi $, $ \forall k \geq 0 $) and for $\tau, \mathtt d > 0$, we define the constants
\begin{equation}\label{definizione alpha}
 \mathfrak s_0 := 2 s_0, \quad \mathtt a := 4 \tau + 8 \mathtt d + 3\,, \quad \mathtt b := \mathtt a + 1\,, \quad  S_q := [q/2] - \overline \mu - \mathtt b\,, \quad \text{with} \quad q >2( \mathfrak s_0 + \overline \mu + \mathtt b )\,.
\end{equation}
In order to state the theorem below, we recall the definition of the space ${\mathcal S}({\mathbb E}_\alpha), \alpha \in \sigma_0(\sqrt{- \Delta})$ given in \eqref{cal S E alpha}, the definition of the norm $\| \cdot \|_{{\rm Op}(\alpha,\beta)}, \alpha, \beta \in \sigma_0(\sqrt{- \Delta})$ given in \eqref{norma operatoriale su matrici alpha beta}, the identity ${\mathbb I}_{\alpha, \beta}, \alpha, \beta \in \sigma_0(\sqrt{- \Delta})$ in \eqref{operatore identita matrici alpha beta}, the definition of $M_L(A)$ in \eqref{definizione moltiplicazione sinistra matrici} and the definition of $M_R(B)$ in \eqref{definizione moltiplicazione destra matrici}. 
\begin{theorem}{{\bf (KAM reducibility)}} \label{thm:abstract linear reducibility}  
Let $\gamma \in (0, 1)$, $\tau, \mathtt d > 0$ and let $q$ satisfy \eqref{definizione alpha}.
There exist, $ N_{0 }  = N_0(q, \tau, \mathtt d,  \nu, d)  \in \N $ large enough, $\delta_q = \delta(q, \tau, \mathtt d,  \nu, d) \in (0, 1)$ (possibly smaller than the one appearing in Lemma \ref{ultimo lemma pre riducibilita})  such that, if 
\begin{equation}\label{piccolezza1}
\e \gamma^{- 1} \leq \delta_q 
\end{equation}
then, for all $ k \geq 0 $:

\begin{itemize}
\item[${\bf(S1)}_k$] 
There exists a Hamiltonian vector field  
\be\label{def:Lj}
{\cal L}_k (\vphi) :={\cal D}_k + {\cal R}_k (\vphi)\,, \qquad \vphi \in \T^\nu\,,
\ee
\begin{equation}\label{cal D nu}
{\cal D}_k = \ii \begin{pmatrix}
 - {\cal D}^{(1)}_k  & 0 \\
0 & \overline{{\cal D}^{(1)}_k}
\end{pmatrix}\,,\qquad {\cal D}_k^{(1)} := {\rm diag}_{\alpha \in \sigma_0(\sqrt{- \Delta}) } [{\cal D}_k^{(1)}]_\alpha^\alpha\,, \quad [{\cal D}_k^{(1)}]_\alpha^\alpha \in {\cal S}({\mathbb E}_\alpha)\,, \quad \forall \alpha \in \sigma_0(\sqrt{- \Delta})\,,
\end{equation}
defined for all $ \omega \in \Omega_{k}^{\g}$, where 
$\Omega_{0}^{\g} :=  DC(\gamma, \tau)$ (see \eqref{diofantei Kn}) and 
 for $k \geq 1$, 
\begin{align}
\Omega_{k}^{\g} & :=  \Big\{\omega \in \Omega_{k - 1}^{\gamma} : \| {\mathbb A}_{k - 1}^{-}(\ell, \alpha, \beta)^{- 1} \|_{{\rm Op}(\alpha, \beta)} \leq \frac{\alpha^{\mathtt d} \beta^{\mathtt d}\langle \ell \rangle^\tau}{\gamma }\,, \quad \forall (\ell, \alpha, \beta) \in \Z^\nu \times \sigma_0(\sqrt{- \Delta}) \times  \sigma_0(\sqrt{- \Delta})\,, \nonumber\\
&  (\ell, \alpha, \beta) \neq (0, \alpha, \alpha)\,, \quad \langle \ell, \alpha, \beta \rangle \leq N_{k - 1}  \quad \text{and} \quad  \| {\mathbb A}_{k - 1}^+(\ell, \alpha, \beta)^{- 1} \|_{{\rm Op}(\alpha, \beta)} \leq \frac{\langle \ell \rangle^\tau}{\gamma \langle \alpha + \beta \rangle}\,, \nonumber\\
&  \forall (\ell, \alpha, \beta) \in \Z^\nu \times \sigma_0(\sqrt{- \Delta}) \times \sigma_0(\sqrt{- \Delta})\,, \quad \langle\ell, \alpha , \beta \rangle \leq N_{k - 1} \Big\}\,. \label{Omgj}
\end{align}
The operators ${\mathbb A}^{\pm}_{k - 1}(\ell, \alpha, \beta) : {\cal B}({\mathbb E}_\beta, {\mathbb E}_\alpha) \to  {\cal B}({\mathbb E}_\beta, {\mathbb E}_\alpha)$ are defined by
\begin{equation}\label{bf A nu - (ell,alpha,beta)}
{\mathbb A}^{-}_{k - 1}(\ell, \alpha, \beta) := \omega \cdot \ell {\mathbb I}_{\alpha, \beta} + M_L([{\cal D}_{k - 1}^{(1)}]_\alpha^\alpha) - M_R([{\cal D}_{k - 1}^{(1)}]_\beta^\beta)\,,
\end{equation}
\begin{equation}\label{bf A nu + (ell,alpha,beta)}
{\mathbb A}^{+}_{k - 1}(\ell,  \alpha, \beta) := \omega \cdot \ell {\mathbb I}_{\alpha, \beta} + M_L([{\cal D}_{k - 1}^{(1)}]_\alpha^\alpha) + M_R([\overline{{\cal D}_{k - 1}^{(1)}}]_\beta^\beta)\,.
\end{equation}
For $k \geq 0$, for all $\alpha \in \sigma_0(\sqrt{- \Delta})$, the self-adjoint operator $[{\cal D}_k^{(1)}]_\alpha^\alpha \in {\cal S}({\mathbb E}_\alpha)$ satisfies 
\begin{equation}  \label{rjnu bounded}
 \|[{\cal D}_k^{(1)} - {\cal D}_0^{(1)}]_\alpha^\alpha \|^\Lipg_{HS}  \lesssim_{q} \e  \alpha^{- S_q} \, \quad \forall \alpha \in \sigma_0(\sqrt{- \Delta})\,. 
\end{equation}
The remainder $  {\cal R}_k  $ is Hamiltonian and $ \forall s \in [ \frak s_0,  S_q ] $,  
\begin{align}
\left| {\cal R}_{k}  \right|_{s}^\Lipg   \leq  \left|{\cal R}_{0} \right|_{s+\mathtt b}^\Lipg N_{k - 1}^{-\mathtt a} \, , \quad \left| {\cal R}_{k} \right|_{s + \mathtt b}^\Lipg  \leq \left| {\cal R}_{0} \right|_{s+ \mathtt b}^\Lipg \,N_{k - 1}\,. \label{Rsb}  
\end{align}
Moreover, for $k \geq 1 $, 
\be	\label{Lnu+1}
{\cal L}_{k} (\vphi)= (\Phi_k)_{\omega*} {\cal L}_{k - 1}(\vphi)  \, , \quad \Phi_{k -1} := {\rm exp}(\Psi_{k  - 1}) \, 
\ee
where the map $ \Psi_{k-1} $  is a Hamiltonian vector field and satisfies 
\be\label{Psinus}  
 |\Psi_{k - 1}|_s^\Lipg \leq 
|{\cal R}_{0}  |_{s+\mathtt b}^\Lipg \gamma^{-1} N_{k -1}^{2 \t + 4 \mathtt d+1} N_{k - 2}^{- \mathtt a}  \, . 
\ee

\item[${\bf(S2)}_{k}$] 
For all $ \alpha \in \sigma_0(\sqrt{- \Delta})$,  there exists a Lipschitz extension to the set $DC(\gamma, \tau)$, that we denote by  
$ [\widetilde{\cal D}_k^{(1)}]_\alpha^\alpha(\cdot ): DC(\gamma, \tau) \to {\cal S}({\mathbb E}_\alpha) $ of  $ [{\cal D}_k^{(1)}]_\alpha^\alpha(\cdot ) : \Omega_k^\g \to {\cal S}({\mathbb E}_\alpha)$ 
satisfying,  for $k \geq 1$, 
\be\label{lambdaestesi}  
\|[\widetilde{\cal D}_k^{(1)}]_\alpha^\alpha -  [\widetilde{\cal D}_{k - 1}^{(1)}]_\alpha^\alpha \|^\Lipg_{HS}  \lesssim  \alpha^{-  S_q} | {\cal R}_{k -1}  |^\Lipg_{S_q}  \lesssim  N_{k - 2}^{- \mathtt a} \alpha^{- S_q} | {\cal R}_{0}  |^\Lipg_{S_q + \mathtt b} \,.
\ee
\end{itemize}
\end{theorem}
\begin{remark}
The constants $\tau, \mathtt d > 0$ in \eqref{Omgj} will be fixed in the formula \eqref{scelta tau mathtt d}, in Section \ref{sezione stime in misura}, in order to prove the measure estimate of the set $\Omega_\infty^{2 \gamma}$ defined in \eqref{Omegainfty} (see Theorem \ref{teorema finale stima in misura}).
\end{remark}
\subsection{Proof of Theorem \ref{thm:abstract linear reducibility}} \label{subsec:proof of thm abstract linear reducibility}

\noindent
{\sc Proof  of ${\bf ({S}i)}_{0}$, $i=1, 2$.} 
Properties \eqref{def:Lj}-\eqref{Rsb} in ${\bf({S}1)}_0$
hold by \eqref{L0}-\eqref{stime primo resto KAM} with $ [{\cal D}_0^{(1)}]_\alpha^\alpha $ given in \eqref{rappresentazione a blocchi cal D0 (1)}  (for \eqref{Rsb} recall  that $ N_{-1} := 1 $, see \eqref{defN}). Moreover, since the constants $m$ and $c(\alpha) = c(\alpha; \omega)$ are real, $ [{\cal D}_0^{(1)}]_\alpha^\alpha $ is self-adjoint, then there is nothing else to verify.  

${\bf({S}2)}_0 $ holds, since the constant $m$ is independent of $\omega$ and $c(\alpha) = c(\alpha; \omega)$, $\alpha \in \sigma_0(\sqrt{- \Delta})$, is already defined for all $\omega \in DC(\gamma, \tau)$.  
\subsection{The reducibility step}
We now describe the inductive step, showing how to define a symplectic transformation 
$ \Phi_k := \exp(  \Psi_k ) $ 
so that the transformed vector field $ {\cal L}_{k +1 }(\vphi) = (\Phi_k)_{\omega*} {\cal L}_k(\vphi) $ has the desired properties.
To simplify notations, in this section we drop the index $ k $ and we write $ + $ instead of $ k + 1$.
At each step of the iteration we have a Hamiltonian vector field 
\begin{equation}\label{operator KAM step}
{\cal L}(\vphi) =  {\cal D} + {\cal R}(\vphi)\,,
\end{equation}
where 
\begin{equation}\label{diagonal operator KAM step}
{\cal D} := 
\ii \begin{pmatrix}
- {\cal D}^{(1)} & 0 \\
0 &   \overline{\cal D}^{(1)}
\end{pmatrix}\,, \qquad {\cal D}^{(1)} := {\rm diag}_{\alpha \in \sigma_0(\sqrt{- \Delta})} [{\cal D}^{(1)}]_\alpha^\alpha \,,\quad [{\cal D}^{(1)}]_\alpha^\alpha \in {\cal S}({\mathbb E}_\alpha) \quad\forall \alpha \in \sigma_0(\sqrt{- \Delta})
\end{equation}
and ${\cal R}(\vphi)$ is a Hamiltonian vector field, namely it has the form 
\begin{equation}\label{cal R Hamiltoniano}
{\cal R} = \ii \begin{pmatrix}
{\cal R}^{(1)} & {\cal R}^{(2)} \\
- \overline{\cal R}^{(2)} & - \overline{\cal R}^{(1)}
\end{pmatrix}\,, \qquad {\cal R}^{(1)}(\vphi) = {\cal R}^{(1)}(\vphi)^*\,, \qquad {\cal R}^{(2)}(\vphi) = {\cal R}^{(2)}(\vphi)^T\,, \quad \forall \vphi \in \T^\nu\,.
\end{equation}
Let us consider a transformation 
\begin{equation}\label{Phi Psi}
\Phi (\vphi) := {\rm exp}(\Psi(\vphi))\,,\qquad \Psi(\vphi) :=  \ii \begin{pmatrix}
\Psi^{(1)}(\vphi) & \Psi^{(2)}(\vphi) \\
- \overline\Psi^{(2)}(\vphi) & - \overline\Psi^{(1)}(\vphi)
\end{pmatrix}\,, \quad \vphi \in \T^\nu
\end{equation}
with $\Psi^{(1)}(\vphi) = \Psi^{(1)}(\vphi)^*$, $\Psi^{(2)}(\vphi) = \Psi^{(2)}(\vphi)^T$, for all $\vphi \in \T^\nu$. Writing 

\begin{equation}\label{espansione Phi}
\Phi = {\rm Id} + \Psi + \Psi_{\geq 2}\,,\quad \Psi_{\geq 2} := \sum_{k \geq 2} \frac{\Psi^k}{k !}\,.
\end{equation}
By \eqref{interpretazione dinamica 6} we have $\Phi_{\omega*} {\cal L}(\vphi) = \Phi(\vphi)^{- 1} \Big( {\cal L}(\vphi) \Phi(\vphi) -  \omega \cdot \partial_\vphi \Phi(\vphi) \Big)\,.$ By the expansion \eqref{espansione Phi}, recalling the definition of the projector operator $\Pi_N {\cal R}$ given in \eqref{SM-block-matrix}, one gets that 
\begin{eqnarray}
{\cal L}(\vphi) \Phi(\vphi) - \omega \cdot \partial_\vphi \Phi(\vphi) & = & \Phi(\vphi)  {\cal D} + \Big(- \omega \cdot \partial_\vphi \Psi + [{\cal D}, \Psi(\vphi)] + \Pi_N {\cal R}(\vphi) \Big) + \Pi_N^\bot {\cal R}(\vphi) \nonumber\\
& & - \omega \cdot \partial_\vphi \Psi_{\geq 2}(\vphi) + [{\cal D}, \Psi_{\geq 2}(\vphi)] + {\cal R}(\vphi) (\Phi(\vphi) - {\rm Id})\,, \label{quasi coniugio riducibilita}
\end{eqnarray}
We want to determine the operator $\Psi(\vphi)$ so that 
\begin{equation}\label{Homological equation}
- \omega \cdot \partial_\vphi \Psi(\vphi) + [{\cal D}, \Psi(\vphi)] + \Pi_N {\cal R}(\vphi) = \Pi_N {\cal R}_{diag}\,, 
\end{equation}
where recalling the definitions \eqref{media diagonale operatore cal R}, \eqref{SM-block-matrix}
\begin{equation}\label{media di cal R}
\Pi_N {\cal R}_{diag}:= \ii \begin{pmatrix}
\Pi_N {\cal R}^{(1)}_{diag} & 0 \\
0 &- \Pi_N \overline{\cal R}^{(1)}_{diag}
\end{pmatrix}\,.
\end{equation}
\begin{lemma} {\bf (Homological equation)}\label{homologica equation} 
For all $ \omega \in \Omega_{k + 1}^{{\gamma}}$ (see \eqref{Omgj}),  there exists a solution
$ \Psi  $ of the homological equation \eqref{Homological equation}, which is Hamiltonian and 
satisfies 
\be\label{PsiR}
|\Psi  |_ s^\Lipg \lesssim   N^{2 \tau + 4 \mathtt d +  1} \gamma^{-1} |{\cal R}  |_s^\Lipg \, . 
\ee
\end{lemma} 
\begin{proof}
Recalling \eqref{cal R Hamiltoniano}, \eqref{Phi Psi}, The equation \eqref{Homological equation} is split in the two equations 
\begin{equation}\label{prima equazione omologica}
- \ii \omega \cdot \partial_\vphi \Psi^{(1)}(\vphi) + [ {\cal D}^{(1)}, \Psi^{(1)}(\vphi)] + \ii \Pi_N {\cal R}^{(1)}(\vphi) = \ii \Pi_N{\cal R}^{(1)}_{diag}\,,
\end{equation}
\begin{equation}\label{seconda equazione omologica}
- \ii \omega \cdot \partial_\vphi \Psi^{(2)}(\vphi) + ({\cal D}^{(1)} \Psi^{(2)}(\vphi) + \Psi^{(2)}(\vphi) \overline{\cal D}^{(1)}) + \ii \Pi_N {\cal R}^{(2)}(\vphi) = 0\,.
\end{equation}
Using the decomposition \eqref{notazione a blocchi} and recalling \eqref{notazione a blocchi vphi}, the equations \eqref{prima equazione omologica}, \eqref{seconda equazione omologica} become 
for any $\alpha, \beta \in \sigma_0(\sqrt{- \Delta})$, $\ell \in \Z^\nu$
\begin{equation}\label{equazione omologica blocco psi diagonale}
 \omega \cdot \ell [\widehat{ \Psi}^{(1)}(\ell)]_\alpha^\beta + [{\cal D}^{(1)}]_\alpha^\alpha [\widehat{\Psi}^{(1)}(\ell)]_\alpha^{\beta} - [\widehat{ \Psi}^{(1)}(\ell)]_\alpha^{\beta} [{\cal D}^{(1)}]_\beta^\beta    =- \ii [\widehat{\Pi_N{\cal R}}^{(1)}(\ell)]_\alpha^{\beta} + \ii [\widehat{\Pi_N {\cal R}}^{(1)}_{diag}(\ell)]_\alpha^{\beta}
\end{equation}
\begin{equation}\label{equazione omologica blocco psi fuori diagonale}
 \omega \cdot \ell [\widehat{\Psi}^{(2)}(\ell)]_\alpha^{\beta} + [{\cal D}^{(1)}]_\alpha^\alpha[\widehat{\Psi}^{(2)} (\ell)]_\alpha^{\beta} + [\widehat{ \Psi}^{(2)}(\ell) ]_\alpha^{\beta} [\overline{\cal D}^{(1)}]_\beta^\beta=  - \ii [\widehat{\Pi_N{\cal R}}^{(2)}(\ell)]_\alpha^{\beta}\,.
\end{equation}
By the Definitions \eqref{bf A nu - (ell,alpha,beta)}, \eqref{bf A nu + (ell,alpha,beta)}, namely setting 
\begin{equation}\label{definition bf A pm ell alpha beta}
{\mathbb A}^{-}(\ell, \alpha, \beta) := \omega \cdot \ell {\mathbb I}_{\alpha, \beta} + M_L([{\cal D}^{(1)}]_\alpha^\alpha) - M_R( [{\cal D}^{(1)}]_\beta^\beta)\,, \,\,{\mathbb A}^{+}(\ell, \alpha, \beta) := \omega \cdot \ell {\mathbb I}_{\alpha, \beta} + M_L([{\cal D}^{(1)}]_\alpha^\alpha) + M_R( [\overline{\cal D}^{(1)}]_\beta^\beta)
\end{equation}
the equations \eqref{equazione omologica blocco psi diagonale}, \eqref{equazione omologica blocco psi fuori diagonale} can be written in the form 
$$
{\mathbb A}^{-}(\ell, \alpha, \beta) [\widehat{ \Psi}^{(1)}(\ell) ]_\alpha^{\beta}  =- \ii [\widehat{\Pi_N{\cal R}}^{(1)}(\ell)]_\alpha^{\beta} + \ii [\widehat{\Pi_N {\cal R}}^{(1)}_{diag}(\ell)]_\alpha^{\beta}\,, \quad {\mathbb A}^{+}(\ell,\alpha, \beta) [\widehat{\Psi}^{(2)}(\ell)]_\alpha^{\beta} = - \ii [\widehat{\Pi_N{\cal R}}^{(2)}(\ell)]_\alpha^{\beta}\,.
$$ 
Then, since $\omega \in \Omega_{k + 1}^{\gamma}$, recalling the Definition \eqref{SM-block-matrix}, we can define for any $(\ell, \alpha, \beta) \in \Z^\nu \times \sigma_0(\sqrt{- \Delta}) \times \sigma_0(\sqrt{- \Delta})$ 
\begin{equation}\label{soluzione equazione omologica 1}
[\widehat{ \Psi}^{(1)}(\ell)]_\alpha^{\beta} := \begin{cases} 
\ii {\mathbb A}^{-}(\ell, \alpha, \beta)^{- 1} [\widehat{\cal R}^{(1)}(\ell)]_\alpha^{\beta} & \quad \text{if} \quad  (\ell, \alpha, \beta) \neq (0, \alpha, \alpha)\,, \quad \langle \ell, \alpha,  \beta \rangle \leq N \\
0 & \quad \text{otherwise}
\end{cases}
\end{equation}
\begin{equation}\label{soluzione equazione omologica 2}
[\widehat{ \Psi}^{(2)}(\ell)]_\alpha^{\beta} := \begin{cases} 
\ii {\mathbb A}^{+}(\ell, \alpha, \beta)^{- 1} [\widehat{\cal R}^{(2)}(\ell)]_\alpha^{\beta} & \quad \text{if} \quad  \langle \ell, \alpha, \beta \rangle \leq N  \\
0 & \quad \text{otherwise.}
\end{cases}
\end{equation}
We have 
$$
\|{\mathbb A}^{-}(\ell, \alpha, \beta)^{- 1} \|_{{\rm Op}(\alpha, \beta)} \leq \frac{\alpha^{\mathtt d} \beta^{\mathtt d}\langle \ell \rangle^\tau}{\gamma }\,, \quad \|{\mathbb A}^{+}(\ell, \alpha, \beta)^{- 1} \|_{{\rm Op}(\alpha, \beta)} \leq \frac{\langle \ell \rangle^\tau}{\gamma ( \alpha +\beta )}\,
$$
and since $[\widehat{ \Psi}^{(1)}(\ell)]_\alpha^{\beta}$, $[\widehat{ \Psi}^{(2)}(\ell)]_\alpha^{\beta}$ are non zero only if $\langle \ell , \alpha,  \beta \rangle \leq N$, we get immediately that 
\begin{equation}\label{stima norma sup coefficienti Psi}
\| [\widehat{ \Psi}^{(1)}(\ell)]_\alpha^{\beta}\|_{HS} \leq N^{\tau + 2 \mathtt d} \gamma^{- 1}  \| [\widehat{\cal R}^{(1)}(\ell)]_\alpha^{\beta}\|_{HS}\,, \quad \| [\widehat{ \Psi}^{(2)}(\ell)]_\alpha^{\beta}\|_{HS} \leq N^\tau \gamma^{- 1} \| [\widehat{\cal R}^{(2)}(\ell)]_\alpha^{\beta}\|_{HS}\,.
\end{equation}
Hence, recalling the definition \eqref{decadimento Kirchoff} of the block-decay norm, one gets that 
\begin{equation}\label{stima norma sup Psi cal R}
| \Psi^{(1)} |_s \lesssim N^{\tau + 2 \mathtt d} \gamma^{- 1} | {\cal R }^{(1)}|_s\,, \quad  | \Psi^{(2)} |_s \lesssim N^{\tau} \gamma^{- 1} | {\cal R }^{(2)}|_s
\end{equation}
Now, let $\omega_1, \omega_2  \in \Omega_{k + 1}^{\gamma}$. As a notation for any function $f = f(\omega)$ depending on the parameter $\omega$, we write $\Delta_{\omega} f := f(\omega_1) - f(\omega_2)\,.$ By \eqref{soluzione equazione omologica 1}, one has 
\begin{equation}\label{Delta omega 12 coefficienti Psi}
\Delta_{\omega} [\widehat{ \Psi}^{(1)}(\ell) ]_\alpha^{\beta}= \ii  \Delta_\omega {\mathbb A}^{-}(\ell, \alpha, \beta)^{- 1}   [\widehat{\cal R}^{(1)}(\ell; \omega_1)]_\alpha^{\beta} +  \ii {\mathbb A}^{-}(\ell, \alpha, \beta; \omega_2)^{- 1}  \Delta_\omega [\widehat{\cal R}^{(1)}(\ell ) ]_\alpha^{\beta} \,.
\end{equation}
 As in \eqref{stima norma sup coefficienti Psi}, one gets   
\begin{equation}\label{stima primo pezzo Delta omega 12 coefficienti Psi}
\| {\mathbb A}^{-}(\ell, \alpha, \beta; \omega_2)^{- 1}   \Delta_\omega [\widehat{\cal R}^{(1)} (\ell) ]_\alpha^{\beta}  \|_{HS} \lesssim N^{\tau + 2 \mathtt d} \gamma^{- 1}  \| \Delta_\omega [\widehat{\cal R}^{(1)}(\ell)  ]_\alpha^{\beta} \|_{HS}\,,
\end{equation}
hence it remains to estimate only the first term in \eqref{Delta omega 12 coefficienti Psi}. We have 
\begin{equation}\label{bernardino 0}
\Delta_\omega {\mathbb A}^{-}(\ell, \alpha, \beta)^{- 1} = - {\mathbb A}^{-}(\ell, \alpha, \beta ; \omega_1)^{- 1}\Big( \Delta_\omega {\mathbb A}^{-}(\ell, \alpha, \beta )  \Big) {\mathbb A}^{-}(\ell, \alpha, \beta ; \omega_2)^{- 1}\,,
\end{equation}
Therefore 
\begin{equation}\label{norma A - omega 12 inverso}
\|\Delta_\omega  {\mathbb A}^{-}(\ell, \alpha, \beta)^{- 1} \|_{{\rm Op}(\alpha, \beta)}  \leq \frac{N^{2 \tau} \alpha^{2 \mathtt d} \beta^{2 \mathtt d} }{\gamma^2 } \| \Delta_\omega {\mathbb A}^{-}(\ell, \alpha, \beta)  \|_{{\rm Op}(\alpha, \beta)} \,.
\end{equation}
Moreover 
\begin{align}
\Delta_\omega {\mathbb A}^{- }(\ell, \alpha, \beta) & = (\omega_1 - \omega_2) \cdot \ell \,\,{\mathbb I}_{\alpha, \beta} + M_L(\Delta_\omega [{\cal D}^{(1)}]_\alpha^\alpha) - M_R(\Delta_\omega [{\cal D}^{(1)}]_\beta^\beta) \label{Delta omega 12 bf A -}
\end{align}
and using that, by \eqref{rappresentazione a blocchi cal D0 (1)}, \eqref{rjnu bounded}
\begin{equation}\label{stima bf widehat D alpha (omega)}
[{\cal D}^{(1)}(\omega)]_\alpha^\alpha =  \mu_\alpha^0(\omega) {\mathbb I}_\alpha +[{\cal D}^{(1)} - {\cal D}_0^{(1)}]_\alpha^\alpha\,, \quad \text{with} \quad \| [{\cal D}^{(1)} - {\cal D}_0^{(1)}]_\alpha^\alpha \|^\Lipg_{HS} \lesssim_q  \e \alpha^{- S_q}\,, \quad \forall  \alpha \in \sigma_0(\sqrt{- \Delta})\,,
\end{equation}
we get 
\begin{align}
 M_L(\Delta_\omega [{\cal D}^{(1)}]_\alpha^\alpha) - M_R(\Delta_\omega [{\cal D}^{(1)}]_\beta^\beta) & = \Delta_\omega \big(  \mu_\alpha^0 - \mu_\beta^0  \big) {\mathbb I}_{\alpha, \beta} + M_L(\Delta_\omega [{\cal D}^{(1)} - {\cal D}_0^{(1)}]_\alpha^\alpha) - M_R(\Delta_\omega [{\cal D}^{(1)} - {\cal D}_0^{(1)}]_\beta^\beta) \,. \nonumber
\end{align}
Using that the constant $m$ is independent of $\omega$, i.e. $\Delta_\omega m = 0$ and by recalling \eqref{rappresentazione a blocchi cal D0 (1)}, \eqref{stime c(alpha)}, one gets 
\begin{align}
|\Delta_\omega (\mu_\alpha^0 - \mu_\beta^0)| & \lesssim  |\Delta_\omega c(\alpha) | + |\Delta_\omega c(\beta)|  \lesssim  \sup_{\alpha \in \sigma_0(\sqrt{- \Delta})} |c(\alpha)|^{\rm lip}  |\omega_1 - \omega_2|   \lesssim \gamma^{- 1}  \sup_{\alpha \in \sigma_0(\sqrt{- \Delta})} |c(\alpha)|^\Lipg  |\omega_1 - \omega_2| \nonumber\\
& \lesssim_q \e \gamma^{- 1}  |\omega_1 - \omega_2| \,. \label{mala femmina}
\end{align}
By \eqref{stima bf widehat D alpha (omega)}, \eqref{mala femmina} and using the property \eqref{norma operatoriale ML MR} one gets 
\begin{align}
\| - M_L(\Delta_\omega [{\cal D}^{(1)}]_\alpha^\alpha) + M_R(\Delta_\omega [{\cal D}^{(1)}]_\beta^\beta) \|_{{\rm Op}(\alpha, \beta)} & {\lesssim}  |\Delta_\omega (\mu_0^\alpha - \mu_0^\beta)|\|{\mathbb I}_{\alpha, \beta} \|_{{\rm Op}(\alpha, \beta)} \nonumber\\
& \quad + \| M_R(\Delta_\omega [{\cal D}^{(1)} - {\cal D}_0^{(1)}]_\beta^\beta) - M_L(\Delta_\omega [{\cal D}^{(1)} - {\cal D}_0^{(1)}]_\alpha^\alpha) \|_{{\rm Op}(\alpha, \beta)}  \nonumber\\
& \lesssim_q  \e \gamma^{- 1}  |\omega_1 - \omega_2 |\,. \label{ML MR Delta omega 12}
\end{align}
 Recalling \eqref{Delta omega 12 bf A -}, we get the estimate 
$$
\| \Delta_\omega {\mathbb A}^{- }(\ell, \alpha, \beta) \|_{{\rm Op}(\alpha, \beta)}  \leq \Big(C \langle \ell \rangle + C'(q)\e \gamma^{- 1} \Big) |\omega_1 - \omega_2|\,, 
$$
for some constants $C, C'(q) > 0$, hence, by \eqref{norma A - omega 12 inverso}, by taking $\delta_q$ in \eqref{piccolezza1} small enough (so that $C'(q)\e \gamma^{- 1} \leq 1$), one gets that for $\langle \ell, \alpha,  \beta \rangle \leq N$
$$
\| \Delta_\omega  {\mathbb A}^{-}(\ell, \alpha, \beta)^{- 1} \|_{{\rm Op}(\alpha, \beta)}  \lesssim N^{2 \tau + 4 \mathtt d+  1} \gamma^{- 2} |\omega_1 - \omega_2| \,.
$$
The above estimate implies that 
\begin{align}
&\| \big\{ \Delta_\omega {\mathbb A}^{-}(\ell, \alpha, \beta)^{- 1}  \big\} [\widehat{\cal R}^{(1)}(\ell; \omega_1) ]_\alpha^{\beta} \|_{HS}  \lesssim N^{2 \tau + 4 \mathtt d + 1}  \gamma^{- 2}  \| [\widehat{\cal R}^{(1)}(\ell; \omega_1) ]_\alpha^{\beta}\|_{HS} |\omega_1 - \omega_2|\,. \label{caffetteria 0}
\end{align}
By \eqref{Delta omega 12 coefficienti Psi}, \eqref{stima primo pezzo Delta omega 12 coefficienti Psi}, \eqref{caffetteria 0} we get the estimate 
\begin{equation}\label{stima Delta omega 12 coefficienti Psi finale}
\| \Delta_\omega  [\widehat{ \Psi}^{(1)}(\ell) ]_\alpha^{\beta}\|_{HS} \lesssim N^{\tau + 2 \mathtt d} \gamma^{- 1}  \| \Delta_\omega  [\widehat{\cal R}^{(1)}(\ell)]_\alpha^{\beta}  \|_{HS} + N^{2 \tau + 4 \mathtt d + 1} \gamma^{- 2}  \| [\widehat{\cal R}^{(1)}(\ell; \omega_1) ]_\alpha^{\beta} \|_{HS}\,.
\end{equation}
Thus \eqref{stima norma sup Psi cal R}, \eqref{stima Delta omega 12 coefficienti Psi finale} and the Definitions \eqref{decadimento Kirchoff}, \eqref{norma decadimento lipschitz} imply 
$$
|\Psi^{(1)}  |_s^\Lipg \lesssim  N^{2 \tau + 4 \mathtt d + 1} \gamma^{- 1} |{\cal R}^{(1)} |_s^\Lipg\,.
$$
The estimate of $\Psi^{(2)}$ in terms of ${\cal R}^{(2)}$ follows by similar arguments and then \eqref{PsiR} is proved. 
 \end{proof}

By \eqref{quasi coniugio riducibilita}, \eqref{Homological equation}, we get 
\begin{equation}\label{cal L+}
{\cal L}_+(\vphi) := \Phi_{\omega*} {\cal L}(\vphi) =  {\cal D}_+ + {\cal R}_+(\vphi)\,, \quad \vphi \in \T^\nu\,,
\end{equation}
\begin{equation}\label{cal R+} 
{\cal D}_+ := {\cal D} + \Pi_N {\cal R}_{diag}\,,\qquad {\cal R}_+ :=  (\Phi^{- 1} - {\rm Id}) \Pi_N {\cal R}_{diag} + \Phi^{- 1} \Big( \Pi_N^\bot {\cal R} - \omega \cdot \partial_\vphi \Psi_{\geq 2} + [{\cal D}, \Psi_{\geq 2}] + {\cal R}(\Phi - {\rm Id}) \Big)\,.
\end{equation}

\begin{lemma}[{\bf The new block-diagonal part}]\label{the new diagonal part}
The new block-diagonal part is given by
\begin{equation}\label{new diagonal part}
{\cal D}_+ := {\cal D} + \Pi_N {\cal R}_{diag} = \ii \begin{pmatrix}
 - {\cal D}_+^{(1)} & 0 \\
0 &   \overline{\cal D}_+^{(1)}
\end{pmatrix}\,,\qquad {\cal D}_+^{(1)} := {\cal D}^{(1)} + \Pi_N {\cal R}^{(1)}_{diag} = {\rm diag}_{\alpha \in \sigma_0(\sqrt{- \Delta}) } [{\cal D}_+^{(1)}]_\alpha^\alpha\,,
\end{equation}
where
\begin{equation}\label{mu j +}
[{\cal D}_+^{(1)}]_\alpha^\alpha := \begin{cases}
[{\cal D}]_\alpha^\alpha + [\widehat{\cal R}^{(1)}(0)]_\alpha^\alpha & \quad \text{if} \quad \alpha \leq N \\
[{\cal D}]_\alpha^\alpha & \quad \text{otherwise.}
\end{cases}  
\end{equation}
As a consequence 
\begin{equation}\label{mu + - mu}
 \| [{\cal D}_+^{(1)}]_\alpha^\alpha - [{\cal D}]_\alpha^\alpha \|^\Lipg_{HS} \lesssim  \alpha^{- S_q} |{\cal R}|_{S_q}^\Lipg\,, \quad \forall \alpha \in \sigma_0(\sqrt{- \Delta})
\end{equation}
\end{lemma}

\begin{proof}
Notice that, since ${\cal R}^{(1)}(\vphi)$ is selfadjoint, the operators $[\widehat{\cal R}^{(1)}(0) ]_\alpha^\alpha : {\mathbb E}_\alpha \to {\mathbb E}_\alpha$ are self-adjoint, i.e. $[\widehat{\cal R}^{(1)}(0) ]_\alpha^\alpha\in {\cal S}({\mathbb E}_\alpha)$, for any $\alpha \in \sigma_0(\sqrt{- \Delta})$ and using that $[{\cal D}^{(1)}]_\alpha^\alpha$ is self-adjoint, we get that $[{\cal D}_+^{(1)}]_\alpha^\alpha$ is self-adjoint for all $\alpha \in \sigma_0(\sqrt{- \Delta})$. The formula \eqref{mu j +} follows by \eqref{new diagonal part} and recalling the definitions \eqref{media diagonale operatore cal R}, \eqref{SM-block-matrix}. The estimate \eqref{mu + - mu} follows by 
\begin{align}
\sup_{\alpha \in \sigma_0(\sqrt{- \Delta})} \alpha^{S_q}\|  [{\cal D}_+^{(1)}]_\alpha^\alpha - [{\cal D}]_\alpha^\alpha \|_{HS}^\Lipg & \stackrel{\eqref{mu j +}}{\leq}   \sup_{\alpha \in \sigma_0(\sqrt{- \Delta})} \alpha^{S_q}\|  [\widehat{\cal R}^{(1)}(0)]_\alpha^\alpha \|_{HS}^\Lipg \stackrel{Lemma \,\ref{elementarissimo decay}}{\leq} |{\cal R}|_{S_q}^\Lipg\label{stima Lip rj}
\end{align}
which implies the estimate \eqref{mu + - mu}. 
\end{proof}
\subsection{The iteration}
Let $k \geq 0$ and let us suppose that $({\bf Si})_k$ are true. We prove $({\bf Si})_{k + 1}$. To simplify notations, in this proof we write $| \cdot |_s$ for $| \cdot |_s^\Lipg$.

{\sc Proof of $({\bf S1})_{k + 1}$}. Since the self-adjoint operators $[{\cal D}_k^{(1)}]_\alpha^\alpha\in {\cal S}({\mathbb E}_\alpha)$ are defined on $\Omega_k^\gamma$, the set $\Omega_{k + 1}^\gamma$ is well-defined and by Lemma \ref{homologica equation}, the following estimates hold on $\Omega_{k + 1}^\gamma$
\begin{equation}\label{Psi nu norma alta}
| \Psi_k |_s \lesssim_s  N_k^{2 \tau + 4 \mathtt d + 1} \gamma^{- 1} |{\cal R}_k |_s 
\stackrel{\eqref{Rsb}}{\lesssim_s} N_k^{2 \tau + 4 \mathtt d + 1} N_{k - 1}^{- \mathtt a} \gamma^{- 1} | {\cal R}_0 |_{s + {\mathtt b}}  \,, \quad \forall s \in [\frak s_0 , [q/2]- \overline \mu].
\end{equation}
In particular, by \eqref{piccolezza1}, \eqref{definizione alpha}, \eqref{defN}, taking $\delta_q$ small enough, 
\begin{equation}\label{Psi nu norma bassa}
| \Psi_k  |_{\frak s_0} \leq  1\,.
\end{equation}
By \eqref{Psi nu norma bassa}, we can apply Lemma \ref{lem:inverti} to the map $\Phi_k^{\pm 1} := {\rm exp}(\pm \Psi_k)$, obtaining that
\begin{equation}\label{Phi nu bassa alta}
 | \Phi_k^{\pm 1} - {\rm Id}  |_s \lesssim_s | \Psi_k  |_s\,, \quad  | \Phi_k^{\pm 1} - {\rm Id}  |_s \lesssim_s | \Psi_k  |_s \,, \quad \forall s \in [\frak s_0 , [q/2]- \overline \mu]\,. 
\end{equation}
By \eqref{cal L+} we get  ${\cal L}_{k + 1}(\vphi) := (\Phi_k)_{\omega*} {\cal L}_k (\vphi) =  {\cal D}_{k + 1} + {\cal R}_{k + 1}(\vphi)$, where ${\cal D}_{k + 1} := {\cal D}_k + \Pi_{N_k}({\cal R}_k)_{diag}$ and 
\begin{align}
 {\cal R}_{k + 1} := (\Phi_k^{- 1} - {\rm Id}) \Pi_{N_k}({\cal R}_k)_{diag} +  \Phi_k^{- 1} \Big( \Pi_{N_k}^\bot {\cal R}_k - \omega \cdot \partial_\vphi \Psi_{k , \geq 2} + [{\cal D}_k, \Psi_{k , \geq 2}] +  {\cal R}_k (\Phi_k - {\rm Id})  \Big)\,. \label{cal L nu + 1}
\end{align}
Note that, since ${\cal R}_k$ is defined on $\Omega_k^\gamma$ and $\Psi_k$ is defined on $\Omega_{k + 1}^\gamma$, the remainder ${\cal R}_{k + 1}$ is defined on $\Omega_{k + 1}^\gamma$ too. 
Since the remainder ${\cal R}_k$ is Hamiltonian, the map $\Psi_k$ is Hamiltonian, then $\Phi_k$ is symplectic and the operator ${\cal L}_{k + 1}$ is Hamiltonian.  

Now let us prove the estimates \eqref{Rsb} for ${\cal R}_{k + 1}$. Applying Lemmata \ref{elementarissimo decay}, \ref{interpolazione decadimento Kirchoff}, \ref{lem:inverti} and the estimates  \eqref{Psi nu norma bassa}, \eqref{Psi nu norma alta}, \eqref{Phi nu bassa alta}, for any $s \in [\frak s_0, [q/2] - \overline \mu]$, we get 
\begin{align}
|  (\Phi_k^{- 1} - {\rm Id}) \Pi_{N_k}({\cal R}_k)_{diag} |_s\,,\, |  \Phi_k^{- 1}{\cal R}_k (\Phi_k - {\rm Id})  |_s &   \lesssim_s N_k^{2 \tau +4 \mathtt d +  1}\gamma^{- 1}  |{\cal R}_k  |_s |{\cal R}_k |_{\frak s_0}\, \label{primo pezzo cal R nu + 1}
\end{align}
and 
\begin{equation}\label{secondo pezzo cal R nu + 1}
| \Phi_{k}^{- 1} \Pi_{N_k}^\bot {\cal R}_k|_s \lesssim_s | \Pi_{N_k}^\bot {\cal R}_k |_s + N_k^{2 \tau + 4 \mathtt d + 1}\gamma^{- 1}  | {\cal R}_k  |_s |{\cal R}_k |_{s_0}\,.
\end{equation}
Then, it remains to estimate the term $\Phi_k^{- 1} \big( - \omega \cdot \partial_\vphi \Psi_{k , \geq 2} + [{\cal D}_k , \Psi_{k , \geq 2}] \big)$ in \eqref{cal L nu + 1}. A direct calculation shows that for all $n \geq 2$
\begin{align}
- \omega \cdot \partial_\vphi (\Psi_k^n) + [{\cal D}_k, \Psi_k^n] & = \sum_{i + j = n - 1} \Psi_k^i (- \omega \cdot \partial_\vphi \Psi_k + [{\cal D}_k, \Psi_k]) \Psi_k^j \nonumber\\
& \stackrel{\eqref{Homological equation}}{=} \sum_{i + j = n - 1} \Psi_k^i \Big(\Pi_{N_k}({\cal R}_k)_{diag} - \Pi_{N_k} {\cal R}_k \Big) \Psi_k^j\,, \label{vesuvio}
\end{align}
therefore using \eqref{Psi nu norma bassa}, \eqref{Psi nu norma alta}, Lemmata \ref{elementarissimo decay}, \ref{interpolazione decadimento Kirchoff} and the estimate \eqref{Mnab} we get that for any $n \geq 2$, for any $s \in [\frak s_0, [q/2] - \overline \mu]$
\begin{align}
& \Big| - \omega \cdot \partial_\vphi (\Psi_k^n) + [{\cal D}_k, \Psi_k^n]  \Big|_s  \leq  n^2 C(s)^n  \Big(| \Psi_k |_{\frak s_0}^{n - 1} |{\cal R}_k |_s   + | \Psi_k |_{\frak s_0}^{n - 2} | \Psi_k |_s |{\cal R}_k|_{\frak s_0}  \Big)  \nonumber\\
& \stackrel{\eqref{Psi nu norma alta}, \eqref{Psi nu norma bassa}}{\leq} 2 n^2 C(s)^n N_k^{2 \tau + 4 \mathtt d + 1} \gamma^{- 1} |{\cal R}_k |_s |{\cal R}_k |_{\frak s_0}\,. \label{peppino 0}
\end{align}
The estimate \eqref{peppino 0} implies that  
\begin{align}
\Big|   \omega \cdot \partial_\vphi \Psi_{k , \geq 2} + [{\cal D}_k, \Psi_{k , \geq 2}] \Big|_s & \lesssim \sum_{n \geq 2} \frac{1}{n !} \Big| \omega \cdot \partial_\vphi (\Psi_k^n) + [{\cal D}_k, \Psi_k^n]   \Big|_s \nonumber\\
& \stackrel{\eqref{peppino 0}}{\lesssim} N_k^{2 \tau + 4 \mathtt d + 1} \gamma^{- 1} |{\cal R}_k  |_s |{\cal R}_k|_{\frak s_0} \sum_{n \geq 2}\frac{C(s)^n n^2}{n !} \nonumber\\
& \lesssim_s N_k^{2 \tau+ 4 \mathtt d + 1} \gamma^{- 1} |{\cal R}_k |_s | {\cal R}_k |_{\frak s_0}\,. \label{peppino 1}
\end{align}
Using again \eqref{Psi nu norma alta}-\eqref{Phi nu bassa alta} and Lemma \ref{interpolazione decadimento Kirchoff} we get 
\begin{equation}\label{quarto pezzo cal R nu + 1}
\Big|\Phi_k^{- 1} \big( - \omega \cdot \partial_\vphi \Psi_{k , \geq 2} + [{\cal D}_k, \Psi_{k , \geq 2}] \big)   \Big|_s \lesssim_s N_k^{2 \tau + 4 \mathtt d + 1} \gamma^{- 1} | {\cal R}_k  |_{s} | {\cal R}_k |_{\frak s_0}\,, \quad \forall s \in [\frak s_0, [q/2] - \overline \mu]\,.
\end{equation}
Collecting the estimates \eqref{primo pezzo cal R nu + 1}-\eqref{quarto pezzo cal R nu + 1} we obtain 
\begin{equation}\label{stima induttiva cal R nu 1}
|{\cal R}_{k + 1} |_s \lesssim_s | \Pi_{N_k} {\cal R}_k  |_s + N_k^{2 \tau + 4 \mathtt d +  1} \gamma^{- 1} | {\cal R}_k  |_s | {\cal R}_k  |_{\frak s_0}\,, \quad \forall s \in [\frak s_0, [q/2] - \overline \mu]\,.
\end{equation}
Recalling that $S_q = [q/2] - \overline \mu- \mathtt b$, see \eqref{definizione alpha}, using the smoothing property \eqref{smoothingN} and by \eqref{piccolezza1}, \eqref{Rsb}, one gets for any $s \in [\frak s_0, S_q]$ 
\begin{equation}\label{stima induttiva cal R nu 2 bassa}
| {\cal R}_{k + 1} |_s \lesssim_s N_k^{- \mathtt b } |{\cal R}_{k } |_{s + \mathtt b} + N_k^{2 \tau + 4 \mathtt d +  1} \gamma^{- 1} | {\cal R}_{k}  |_s | {\cal R}_{k}  |_{\frak s_0}\,, \quad | {\cal R}_{k + 1}  |_{s + \mathtt b} \lesssim_s | {\cal R}_{k }  |_{s + \mathtt b}\,.
\end{equation}
By the second inequality in \eqref{stima induttiva cal R nu 2 bassa}
$$
| {\cal R}_{k + 1}  |_{s + \mathtt b} \leq C(s) | {\cal R}_{k}  |_{s + \mathtt b} \stackrel{\eqref{Rsb}}{\leq} C(s ) | {\cal R}_0  |_{s + \mathtt b} N_{k - 1} \leq |{\cal R}_0   |_{s + \mathtt b} N_k\,,
$$
provided $N_{k - 1}^{\chi - 1} \geq C(s)$ for any $k \geq 0$, which is verified by taking $N_0 > 0$ large enough. Therefore, the second inequality in \eqref{Rsb} for ${\cal R}_{k + 1}$ has been proved. Let us prove the first inequality in \eqref{Rsb} at the step $k + 1$. We have 
\begin{align*}
| {\cal R}_{k + 1} |_s & \stackrel{ \eqref{Rsb}}{\lesssim_s} N_k^{- \mathtt b} N_{k - 1} |  {\cal R}_0  |_{s + \mathtt b} + N_k^{2 \tau + 4 \mathtt d + 1}  N_{k - 1}^{- 2 \mathtt a} \gamma^{- 1} | {\cal R}_0 |_{\frak s_0 + \mathtt b} | {\cal R}_0  |_{s + \mathtt b}   \leq |{\cal R}_0  |_{s + \mathtt b} N_k^{- \mathtt a}\,,
\end{align*}
provided 
$$
N_{k}^{\mathtt b - \mathtt a} N_{k - 1}^{- 1} \geq 2 C(s)\,,\quad \gamma^{- 1} | {\cal R}_0 |_{\frak s_0 + \mathtt b} \leq \frac{ N_{k - 1}^{2 \mathtt a} N_k^{- {\mathtt a} - 2 \tau - 4 \mathtt d - 1}}{2 C(s)}\,, \quad \forall k \geq 0
$$
which are verified by \eqref{stime primo resto KAM}, \eqref{defN}, \eqref{definizione alpha} and \eqref{piccolezza1}, by taking $N_0 > 0$ large enough and $\delta_q$ small enough. 

\noindent
The estimate \eqref{rjnu bounded} for $[{\cal D}_{k  +1}^{(1)}]_\alpha^\alpha -  [{\cal D}_{0}^{(1)}]_\alpha^\alpha$ follows, since 
$$
\| [{\cal D}_{k +1}^{(1)}]_\alpha^\alpha -  [{\cal D}_{0}^{(1)}]_\alpha^\alpha \|^\Lipg_{HS} \leq \sum_{j = 0}^k \| [{\cal D}_{j +1}^{(1)}]_\alpha^\alpha -  [{\cal D}_{j}^{(1)}]_\alpha^\alpha \|^\Lipg_{HS} \stackrel{\eqref{mu + - mu}, \eqref{Rsb}}{\lesssim_{q}} \alpha^{- S_q} |{\cal R}_0|_{S_q + \mathtt b}^\Lipg \sum_{j \geq 0} N_{j - 1}^{- \mathtt a} \stackrel{\eqref{stime primo resto KAM}}{\lesssim_{q}}  \alpha^{- S_q} \e\,.
$$

{\sc Proof of $({\bf S2})_{k + 1}$} We now construct a Lipschitz extension of the function $\omega \in \Omega_{k + 1}^\gamma \mapsto [{\cal D}^{(1)}_{k + 1}(\omega)]_\alpha^\alpha \in {\cal S}({\mathbb E}_\alpha)$, for any $\alpha \in \sigma_0(\sqrt{- \Delta})$. We apply Lemma M.5 in \cite{KP} to functions with values in ${\cal S}({\mathbb E}_\alpha)$. Recall that the space ${\cal S}({\mathbb E}_\alpha)$ is a Hilbert subspace of ${\cal B}({\mathbb E}_\alpha)$ equipped by the scalar product defined in \eqref{prodotto scalare traccia matrici}, thus Lemma M.5 in \cite{KP} can be applied, since it holds for functions with values in a Hilbert space. By the inductive hyphothesis, there exists a Lipschitz function $[\widetilde{\cal D}_k^{(1)}]_\alpha^\alpha :  DC(\gamma, \tau) \to {\cal S}({\mathbb E}_\alpha)$, satisfying $[\widetilde{\cal D}_k^{(1)}(\omega)]_\alpha^\alpha = [{\cal D}_k^{(1)}(\omega)]_\alpha^\alpha$, for all $\omega \in \Omega_k^\gamma$. For any $\alpha \in \sigma_0(\sqrt{- \Delta})$, let us define $F_{k, \alpha}(\omega) := [{\cal D}_{k + 1}^{(1)}(\omega)]_\alpha^\alpha- [{\cal D}_k^{(1)}(\omega)]_\alpha^\alpha$, $\omega \in \Omega_{k + 1}^\gamma$. By the estimate \eqref{mu + - mu} one has that  
$$
 \| F_{k, \alpha} \|^\Lipg_{HS}  \leq \alpha^{- S_q} |{\cal R}_k |_{S_q}^\Lipg  \stackrel{\eqref{Rsb}}{\leq} \alpha^{- S_q} |{\cal R}_0 |_{S_q + \mathtt b} N_{k - 1}^{- \mathtt a}
$$
and then by Lemma M.5 in \cite{KP} there exists a Lipschitz extension $\widetilde F_{k , \alpha} : DC(\gamma, \tau) \to {\cal S}({\mathbb E}_\alpha)$ still satisfying the above estimate. Then we define 
$$
[\widetilde{\cal D}_{k + 1}^{(1)}]_\alpha^\alpha := [\widetilde{\cal D}_{k}^{(1)}]_\alpha^\alpha + \widetilde F_{k, \alpha}\,, \qquad \forall \alpha \in \sigma_0(\sqrt{- \Delta})\, 
$$ 
and the claimed estimate \eqref{lambdaestesi} holds at the step $k  + 1$. 
\begin{corollary}\label{lem:convPhi}
{\bf (KAM transformation)} 
Let $q/2 > \frak s_0 + \overline \mu + \mathtt b + 2 s_0 + 2$ (recall \eqref{scelta di M nuove norme}, \eqref{definizione alpha}). 
Then $ \forall \omega \in  \cap_{k \geq 0} \Omega_{k}^{\g} $ 
the sequence 
\be\label{Phicompo}
\widetilde{\Phi}_{k} := \Phi_{0} \circ \Phi_1 \circ \cdots\circ \Phi_{k} 
\ee
is in ${\cal C}^1(\T^\nu, {\cal B}({\bf H}^s_0))$ for any $0 \leq s \leq S_q -2 - 2 s_0$ (recall the definition of $S_q$ given in \eqref{definizione alpha}) and it converges in $ \|\cdot \|_{{\cal C}^1(\T^\nu, {\cal B}({\bf H}^s_0))}^{\Lipg}$ to an operator $\Phi_{\infty}$ which satisfies
\be\label{Phinftys}
\left\|\Phi_{\infty}^{\pm 1} - {\rm Id} \right\|_{{\cal C}^1(\T^\nu, {\cal B}({\bf H}^s_0))}^{\Lipg} \lesssim_q \e \gamma^{-1} \, . 
\ee
Moreover $\Phi_\infty^{\pm 1}$ is symplectic. 
\end{corollary}

\begin{proof}
To simplify notations, we write $| \cdot |_s$ instead of $|\cdot |_s^\Lipg$. First, note that for any $k \geq 0$
\begin{equation}\label{Phi nu widetilde Psi nu}
\Phi_k = {\rm exp}(\Psi_k) = {\rm Id} + {\cal M}_k\,, \qquad {\cal M}_k := \sum_{j \geq 1} \frac{\Psi_k^j}{j !}
\end{equation}
with 
\begin{equation}\label{stima cal M nu}
|{\cal M}_k|_s \stackrel{\eqref{PhINV}}{\lesssim_s} |\Psi_k|_s \stackrel{\eqref{Psinus}}{\lesssim_s} |{\cal R}_{0}|_{s+\mathtt b}^\Lipg \gamma^{-1} N_{k}^{2 \t+ 4 \mathtt d +1} N_{k - 1}^{- \mathtt a} \stackrel{\eqref{stime primo resto KAM}}{\lesssim_q} \e \gamma^{- 1} N_{k}^{2 \t+ 4 \mathtt d +1} N_{k - 1}^{- \mathtt a}\,, \quad \forall \frak s_0 \leq s \leq S_q\,.
\end{equation}
Therefore, by applying Lemma \ref{lemma decadimento Kirchoff in x}-$(ii)$ one gets that for any $0 \leq s \leq S_q - 2 - 2 s_0 $, ${\cal M}_k \in W^{2, \infty}(\T^\nu, {\cal B}({\bf H}^s_0))$ with $\| {\cal M}_k \|_{W^{2, \infty}(\T^\nu, {\cal L}({\bf H}^s_0))} \lesssim_q \e \gamma^{- 1} N_{k}^{2 \t+ 4 \mathtt d +1} N_{k - 1}^{- \mathtt a}$. By the property \eqref{immersione W k infinito Ck}, applied with $p = 1$ and $E = {\cal B}({\bf H}^s_0)$, one gets that ${\cal M}_k \in {\cal C}^1(\T^\nu, {\cal B}({\bf H}^s_0))$ and 
\begin{equation}\label{stima C1 Hs X}
\| {\cal M}_k \|_{{\cal C}^1(\T^\nu, {\cal B}({\bf H}^s_0))} \leq \| {\cal M}_k \|_{W^{2, \infty}(\T^\nu, {\cal B}({\bf H}^s_0))} \lesssim_q \e \gamma^{- 1} N_{k}^{2 \t+ 4 \mathtt d +1} N_{k - 1}^{- \mathtt a}\,, \quad \forall 0 \leq s \leq S_q - 2 - 2 s_0\,. 
\end{equation}
Therefore one gets that $\Phi_k \in {\cal C}^1(\T^\nu, {\cal B}({\bf H}^s_0))$ and hence $\widetilde \Phi_k \in {\cal C}^1(\T^\nu, {\cal B}({\bf H}^s_0))$ for any $k \geq 0$, using the algebra property of the space ${\cal C}^1(\T^\nu, {\cal B}({\bf H}^s_0))$. 
By \eqref{Phicompo}-\eqref{Phi nu widetilde Psi nu}, for any $k \geq 0$, one gets
\begin{equation}\label{marek hamsik0}
\widetilde \Phi_{k + 1} = \widetilde \Phi_k  \Phi_{k + 1} =\widetilde \Phi_k + \widetilde \Phi_k {\cal M}_{k + 1}\,,
\end{equation}
therefore \eqref{stima C1 Hs X} imply that 
\begin{align}
\| \widetilde \Phi_{k + 1}\|_{{\cal C}^1(\T^\nu, {\cal B}({\bf H}^s_0))} & \leq \| \widetilde \Phi_{k}\|_{{\cal C}^1(\T^\nu, {\cal B}({\bf H}^s_0))} (1 + \e_k (q) )\,, \quad \e_k (q) := C(q) \e \gamma^{- 1} N_{k + 1}^{2 \tau + 4 \mathtt d + 1} N_{k }^{- \mathtt a}\,. \label{bound induttivo widetilde Phi nu}
\end{align}
Iterating the above inequality, one then prove that for any $k \geq 0$ 
\begin{align}
\| \widetilde \Phi_{k }\|_{{\cal C}^1(\T^\nu, {\cal B}({\bf H}^s_0))} \leq \prod_{j = 0}^{k - 1} (1 + \e_j(q))\,. \label{bound widetilde Phi nu}
\end{align}
Using that 
$$
\ln\Big( \prod_{j = 0}^{k - 1} (1 + \e_j(q)) \Big) = \sum_{j = 0}^{k - 1} \ln (1 + \e_j(q)) \leq \sum_{j \geq 0} \e_j(q) \stackrel{\eqref{bound induttivo widetilde Phi nu}, \eqref{definizione alpha}, \eqref{piccolezza1}}{\leq} C_1(q)\,,
$$
One gets that 
\begin{equation}\label{bound widetilde Phi nu finale}
\| \widetilde \Phi_{k }\|_{{\cal C}^1(\T^\nu, {\cal B}({\bf H}^s_0))}  \leq {\rm exp}(C_1(q)) =: C_2(q)\,, \quad \forall \nu \geq 0\,. 
\end{equation}
Now we show that $(\widetilde \Phi_k)_{k \geq 0}$ is a Cauchy sequence with respect to the norm $\| \cdot \|_{{\cal C}^1(\T^\nu, {\cal B}({\bf H}^s_0))}$. One has  
\begin{align}
\|\widetilde \Phi_{k + j} - \widetilde\Phi_k\|_{{\cal C}^1(\T^\nu, {\cal B}({\bf H}^s_0))} & \leq \sum_{i = k}^{k + j - 1} \|\widetilde\Phi_{i + 1} - \widetilde\Phi_i \|_{{\cal C}^1(\T^\nu, {\cal B}({\bf H}^s_0))}   \stackrel{\eqref{marek hamsik0}}{\lesssim}  \sum_{i = k}^{k + j - 1} \|\widetilde\Phi_i \|_{{\cal C}^1(\T^\nu, {\cal B}({\bf H}^s_0))} \| {\cal M}_{i + 1}\|_{{\cal C}^1(\T^\nu, {\cal B}({\bf H}^s_0))} \nonumber\\
& \stackrel{\eqref{bound widetilde Phi nu finale}, \eqref{stima C1 Hs X}}{\lesssim_q} \e \gamma^{- 1} \sum_{i \geq k} N_{i + 1}^{2 \tau + 4 \mathtt d + 1} N_i^{- \mathtt a} \lesssim_q \e \gamma^{- 1} N_{k + 1}^{2 \tau + 4 \mathtt d + 1} N_k^{- \mathtt a} \to 0
 \end{align}
 by using \eqref{defN}, \eqref{definizione alpha}. 
Thus $\widetilde\Phi_k$ converges with respect to the norm $\| \cdot \|_{{\cal C}^1(\T^\nu, {\cal B}({\bf H}^s_0))}$ to an operator $\Phi_\infty$ which satisfies the estimate 
$$
\|\Phi_\infty - {\rm Id}\|_{{\cal C}^1(\T^\nu, {\cal B}({\bf H}^s_0))} \lesssim_q \e \gamma^{- 1}\,.
$$
 Similarly, one can show that 
$$
\widetilde \Phi_k^{- 1} = \Phi_k^{- 1} \circ \ldots \circ \Phi_0^{- 1}
$$
is a Cauchy sequence and since $\widetilde \Phi_k^{- 1} \widetilde \Phi_k = {\rm Id}$ for any $k \geq 0$, $\widetilde \Phi_k^{- 1}$ converges to $\Phi_\infty^{- 1}$ and the estimate \eqref{Phinftys} for $\Phi_\infty^{- 1}$ holds. Since $\Phi_k$ is symplectic for any $k \geq 0$, $\Phi_\infty$ is a symplectic map too. 
\end{proof}

Let us define for all $\alpha \in \sigma_0(\sqrt{- \Delta})$, for all $\omega \in DC(\gamma, \tau)$, the self-adjoint blocks $[{\cal D}_\infty^{(1)}(\omega)]_\alpha^\alpha$ as 
\begin{equation}\label{definizione bf D infinito}
[{\cal D}_\infty^{(1)}(\omega)]_\alpha^\alpha : = \lim_{\nu \to + \infty} [\widetilde{\cal D}_\nu^{(1)}(\omega)]_\alpha^\alpha 
\end{equation}
It could happen that $ \Omega_{k_0}^\g = \emptyset $ (see \eqref{Omgj}) for some $ k_0 $. In such a case
the iterative process of Theorem \ref{thm:abstract linear reducibility} stops  after finitely many steps. 
However, we can always set $ [\widetilde{\cal D}_{k}^{(1)}]_\alpha^\alpha  :=  [\widetilde{\cal D}_{k_0}^{(1)}]_\alpha^\alpha $, $ \forall k \geq k_0 $, for all $\alpha \in \sigma_0(\sqrt{- \Delta})$ and the functions $ [{\cal D}_\infty^{(1)}(\cdot)]_\alpha^\alpha : DC(\gamma, \tau) \to {\cal S}({\mathbb E}_\alpha)$  are always well defined.

\begin{corollary} {\bf (Final blocks)} For any $ k \geq 0, \alpha \in \sigma_0(\sqrt{- \Delta}) $, 
\be\label{autovcon}
\| [{\cal D}_\infty^{(1)}]_\alpha^\alpha - [\widetilde{\cal D}_{k}^{(1)}]_\alpha^\alpha \|^{\Lipg}_{HS} 
\lesssim_q
 \alpha^{- S_q} \e N_{k -1}^{-\mathtt a}  \, , \ \ 
\|  [{\cal D}_\infty^{(1)}]_\alpha^\alpha - [\widetilde{\cal D}_{0}^{(1)}]_\alpha^\alpha\|^{\Lipg}_{HS} \lesssim_q  \alpha^{- S_q} \e \, . 
\ee
\end{corollary}

\begin{proof}
The bound \eqref{autovcon} follows by \eqref{lambdaestesi}, \eqref{Rsb}, \eqref{stime primo resto KAM} by summing the telescoping series. 
\end{proof}
Now we define the set 
\begin{align}
\Omega_{\infty}^{2 \g} & :=  \Big\{\omega \in DC(\gamma, \tau) : \| {\mathbb A}_{\infty}^{-}(\ell, \alpha, \beta ; \omega)^{- 1} \|_{{\rm Op}(\alpha, \beta)} \leq \frac{\alpha^{\mathtt d} \beta^{\mathtt d}\langle \ell \rangle^\tau}{2 \gamma }\,, \quad \forall (\ell, \alpha, \beta) \in \Z^\nu \times \sigma_0(\sqrt{- \Delta}) \times \sigma_0(\sqrt{- \Delta})\,, \nonumber\\ 
& \qquad (\ell, \alpha, \beta)\neq (0, \alpha, \alpha)\,,
\,\,  \| {\mathbb A}_{\infty}^+(\ell, \alpha, \beta; \omega)^{- 1} \|_{{\rm Op}(\alpha, \beta)} \leq \frac{\langle \ell \rangle^\tau}{2 \gamma \langle \alpha + \beta \rangle}\,, \forall (\ell, \alpha , \beta) \in \Z^\nu \times \sigma_0(\sqrt{- \Delta}) \times \sigma_0(\sqrt{- \Delta})\Big\}\label{Omegainfty}
\end{align}
where the operators ${\mathbb A}^{\pm}_{\infty}(\ell, \alpha, \beta) = {\mathbb A}^{\pm}_{\infty}(\ell, \alpha, \beta; \omega)  : {\cal B}({\mathbb E}_\beta, {\mathbb E}_\alpha) \to {\cal B}({\mathbb E}_\beta, {\mathbb E}_\alpha)$ are defined for any $\omega \in DC(\gamma, \tau)$, $(\ell, \alpha, \beta) \in \Z^\nu \times \sigma_0(\sqrt{- \Delta}) \times \sigma_0(\sqrt{- \Delta})$ as 
\begin{equation}\label{bf A infinito - (ell,alpha,beta)}
{\mathbb A}^{-}_{\infty}(\ell, \alpha, \beta)  := \omega \cdot \ell {\mathbb I}_{\alpha, \beta} + M_L([{\cal D}_\infty^{(1)}]_\alpha^\alpha) - M_R([{\cal D}_\infty^{(1)}]_\beta^\beta)
\end{equation}
\begin{equation}\label{bf A infinito + (ell,alpha,beta)}
{\mathbb A}^{+}_{\infty}(\ell, \alpha, \beta) := \omega \cdot \ell {\mathbb I}_{\alpha, \beta} + M_L([{\cal D}_\infty^{(1)}]_\alpha^\alpha) + M_R([\overline{\cal D}_\infty^{(1)}]_\beta^\beta)\,.
\end{equation}

\begin{lemma} 
One has 
\be\label{cantorinclu}
\Omega_{\infty}^{2 \g} \subset \cap_{k \geq 0} \Omega_{k}^\g \, . 
\ee 
\end{lemma}

\begin{proof} 
It suffices to show that for any $k \geq 0$, $\Omega_\infty^{2 \gamma} \subseteq \Omega_k^\gamma$. We argue by induction. For $k = 0$, since $\Omega_0^\gamma = DC(\gamma, \tau)$, it follows from the definition \eqref{Omegainfty} that $\Omega_\infty^{2 \gamma} \subseteq \Omega_0^\gamma$.
Assume that $\Omega_\infty^{2 \gamma} \subseteq \Omega_k^\gamma$ for some $k \geq 0$ and let us prove that $\Omega_\infty^{2 \gamma} \subseteq \Omega_{k + 1}^\gamma$. Let $\omega \in \Omega_\infty^{2 \gamma}$. By the inductive hyphothesis $\omega \in \Omega_k^\gamma$, hence by Theorem \ref{thm:abstract linear reducibility}, the operators $[{\cal D}^{(1)}_k(\omega)]_\alpha^\alpha \in {\cal S}({\mathbb E}_\alpha)$ are well defined for all $\alpha \in \sigma_0(\sqrt{- \Delta})$ and $[{\cal D}^{(1)}_k(\omega)]_\alpha^\alpha = [\widetilde{\cal D}^{(1)}_k(\omega)]_\alpha^\alpha$.

\noindent
Let $(\ell, \alpha, \beta) \in \Z^\nu \times \sigma_0(\sqrt{- \Delta}) \times \sigma_0(\sqrt{- \Delta})$ with $(\ell, \alpha, \beta) \neq (0, \alpha, \alpha)$, $\langle \ell, \alpha, \beta \rangle \leq N_k$. By the definitions \eqref{bf A nu - (ell,alpha,beta)}, \eqref{bf A nu + (ell,alpha,beta)}, also the operators ${\mathbb A}_k^{\pm}(\ell, \alpha, \beta; \omega)$ are well defined.
Since $\omega \in \Omega_\infty^{2 \gamma}$, the operator ${\mathbb A}^-_\infty(\ell, \alpha, \beta; \omega)$ is invertible and we may write 
$$
{\mathbb A}_k^-(\ell, \alpha, \beta; \omega) = {\mathbb A}_\infty^- (\ell, \alpha, \beta; \omega) + \Delta_\infty^-(\ell, \alpha, \beta; \omega) =  {\mathbb A}_\infty^- (\ell, \alpha, \beta; \omega) \Big( {\mathbb I}_{\alpha, \beta} + {\mathbb A}_\infty^- (\ell, \alpha, \beta; \omega)^{- 1}\Delta_\infty^-(\ell, \alpha, \beta; \omega)  \Big)  
$$
where
$$
\Delta^-_\infty(\ell, \alpha, \beta; \omega) := M_L\Big( [{\cal D}_k^{(1)}(\omega)]_\alpha^\alpha - [{\cal D}_\infty^{(1)}(\omega)]_\alpha^\alpha \Big) - M_R \Big( [{\cal D}_k^{(1)}(\omega)]_\beta^\beta - [{\cal D}_\infty^{(1)}(\omega)]_\beta^\beta \Big)\,.
$$
By the property \eqref{norma operatoriale ML MR} and by the estimate \eqref{autovcon}
\begin{equation}\label{copenaghen 0}
\| {\Delta}^-_\infty(\ell, \alpha, \beta; \omega) \|_{{\rm Op}(\alpha, \beta)} \lesssim_q N_{k - 1}^{- \mathtt a} \e (\alpha^{- S_q} + \beta^{- S_q})\,.
\end{equation}
Since $\langle \ell , \alpha , \beta \rangle\leq N_k$, one has
\begin{align}
\| {\mathbb A}_\infty^-(\ell, \alpha, \beta; \omega)^{- 1} \Delta_\infty^-(\ell, j, j'; \omega)  \|_{{\rm Op}(\alpha, \beta)} & \lesssim_q \frac{\langle \ell \rangle^{\tau} \alpha^{\mathtt d} \beta^{\mathtt d}}{ \gamma} N_{k - 1}^{- \mathtt a} \e (\alpha^{- S_q} + \beta^{- S_q})  \nonumber\\ 
& \lesssim_q N_k^{\tau + 2 \mathtt d} N_{k - 1}^{- \mathtt a} \e \gamma^{- 1} \stackrel{\eqref{definizione alpha}-\eqref{piccolezza1}}{\leq} \frac12\,. \label{copenaghen - 10}
\end{align}
for $N_0 > 0$ in \eqref{piccolezza1} large enough and $\delta_q$ in \eqref{piccolezza1} small enough. Thus the operator ${\mathbb A}_k^-(\ell, \alpha, \beta; \omega)$ is invertible, with inverse given by the Neumann series. Hence 
\begin{align*}
\|{\mathbb A}_k^-(\ell, \alpha, \beta; \omega)^{- 1} \|_{{\rm Op}(\alpha, \beta)} & \leq \frac{\| {\mathbb A}_\infty^-(\ell, \alpha, \beta; \omega)^{- 1}\|_{{\rm Op}(\alpha, \beta)}}{1 - \| {\mathbb A}_\infty^-(\ell, \alpha, \beta ; \omega)^{- 1} {\Delta}_\infty^-(\ell, \alpha, \beta; \omega) \|_{{\rm Op}(\alpha, \beta)}}  \\
& \stackrel{\eqref{copenaghen - 10}}{\leq } 2 \| {\mathbb A}_\infty^-(\ell, \alpha, \beta; \omega)^{- 1}\|_{{\rm Op}(\alpha, \beta)} \stackrel{\eqref{Omegainfty}}{\leq} \frac{\langle \ell \rangle^\tau \alpha^{\mathtt d} \beta^{\mathtt d}}{\gamma }\,.
\end{align*}
By similar arguments, one can also obtain that $\| {\mathbb A}_k^+(\ell, \alpha, \beta; \omega)^{- 1}\|_{{\rm Op}(\alpha, \beta)} \leq \frac{\langle \ell \rangle^\tau}{\gamma ( \alpha + \beta )}$, for any $(\ell, \alpha, \beta) \in \Z^\nu \times \sigma_0(\sqrt{- \Delta}) \times \sigma_0(\sqrt{- \Delta})$ with $\langle \ell, \alpha , \beta \rangle\leq N_k$, then $\omega \in \Omega_{k + 1}^\gamma$ and the proof is concluded.  
\end{proof}

\noindent
To state the main result of this section we introduce the operator
\begin{equation}\label{cal D infinito}
{\cal D}_\infty = {\cal D}_\infty(\omega):= \ii \begin{pmatrix}
 - {\cal D}_\infty^{(1)}(\omega) & 0 \\
0 &    \overline{\cal D}_\infty^{(1)}(\omega)
\end{pmatrix}\,,\quad {\cal D}_\infty^{(1)}(\omega) := {\rm diag}_{\alpha \in \sigma_0(\sqrt{- \Delta})}[{\cal D}_\infty^{(1)}(\omega)]_\alpha^\alpha\,,
\end{equation}
for any $\omega \in DC(\gamma, \tau)$, where the self-adjoint operators $[{\cal D}_\infty^{(1)}(\omega)]_\alpha^\alpha \in {\cal S}({\mathbb E}_\alpha)$, $\alpha \in \sigma_0(\sqrt{- \Delta})$, are defined in \eqref{definizione bf D infinito}. For any $\omega \in DC(\gamma, \tau)$, the vector field ${\cal D}_\infty(\omega)$ is a $\vphi$-independent block-diagonal bounded linear operator ${\cal D}_\infty(\omega) : {\bf H}^s_0 \to {\bf H}^{s - 1}_0$, for any $s \geq 1$.

\begin{theorem}\label{teoremadiriducibilita}
Let  $q/2 > \frak s_0 + \bar \mu + \mathtt b + 2 s_0 + 2$. Then there exists a constant $\delta_q = \delta(q, \tau, \mathtt d, \nu, d) > 0$ (possibly smaller than the one in \eqref{piccolezza1}) such that if 
\begin{equation}\label{final KAM smallness condition}
\e \gamma^{- 1} \leq \delta_q\,,
\end{equation} 
on the set $ \Omega_\infty^{2 \gamma}$, the Hamiltonian vector field ${\cal L}_0(\vphi)$ in \eqref{L0} is conjugated to the Hamiltonian vector field ${\cal D}_\infty$ by $\Phi_\infty$, namely for all $\omega \in \Omega_\infty^{2 \gamma}$, 
\be\label{Lfinale}
{\cal D}_{\infty}(\omega)
= (\Phi_{\infty})_{\omega *}{\cal L}_0(\vphi; \omega) \,.
\ee
\end{theorem}
\begin{proof}  Since $\Omega_\infty^{2 \gamma} \stackrel{\eqref{cantorinclu}}{\subseteq} \cap_{k \geq 0} \Omega_k^\gamma$, the estimate \eqref{Phinftys} holds on the set $\Omega_\infty^{2 \gamma}$, and 
$$
\|\Phi_\infty^{\pm 1} - {\rm Id} \|_{{\cal C}^1(\T^\nu, {\cal B}({\bf H}^s_0))}^\Lipg  {\lesssim_q} \e \gamma^{- 1}\,, \qquad \forall 0 \leq s \leq S_q - 2 s_0 - 2\,.
$$
By \eqref{Lnu+1}, \eqref{Phicompo}, for any $k \geq 1$, we get 
\begin{equation}\label{spanish}
{\cal L}_k (\vphi) = (\widetilde{\Phi}_{k - 1})_{\omega *}{\cal L}_0 = \widetilde \Phi_k(\vphi)^{- 1} \Big({\cal L}_0 (\vphi) \widetilde \Phi_k (\vphi) - \omega \cdot \partial_\vphi \widetilde \Phi_k (\vphi)  \Big) = {\cal D}_k  + {\cal R}_k (\vphi)\,, \qquad \widetilde  \Phi_k = \Phi_0 \circ \ldots \circ \Phi_k \,.
\end{equation}
Note that, for all $k \geq 0$, for any $s \in [0, S_q]$ 
\begin{align}
& |{\cal D}_\infty^{(1)} - {\cal D}_k^{(1)}|^\Lipg_{s} \leq |{\cal D}_\infty^{(1)} - {\cal D}_k^{(1)}|^\Lipg_{S_q} = \sup_{\alpha \in \sigma_0(\sqrt{- \Delta})} \alpha^{S_q} \| [{\cal D}_k^{(1)}]_\alpha^\alpha - [{\cal D}_\infty^{(1)}]_\alpha^\alpha \|^\Lipg_{HS} \nonumber\\
&  \stackrel{\eqref{autovcon}}{\lesssim_q}  \e N_{k - 1}^{- \mathtt a} \stackrel{k \to + \infty}{\to} 0 \qquad \text{and} \qquad |{\cal R}_k |_{s}^\Lipg  \stackrel{\eqref{Rsb}, \eqref{stime primo resto KAM}}{\lesssim_q} \e N_{k - 1}^{- \mathtt a} \stackrel{k \to + \infty}{\to} 0\,.
\end{align}
Hence, $|{\cal L}_k - {\cal D}_\infty|_s^\Lipg \stackrel{k \to + \infty}{\to}0$ for all $\frak s_0 \leq s \leq S_q$. By applying Lemma \ref{lemma decadimento Kirchoff in x} and the property \eqref{immersione W k infinito Ck}, ${\cal R}_k \in W^{1, \infty}(\T^\nu, {\cal B}({\bf H}^s_0)) \subseteq {\cal C}^0(\T^\nu, {\cal B}({\bf H}^s_0) ) $ for any $0 \leq s \leq S_q  - 2 s_0 - 1$ with 
$$
\| {\cal R}_k\|_{{\cal C}^0(\T^\nu, {\cal B}({\bf H}^s_0) )} \leq \|{\cal R}_k \|_{W^{1, \infty}(\T^\nu, {\cal B}({\bf H}^s_0))} \lesssim | {\cal R}_k|_{s + 2 s_0 + 1} \to 0
$$
and 
$$
\| {\cal D}_k - {\cal D}_\infty\|_{{\cal B}({\bf H}^s_0)} \leq | {\cal D}_k - {\cal D}_\infty |_{s + 2 s_0} \to 0\,.
$$
Thus, ${\cal L}_k \to {\cal D}_\infty$ with respect to the norm $\| \cdot \|_{{\cal C}^0(\T^\nu, {\cal B}({\bf H}^s_0))}$, for any $0 \leq s \leq S_q - 2 s_0 - 1$.  Since, by Lemma \ref{lem:convPhi}, $\widetilde\Phi_k^{\pm 1} \stackrel{k \to + \infty}{\to} \Phi_\infty^{\pm 1}$ with respect to the norm $\|\cdot \|_{{\cal C}^1(\T^\nu, {\cal B}({\bf H}^s_0))}^\Lipg$, formula \eqref{Lfinale} follows by taking the limit for $k \to + \infty$ in \eqref{spanish}. 
\end{proof}
\section{Measure estimates}\label{sezione stime in misura}
 In this Section we estimate the measure of the set $\Omega_\infty^{2 \gamma}$ defined in \eqref{Omegainfty}. We fix the constants $\tau$ and $\mathtt d$ in \eqref{Omegainfty} as 
 \begin{equation}\label{scelta tau mathtt d}
 \mathtt d:= 2d \,, \qquad \tau := \nu + 4 d \,.
 \end{equation} 
 We prove the following Theorem: 
 \begin{theorem}\label{teorema finale stima in misura}
 Under the same assumptions of Theorem \ref{teoremadiriducibilita}, one has 
 $$
 |\Omega \setminus \Omega_\infty^{2 \gamma}| = O(\gamma)\,. 
 $$
 \end{theorem} 
The rest of this section is devoted to the proof of Theorem \ref{teorema finale stima in misura}.
 
 \noindent
 By the definition \eqref{Omegainfty} one can write that 
 \begin{equation}\label{Omega - Omega infinito preliminare}
\Omega \setminus \Omega_\infty^{2 \gamma} =  \big(\Omega \setminus DC(\gamma, \tau)\big) \cup \big( DC(\gamma, \tau) \setminus \Omega_\infty^{2 \gamma} \big)\,. 
 \end{equation}
 By a standard volume estimate one has 
 \begin{equation}\label{misura diofantei calssica}
 |\Omega \setminus DC(\gamma, \tau)| \lesssim \gamma\,. 
 \end{equation}
 Using again the definition \eqref{Omegainfty}, we write 
 \begin{equation}\label{Omega - Omega infinito}
DC(\gamma, \tau) \setminus \Omega_\infty^{2 \gamma} = \bigcup_{\begin{subarray}{c}
(\ell, \alpha, \beta) \in \Z^\nu \times \sigma_0(\sqrt{- \Delta}) \times \sigma_0(\sqrt{- \Delta}) \\
(\ell, \alpha, \beta) \neq (0, \alpha, \alpha)
\end{subarray}}  R(\ell, \alpha, \beta) \bigcup_{(\ell, \alpha, \beta) \in \Z^\nu \times \sigma_0(\sqrt{- \Delta}) \times \sigma_0(\sqrt{- \Delta})} Q(\ell, \alpha, \beta)  \,,
 \end{equation}
 where for any $(\ell, \alpha, \beta) \in \Z^\nu \times \sigma_0(\sqrt{- \Delta}) \times \sigma_0(\sqrt{- \Delta})$, $(\ell, \alpha, \beta) \neq (0, \alpha, \alpha)$, 
 \begin{align}
 R(\ell, \alpha, \beta) & := \Big\{ \omega \in DC(\gamma, \tau) : {\mathbb A}_{\infty}^{-}(\ell, \alpha, \beta ; \omega)\,\text{is not invertible or it is invertible and} \nonumber\\
 & \quad \| {\mathbb A}_{\infty}^{-}(\ell, \alpha, \beta ; \omega)^{- 1} \|_{{\rm Op}(\alpha, \beta)} > \frac{\alpha^{\mathtt d} \beta^{\mathtt d}\langle \ell \rangle^\tau}{2 \gamma } \Big\} \label{definizione R ell alpha beta}
 \end{align}
 and for any $(\ell, \alpha, \beta) \in \Z^\nu \times \sigma_0(\sqrt{- \Delta}) \times \sigma_0(\sqrt{- \Delta})$ 
 \begin{align}
 Q(\ell, \alpha, \beta) & := \Big\{ \omega \in DC(\gamma, \tau) : {\mathbb A}_{\infty}^{+}(\ell, \alpha, \beta ; \omega)\,\text{is not invertible or it is invertible and} \nonumber\\
 & \quad \| {\mathbb A}_{\infty}^{+}(\ell, \alpha, \beta ; \omega)^{- 1} \|_{{\rm Op}(\alpha, \beta)} > \frac{\langle \ell \rangle^\tau}{2 \gamma(\alpha + \beta) } \Big\}\,.  \label{definizione Q ell alpha beta}
 \end{align}
 By \eqref{rappresentazione a blocchi cal D0 (1)}, for any $\alpha \in \sigma_0(\sqrt{- \Delta})$, we can write  
 $$
 [{\cal D}_\infty^{(1)}]_\alpha^\alpha  = \mu_\alpha^0 {\mathbb I}_\alpha + R_{ \infty, \alpha}\,, \qquad R_{ \infty, \alpha} := [{\cal D}_\infty^{(1)}]_\alpha^\alpha - [{\cal D}_0^{(1)}]_\alpha^\alpha \in {\cal S}({\mathbb E}_\alpha)
 $$ which is self-adjoint and Lipschitz continuous with respect to the parameter $\omega \in DC(\gamma, \tau)$. We set 
 \begin{equation}\label{spettro correzione cappuccio}
 {\rm spec}(R_{ \infty, \alpha}(\omega)) := \big\{ r_k^{(\alpha)}(\omega)\,, k =1 , \ldots, d_\alpha \big\} \quad \text{with} \quad r_1^{(\alpha)}(\omega)\leq r_2^{(\alpha)}(\omega) \leq \ldots \leq r_{n_\alpha}^{(\alpha)}(\omega)\,,
 \end{equation}
 where $n_\alpha$ is the dimension of the finite dimensional space ${\mathbb E}_\alpha$ 
 \begin{equation}\label{dimensione spazi bf E alpha}
 n_\alpha = {\rm card}\Big\{ j \in \Z^d \setminus \{ 0\} : |j| = \alpha \Big\} \simeq \alpha^{d - 1}\,.
 \end{equation}
 By the property \eqref{proprieta elementare norma operatoriale autovalori}, one has that 
 \begin{align}
 |r_k^{(\alpha)}(\omega)| & \leq \| R_{\infty, \alpha} \|_{{\cal B}({\mathbb E}_\alpha)} \stackrel{Lemma\,\ref{proprieta facili norma L2 matrici}-(i)}{\leq} \| R_{\infty, \alpha} \|_{HS} \stackrel{\eqref{autovcon}}{\lesssim_q} \e \alpha^{- S_q}\, \label{stima sup r k alpha}
 \end{align}
 uniformly for any $\omega \in DC(\gamma, \tau)$. 
 
 \noindent
 Furthermore, by Lemma \ref{risultato astratto operatori autoaggiunti}-$(i)$ the functions $\omega \mapsto r_k^{(\alpha)}(\omega)$ are Lipschitz with respect to $\omega$, since 
 \begin{align}
 |r_k^{(\alpha)}(\omega_1) - r_k^{(\alpha)}(\omega_2)| &\leq \| R_{\infty, \alpha}(\omega_1) -  R_{\infty, \alpha}(\omega_2)  \|_{{\cal B}({\mathbb E}_\alpha)} \stackrel{Lemma\,\ref{proprieta facili norma L2 matrici}-(i)}{\leq} \| R_{\infty, \alpha}(\omega_1) -  R_{\infty, \alpha}(\omega_2)  \|_{HS} \nonumber\\
 & \leq \| R_{\infty, \alpha} \|_{HS}^{\rm lip} |\omega_1- \omega_2| \stackrel{\eqref{autovcon}}{\lesssim_q} \e \gamma^{- 1} \alpha^{- S_q} |\omega_1- \omega_2|\,. \label{dino zoff}
 \end{align}
 We also set
 $$
{\rm spec}([{\cal D}_\infty^{(1)}(\omega)]_\alpha^\alpha) := \big\{ \lambda_k^{(\alpha)}(\omega)\,, k =1 , \ldots, n_\alpha \big\} \quad \text{with} \quad \lambda_1^{(\alpha)}(\omega)\leq \lambda_2^{(\alpha)}(\omega) \leq \ldots \leq \lambda_{n_\alpha}^{(\alpha)}(\omega)\,.
$$
By Lemma \ref{risultato astratto operatori autoaggiunti}-$(ii)$ we have that 
\begin{equation}\label{espansione asintotica autovalori}
\lambda_k^{(\alpha)}(\omega) = \mu_\alpha^0(\omega) + r_k^{(\alpha)}(\omega) \stackrel{\eqref{rappresentazione a blocchi cal D0 (1)}}{=} m\, \alpha + \mathtt r_k^{(\alpha)}(\omega)\,, \qquad \mathtt r_k^{(\alpha)} := c(\alpha) + r_k^{(\alpha)}\,, \quad  \forall k = 1, \ldots, n_\alpha\,.
\end{equation}
By the estimates \eqref{stime c(alpha)}, \eqref{stima sup r k alpha}, \eqref{dino zoff}, one gets 
\begin{equation}\label{stima mathtt r alpha k}
|\mathtt r_k^{(\alpha)}|^{\Lipg} \lesssim_q \e \alpha^{- 1}\,, \quad \forall \alpha \in \sigma_0(\sqrt{- \Delta})\,,\quad  \forall k = 1, \ldots, n_\alpha\,. 
\end{equation}
By the definitions \eqref{bf A infinito - (ell,alpha,beta)}, \eqref{bf A infinito + (ell,alpha,beta)} and by Lemmata \ref{properties operators matrices}, \ref{risultato astratto operatori autoaggiunti}-$(ii)$ the operators ${\mathbb A}_\infty^{\pm}(\ell, \alpha, \beta) : {\cal B}({\mathbb E}_\beta, {\mathbb E}_\alpha) \to {\cal B}({\mathbb E}_\beta, {\mathbb E}_\alpha) $ are self-adjoint with respect to the scalar product \eqref{prodotto scalare traccia matrici} and the following holds: 

\noindent
 for any $(\ell, \alpha, \beta) \in \Z^\nu \times \sigma_0(\sqrt{- \Delta}) \times \sigma_0(\sqrt{- \Delta})$, $(\ell, \alpha, \beta) \neq (0, \alpha, \alpha)$
$$
{\rm spec}\Big({\mathbb A}_\infty^{-}(\ell, \alpha, \beta; \omega)\Big) = \Big\{ \omega \cdot \ell + \lambda_k^{(\alpha)}(\omega) - \lambda_j^{(\beta)}(\omega)\,, \quad k = 1, \ldots, n_\alpha\,,\quad j = 1, \ldots, n_\beta \Big\}
$$
and for any $(\ell, \alpha, \beta) \in \Z^\nu \times \sigma_0(\sqrt{- \Delta}) \times \sigma_0(\sqrt{- \Delta})$
$$
{\rm spec}\Big({\mathbb A}_\infty^{+}(\ell, \alpha, \beta; \omega)\Big) = \Big\{ \omega \cdot \ell + \lambda_k^{(\alpha)}(\omega) + \lambda_j^{(\beta)}(\omega)\,, \quad k = 1, \ldots, n_\alpha\,,\quad j = 1, \ldots, n_\beta \Big\}\,.
$$
Therefore, recalling the definitions \eqref{definizione R ell alpha beta}, \eqref{definizione Q ell alpha beta} and also by applying Lemma \ref{risultato astratto operatori autoaggiunti}-$(iii)$, one has  
\begin{equation}\label{decomposizione risonante 2 autovalori}
R(\ell, \alpha, \beta) \subseteq \widetilde R(\ell, \alpha, \beta) := \bigcup_{k = 1}^{n_\alpha} \bigcup_{j = 1}^{n_\beta} \widetilde R_{k j}(\ell, \alpha, \beta)\,, \quad \forall (\ell, \alpha, \beta) \in \Z^\nu \times \sigma_0(\sqrt{- \Delta}) \times \sigma_0(\sqrt{- \Delta})\,, \quad (\ell, \alpha, \beta) \neq (0, \alpha, \alpha)\,,
\end{equation}
\begin{equation}\label{decomposizione risonante 2 autovalori a}
Q(\ell, \alpha, \beta) \subseteq \widetilde Q(\ell, \alpha, \beta) := \bigcup_{k = 1}^{n_\alpha} \bigcup_{j = 1}^{n_\beta} \widetilde Q_{k j}(\ell, \alpha, \beta)\,, \quad \forall (\ell, \alpha, \beta) \in \Z^\nu \times \sigma_0(\sqrt{- \Delta}) \times \sigma_0(\sqrt{- \Delta})\,
\end{equation}
where for any $(\ell, \alpha, \beta) \in \Z^\nu \times \sigma_0(\sqrt{- \Delta}) \times \sigma_0(\sqrt{- \Delta})$, $(\ell, \alpha, \beta) \neq (0, \alpha, \alpha)$, $k = 1, \ldots,n_\alpha $, $j = 1, \ldots, n_\beta$
\begin{equation}\label{def widetilde R kj alpha beta ell}
\widetilde R_{k j}(\ell, \alpha, \beta) := \Big\{ \omega \in DC(\gamma, \tau)  : |\omega \cdot \ell + \lambda_k^{(\alpha)}(\omega) - \lambda_j^{(\beta)}(\omega)| < \frac{2 \gamma }{\langle \ell \rangle^\tau \alpha^{\mathtt d} \beta^{\mathtt d}}\Big\}\,
\end{equation}
and for any $(\ell, \alpha, \beta) \in \Z^\nu \times \sigma_0(\sqrt{- \Delta}) \times \sigma_0(\sqrt{- \Delta})$, $k = 1, \ldots,n_\alpha $, $j = 1, \ldots, n_\beta$
\begin{equation}\label{def widetilde Q kj alpha beta ell}
\widetilde Q_{k j}(\ell, \alpha, \beta) := \Big\{ \omega \in DC(\gamma, \tau)  : |\omega \cdot \ell + \lambda_k^{(\alpha)}(\omega) + \lambda_j^{(\beta)}(\omega)| < \frac{2 \gamma (\alpha + \beta) }{\langle \ell \rangle^\tau}\Big\}\,.
\end{equation}
Thus, by \eqref{Omega - Omega infinito} one has 
 \begin{equation}\label{Omega - Omega infinito finale}
DC(\gamma, \tau) \setminus \Omega_\infty^{2 \gamma} \subseteq \bigcup_{\begin{subarray}{c}
(\ell, \alpha, \beta) \in \Z^\nu \times \sigma_0(\sqrt{- \Delta}) \times \sigma_0(\sqrt{- \Delta}) \\
(\ell, \alpha, \beta) \neq (0, \alpha, \alpha)
\end{subarray}}  \widetilde R(\ell, \alpha, \beta) \bigcup_{(\ell, \alpha, \beta) \in \Z^\nu \times \sigma_0(\sqrt{- \Delta}) \times \sigma_0(\sqrt{- \Delta})} \widetilde Q(\ell, \alpha, \beta)  \,. 
 \end{equation}
 \begin{lemma}\label{limitazioni indici risonanti}
 $(i)$ If $\widetilde R(\ell, \alpha, \beta) \neq \emptyset$, then $|\alpha - \beta| \lesssim \langle \ell\rangle$. Moreover for any $\alpha, \beta \in \sigma_0(\sqrt{- \Delta})$, $\alpha \neq \beta$, then $\widetilde R(0, \alpha, \beta) = \emptyset$. 
 
 \noindent
 $(ii)$ If $\widetilde Q(\ell, \alpha, \beta) \neq \emptyset$, then $\alpha, \beta \lesssim \langle \ell\rangle$. Moreover for any $\alpha, \beta \in \sigma_0(\sqrt{- \Delta})$ then $\widetilde Q(0, \alpha, \beta) = \emptyset$.
 \end{lemma}
 \begin{proof}
 We prove item $(i)$. The proof of item $(ii)$ is similar. Assume that $\widetilde R(\ell, \alpha, \beta) \neq \emptyset$. Then there exist $k \in \{1, \ldots, n_\alpha \}$, $j \in \{ 1, \ldots, n_\beta \}$ such that $\widetilde R_{kj}(\ell, \alpha, \beta) \neq \emptyset$. For any $\omega \in \widetilde R_{kj}(\ell, \alpha, \beta)$, one has 
 $$
 |\omega \cdot \ell + \lambda_k^{(\alpha)}(\omega ) - \lambda_j^{(\beta)}(\omega)| < \frac{2 \gamma}{\langle \ell \rangle^\tau \alpha^{\mathtt d} \beta^{\mathtt d}}
 $$
 and using \eqref{espansione asintotica autovalori} and the estimates \eqref{stime m}, \eqref{stima mathtt r alpha k}, for $\e$ small enough, one gets that 
 \begin{equation}\label{lambda alpha beta k j differenza}
 |\lambda_k^{(\alpha)} - \lambda_j^{(\beta)}| \geq \frac12 |\alpha - \beta| - C(q) \e (\alpha^{- 1} + \beta^{- 1})
 \end{equation}
 implying that 
 $$
 |\alpha - \beta| \leq |\omega | |\ell| + \frac{2 \gamma}{\langle \ell \rangle^\tau \alpha^{\mathtt d} \beta^{\mathtt d}} + C(q) \e (\alpha^{- 1} + \beta^{- 1}) \lesssim \langle \ell \rangle\,. 
 $$
 Now we show that if $\alpha, \beta \in \sigma_0(\sqrt{- \Delta})$ with $\alpha \neq \beta$, then $\widetilde R_{kj}(0, \alpha, \beta) = \emptyset$ for any $k \in \{ 1, \ldots, n_\alpha \}$, $j \in \{ 1, \ldots, n_\beta \}$. By using \eqref{lambda alpha beta k j differenza} and Lemma \ref{lemma spettro laplaciano}-$(ii)$, for $\e$ small enough one gets 
 \begin{align}
 |\lambda_k^{(\alpha)} - \lambda_j^{(\beta)}| & \geq  C_1 \Big(\frac{1}{\alpha} + \frac{1}{\beta} \Big) \stackrel{\alpha, \beta \geq 1}{\geq} \frac{C_1}{\alpha \beta}\,,
 \end{align}
for some constant $C_1 > 0$ implying that $\widetilde R_{kj}(0, \alpha, \beta) = \emptyset$ by the definition \eqref{def widetilde R kj alpha beta ell}, since $\mathtt d > 1$ and taking $0 < \gamma < C_1$. Item $(i)$ then follows by recalling the definition of $\widetilde R(\ell, \alpha, \beta)$ in \eqref{decomposizione risonante 2 autovalori}. 
 \end{proof}
\begin{lemma}\label{stima in misura insiemi risonanti}
For $\e \gamma^{-1}$ small enough, the following holds:

\noindent
$(i)$ For any $(\ell, \alpha, \beta) \in \Z^\nu \times \sigma_0(\sqrt{- \Delta}) \times \sigma_0(\sqrt{- \Delta})$, $(\ell, \alpha, \beta) \neq (0, \alpha, \alpha)$, if $\widetilde R(\ell, \alpha, \beta) \neq \emptyset$ then $|\widetilde R(\ell, \alpha, \beta) | \lesssim \gamma \alpha^{d - 1- \mathtt d} \beta^{d - 1- \mathtt d} \langle \ell \rangle^{- \tau - 1}$. 

\noindent
$(ii)$ For any $(\ell, \alpha, \beta) \in \Z^\nu \times \sigma_0(\sqrt{- \Delta}) \times \sigma_0(\sqrt{- \Delta})$, if $\widetilde Q(\ell, \alpha, \beta) \neq \emptyset$ then $|\widetilde Q(\ell, \alpha, \beta)| \lesssim \gamma \alpha^{d - 1} \beta^{d - 1}\langle \alpha + \beta \rangle \langle \ell \rangle^{- \tau - 1}$.

\end{lemma}
 \begin{proof}
 Let us prove item $(i)$. The proof of item $(ii)$ can be done by using similar arguments. Let $(\ell, \alpha, \beta) \in \Z^\nu \times \sigma_0(\sqrt{- \Delta}) \times \sigma_0(\sqrt{- \Delta})$ with $(\ell, \alpha, \beta) \neq (0, \alpha, \alpha)$. By \eqref{decomposizione risonante 2 autovalori}, it is enough to estimate the measure of the set $\widetilde R_{kj}(\ell, \alpha, \beta)$ for any $k =1, \ldots, n_\alpha$, $j = 1, \ldots, n_\beta$.
Since, by Lemma \ref{limitazioni indici risonanti}-$(i)$, $\ell \neq 0$, we can write 
$$
\omega = \frac{\ell}{|\ell|} s + v\,, \quad \text{with} \quad v \cdot \ell = 0\,.
$$
and we define 
\begin{equation}\label{phi (s) importante}
\phi(s ) := |\ell| s + \lambda_k^{(\alpha)}(s) - \lambda_j^{(\beta)}(s)\,, 
\end{equation}
$$
\lambda_k^{(\alpha)}(s) := \lambda^{(\alpha)}_k\Big( \frac{\ell}{|\ell|} s + v \Big)\,, \quad \forall \alpha \in \sigma_0(\sqrt{- \Delta})\,, \quad \forall k = 1, \ldots, n_\alpha
$$
and according to \eqref{espansione asintotica autovalori}, \eqref{stima mathtt r alpha k}
\begin{equation}\label{genesis 0}
\lambda_k^{(\alpha)}(s) =  m \, \alpha + \mathtt r_k^{(\alpha)}(s)\,,\qquad |\mathtt r_k^{(\alpha)}|^{\Lipg} \lesssim_q  \e\,\alpha^{- 1}\,.
\end{equation}
Using that $|\cdot |^{\rm lip} \leq \gamma^{- 1} |\cdot |^\Lipg$, recalling that $m$ does not depend on $\omega$ (see Section \ref{subsection rimparametrizzazione tempo}), one gets 
\begin{align}
|\phi(s_1) - \phi(s_2)| & \geq \Big( |\ell| - (|\mathtt r_j^{(\alpha)}|^{\rm lip} + |\mathtt r_k^{(\beta)}|^{\rm lip}) \Big)|s_1 - s_2|   \geq \Big( |\ell| - \gamma^{- 1}(|\mathtt r_j^{(\alpha)}|^{\Lipg} + |\mathtt r_k^{(\beta)}|^{\Lipg})  \Big)|s_1 - s_2|   \nonumber\\
& \stackrel{ \eqref{genesis 0}}{\geq} \Big( |\ell| -  C(q)\e \gamma^{- 1} \Big) |s_1 - s_2|  {\geq} \frac{|\ell|}{2} |s_1 - s_2|\,
\end{align}
for $\e \gamma^{- 1}$ small enough.
The above estimate implies that  
$$
\Big| \Big\{ s : \frac{\ell}{|\ell|}s + v \in \widetilde R_{kj}(\ell, \alpha, \beta) \Big\} \Big| \lesssim \frac{\gamma }{\alpha^{\mathtt d} \beta^{\mathtt d}\langle \ell \rangle^{\tau + 1}}
$$
and by Fubini Theorem we get 
$
|\widetilde R_{kj}(\ell, \alpha, \beta)| \lesssim \frac{ \gamma }{\alpha^{\mathtt d} \beta^{\mathtt d} \langle \ell \rangle^{\tau + 1}}\,.
$
Finally recalling \eqref{dimensione spazi bf E alpha} and \eqref{decomposizione risonante 2 autovalori}, we get the claimed estimate for the measure of $\widetilde R (\ell, \alpha, \beta)$ and the proof is concluded. 
 \end{proof}
 
 \medskip
 
 \noindent
 {\sc Proof of Theorem \ref{teorema finale stima in misura} concluded.}
 By \eqref{Omega - Omega infinito finale}, by applying Lemmata \ref{limitazioni indici risonanti}, \ref{stima in misura insiemi risonanti} and recalling the definitions of the constants $\tau$ and $\mathtt d$ made in \eqref{scelta tau mathtt d}, one gets the estimate 
 \begin{align}
| DC(\gamma, \tau) \setminus \Omega_\infty^{2 \gamma}| 
& \lesssim \sum_{\begin{subarray}{c}
\ell \in \Z^\nu\,,\, j, j' \in \Z^d
\end{subarray}}  \frac{ \gamma }{\langle j \rangle^{\mathtt d + 1 - d} \langle j' \rangle^{\mathtt d + 1 - d}\langle \ell \rangle^{\tau + 1}}  + \sum_{\begin{subarray}{c}
\ell \in \Z^\nu\,,\, j, j' \in \Z^d \\
|j|, |j'| \lesssim \langle \ell \rangle
\end{subarray}}  \frac{\gamma}{ \langle \ell \rangle^{\tau + 1 - 2 d} } \lesssim \gamma\,. \label{stima DC gamma tau meno Omega infinito}
 \end{align}
  Hence, the Theorem \ref{teorema finale stima in misura} follows by \eqref{Omega - Omega infinito preliminare}, \eqref{misura diofantei calssica}, \eqref{stima DC gamma tau meno Omega infinito}. 

\section{ Proof of Theorem \ref{theorem linear stability} and Corollary \ref{crescita norme di sobolev}.}\label{conclusione proof}
In this section we prove Theorem \ref{theorem linear stability} and Corollary \ref{crescita norme di sobolev}. We define 
\begin{equation}
{\cal W}_1(\vphi) := {\cal S}(\vphi) \circ {\cal C}\,, \qquad {\cal W}_2(\vphi) := {\cal T}(\vphi) \circ \Phi_\infty(\vphi)\,, \quad \vphi \in \T^\nu
\end{equation}
where the maps ${\cal S}$, ${\cal C}$, ${\cal T}$ are defined in \eqref{cal S1}, \eqref{matrice coordinate complesse}, \eqref{definizione cal TM} and the map $\Phi_\infty$ is given in Corollary \ref{lem:convPhi}. We define the constants
$$
\overline q = \overline q(\nu, d) := 2(\frak s_0 + \overline \mu + \mathtt b + 2 s_0 + 2 )
$$
and for any $q > \overline q$, we define 
$$
\mathfrak S_q = \mathfrak S(q, \nu, d)  := S_q -2 - 2 s_0 = [q/2] - \overline \mu - \mathtt b - 2 s_0 - 2
$$
where we recall the definitions \eqref{scelta di M nuove norme}, \eqref{definizione alpha}, \eqref{scelta tau mathtt d}. By Lemmata \ref{Lemma dopo simmetrizzazione}, \ref{ultimo lemma pre riducibilita}, \ref{stima Hs x fourier multiplier} and Corollary \ref{lem:convPhi} one gets that for $\e \gamma^{- 1} \leq \delta_q$ (for some constant $\delta_q$ small enough depending on $q, \nu, d$), for any $\vphi \in \T^\nu$, for any $\omega \in \Omega_\infty^{2 \gamma}$ the maps ${\cal W}_i(\vphi) = {\cal W}_i(\vphi; \omega)$, $i = 1, 2$ are bounded and invertible with 
$$
{\cal W}_1(\vphi) : {\bf H}^{s }_0(\T^d) \to H^{s + \frac12}_0(\T^d, \R) \times H^{s - \frac12}_0(\T^d, \R)\,, \quad {\cal W}_1(\vphi)^{- 1} :  H^{s + \frac12}_0(\T^d, \R) \times H^{s - \frac12}_0(\T^d, \R) \to {\bf H}^s_0(\T^d)\,,
$$
for any $1/2 \leq s \leq \mathfrak S_q$ and 
$$
{\cal W}_2(\vphi)^{\pm 1} : {\bf H}^s_0(\T^d) \to {\bf H}^s_0(\T^d)\,, \qquad \forall 0 \leq s \leq \mathfrak S_q\,. 
$$
Let $1/2 \leq s \leq \mathfrak S_q$ and $(v^{(0)}, \psi^{(0)}) \in H^{s + \frac12}(\T^d, \R) \times H^{s - \frac12}(\T^d, \R)$. For any $\omega \in \Omega_\infty^{2 \gamma}$, defining ${\cal W}_\infty(\vphi) := {\cal W}_1(\vphi) \circ {\cal A} \circ{\cal W}_2(\vphi)$, by the change of variable 
\begin{equation}\label{equazione pernacchio}
(v (t, \cdot), \psi(t, \cdot)) = {\cal W}_\infty(\omega t) [{\bf u}(t, \cdot)]\,, \quad {\bf u} = (u, \overline u)
\end{equation}
(recall that ${\cal A}$ is the reparametrization of time defined in \eqref{reparametrizzazione tempo dinamica}), the Cauchy problem 
\begin{equation}\label{equazione compatta nella prova finale}
\begin{cases}
(\partial_t v, \partial_t \psi) = {\cal L}(\omega t)[(u, \psi)]\,.  \\
(v(0, \cdot), \psi(0, \cdot)) = (v^{(0)}, \psi^{(0)})\,. 
\end{cases}
\end{equation}
 is transformed into 
\begin{equation}\label{problema di cauchy h bar h}
\begin{cases}
\partial_t {\bf u} = {\cal D}_\infty {\bf u}  \\
{\bf u}(0, \cdot) = {\bf u}^{(0)}
\end{cases}\,, \quad {\bf u}^{(0)} = (u^{(0)}, \overline u^{(0)}) = {\cal W}_2(0)^{- 1} \circ {\cal W}_1(0)^{- 1}[(v^{(0)}, \psi^{(0)})]
\end{equation}
where the operator ${\cal D}_\infty = \begin{pmatrix}
- \ii {\cal D}_\infty^{(1)} & 0 \\
0 & \ii \overline{\cal D}_\infty^{(1)}
\end{pmatrix}$ is defined in \eqref{cal D infinito}. Note that, since for any $\alpha \in \sigma_0(\sqrt{- \Delta})$, the block $[{\cal D}^{(1)}_\infty]_\alpha^\alpha$ is self-adjoint, one has that the operator ${\cal D}_\infty^{(1)}$ is self-adjoint, i.e. 
\begin{equation}\label{autoaggiuntezza D indinito (1)}
{\cal D}_\infty^{(1)} = \big( {\cal D}_\infty^{(1)} \big)^*\,.
\end{equation} 
Then, we consider the Cauchy problem 
\begin{equation}\label{problema di cauchy h}
\begin{cases}
\partial_t u = - \ii {\cal D}_\infty^{(1)} u  \\
{u}(0, \cdot) = {u}^{(0)}\,.
\end{cases}
\end{equation}
We prove that 
\begin{equation}\label{stima problema di Cauchy equazione ridotta}
\| u(t, \cdot)\|_{H^s_x} = \| u^{(0)} \|_{H^s_x}\,, \qquad \forall t \in \R\,. 
\end{equation}
Since ${\cal D}_\infty^{(1)}$ is a block-diagonal operator, one can easily verify that the commutator $[|D|^s, {\cal D}_\infty^{(1)}] = 0$ and therefore
\begin{align}
\partial_t \| h(t, \cdot) \|_{H^s_x}^2 &  = - \Big( \ii ({\cal D}_\infty^{(1)} - \big({\cal D}_\infty^{(1)} \big)^*) |D|^s h, |D|^s h \Big)_{L^2_x} \stackrel{\eqref{autoaggiuntezza D indinito (1)}}{=} 0 \nonumber
 \end{align}
 which implies \eqref{stima problema di Cauchy equazione ridotta}. 
 
 \noindent
 Now, by \eqref{equazione pernacchio} one has that for any $1/2 \leq s \leq \mathfrak S_q$ 
 \begin{align}
 \| (u (t, \cdot), \psi(t, \cdot))  \|_{H^{s + \frac12}_x \times H^{s - \frac12}_x} & {\lesssim_q} \| {\cal A} \circ {\cal W}_2(\omega t)[{\bf u}(t, \cdot)] \|_{{\bf H}^s_x} \stackrel{\eqref{reparametrizzazione tempo dinamica}}{\lesssim_q} \|  {\cal W}_2(\omega \tau + \omega \alpha(\omega \tau))[{\bf u}(\tau + \alpha (\omega \tau), \cdot)] \|_{{\bf H}^s_x} \nonumber\\
 & {\lesssim_q} \| {\bf u}(\tau + \alpha (\omega \tau), \cdot) \|_{{\bf H}^s_x} \stackrel{\eqref{stima problema di Cauchy equazione ridotta}}{\lesssim_q} \| {\bf u}_0\|_{{\bf H}^s_x} \stackrel{\eqref{problema di cauchy h bar h}}{\lesssim_q} \| (v^{(0)}, \psi^{(0)})\|_{H^{s + \frac12}_x \times H^{s - \frac12}_x}\,. \nonumber
 \end{align}

 Set $\gamma = \e^a$, with $0 < a < 1$ and $\Omega_\e := \Omega_\infty^{2 \gamma}$. Then $ \e \gamma^{- 1} =  \e^{1 - a} $ and hence the smallness condition $\e \gamma^{- 1} \leq \delta_q$ is fullfilled  by taking $\e$ small enough. Furthermore, by Theorem \ref{teorema finale stima in misura}, since $\gamma = \e^a$, we get that \eqref{stima misura main theorem} holds and therefore Theorem \ref{theorem linear stability} and Corollary \ref{crescita norme di sobolev} have been proved. 

\section{Appendix}
We prove some elementary properties of the set $\sigma_0(\sqrt{- \Delta})$ defined in \eqref{definizione cal N}.
\begin{lemma}\label{lemma spettro laplaciano}
$(i)$ Let $p > d$. Then $\sum_{\alpha \in \sigma_0(\sqrt{- \Delta})} \alpha^{- p} < +\infty$. If $p > d + \nu$, $\sum_{\begin{subarray}{c}
\ell \in \Z^\nu \\
\alpha \in  \sigma_0(\sqrt{- \Delta})
\end{subarray}} \langle \ell , \alpha \rangle^{- p} < + \infty$.

\noindent
$(ii)$ Let $\alpha, \beta \in \sigma_0(\sqrt{- \Delta})$ with $\alpha \neq \beta$. Then there exists a constant $C > 0$ such that $|\alpha - \beta| \geq C (\alpha^{- 1} + \beta^{- 1})$. 
\end{lemma} 
\begin{proof}
{\sc Proof of $(i)$.} By the definition \eqref{definizione cal N} one has that 
$$
\sum_{\alpha \in \sigma_0(\sqrt{- \Delta})} \alpha^{- p}  \leq \sum_{j \in \Z^d} \langle j \rangle^{- p}\,, \quad \sum_{\begin{subarray}{c}
\ell \in \Z^\nu \\
\alpha \in  \sigma_0(\sqrt{- \Delta})
\end{subarray}} \langle \ell , \alpha \rangle^{- p} \leq \sum_{\begin{subarray}{c}
\ell \in \Z^\nu \\
j \in \Z^d
\end{subarray}} \langle \ell , j \rangle^{- p}\,.
$$
the first series on the right hand side converges if $p > d$ and the second one for $p > \nu + d$.

\noindent
{\sc Proof of $(ii)$.} First, we note that if $x, y \in \N$, $x \neq y$ one has that 
 $$
 |\sqrt{x} - \sqrt{y}| \geq {\rm max}\Big\{ \frac{1}{\sqrt{x}}, \frac{1}{\sqrt{y}} \Big\} \geq C \Big( \frac{1}{\sqrt{x}} + \frac{1}{\sqrt{y}} \Big)\,,
 $$
 for some constant $C > 0$. Since by the definition of $\sigma_0(\sqrt{- \Delta})$, if $\alpha, \beta \in \sigma_0(\sqrt{- \Delta})$, $\alpha \neq \beta$, they are square roots of integer numbers, i.e. $\alpha^2, \beta^2 \in \N$, the claimed inequality follows.  
 \end{proof}
 
 \medskip
 
 \noindent
 Now we recall some well known facts concerning linear self-adjoint operators on finite dimensional Hilbert spaces. Let ${\cal H}$ a finite dimensional Hilbert space of dimension $n$ equipped by the inner product $\langle \cdot\,,\,\cdot \rangle_{\cal H}$. Let us denote by ${\cal B}({\cal H})$ the space of the linear operators from ${\cal H}$ onto itself, equipped by the operator norm $\| \cdot \|_{{\cal B}({\cal H})}$. For any self-adjoint operator $A : {\cal H} \to {\cal H}$, we order its eigenvalues as
\begin{equation}\label{spettro hilbert astratto}
{\rm spec}(A) := \big\{\lambda_1(A) \leq \lambda_2(A) \leq \ldots \leq \lambda_n(A)\big\}\,.
\end{equation}
We recall the well known property
\begin{equation}\label{proprieta elementare norma operatoriale autovalori}
\| A \|_{{\cal B}({\cal H})} = {\rm max}_{\lambda \in {\rm spec}(A)}|\lambda|\,.
\end{equation}
Moreover the following lemma holds
\begin{lemma}\label{risultato astratto operatori autoaggiunti}
 Let ${\cal H}$ be a Hilbert space of dimension $n$. Then the following holds:

\noindent
$(i)$ Let $A_1, A_2 : {\cal H} \to {\cal H}$ be self-adjoint operators. Then their eigenvalues, ranked as in \eqref{spettro hilbert astratto}, satisfy the Lipschitz property 
$$
|\lambda_k(A_1) - \lambda_k(A_2)| \leq \| A_1 - A_2 \|_{{\cal B}({\cal H})}\,, \qquad \forall k = 1, \ldots, n\,.
$$

\noindent
$(ii)$ Let $A = \eta {\rm Id}_{\cal H} + B$, where $\eta \in \R$, ${\rm Id}_{\cal H} : {\cal H} \to {\cal H}$ is the identity and $B :{\cal H} \to {\cal H}$ is selfadjoint. Then 
$$
\lambda_k(A) = \eta + \lambda_k(B) \,, \qquad \forall k = 1, \ldots , n\,. 
$$ 

\noindent
$(iii)$ Let $A : {\cal H} \to {\cal H}$ be self-adjoint and assume that ${\rm spec}(A) \subset \R \setminus \{ 0 \}$. Then $A$ is invertible and its inverse satisfies
$$
\| A^{- 1}\|_{{\cal B}({\cal H})} = \dfrac{1}{\min_{k = 1, \ldots, n}|\lambda_k(A)|}\,.
$$
\end{lemma}

\smallskip

\begin{flushright}
Riccardo Montalto,  University of Z\"urich, Winterthurerstrasse 190,
CH-8057, Z\"urich, Switzerland. \\ \emph{E-mail: {\tt riccardo.montalto@math.uzh.ch}} 
\end{flushright}


\begin{thebibliography}{10}



\bibitem{arosio-spagnolo}  A. Arosio, S. Spagnolo, \emph{Global solutions of the Cauchy problem for a nonlinear
hyperbolic equation}. Nonlinear PDEÕs and their applications, Coll\'ege de France
Seminar, Vol. VI, 1Ð26, H. Brezis, J.L. Lions eds., Research Notes Math. 109,
Pitman, Boston, 1984. 

\bibitem{arosio-panizzi}  A. Arosio, S. Panizzi, \emph{On the well-posedness of the Kirchhoff string}. Trans. Amer.
Math. Soc. 348, no.1, 305-330, 1996. 

\bibitem{Baldi Kirchhoff} P. Baldi, \emph{Periodic solutions of forced Kirchhoff equations}. Ann. Scuola Norm. Sup. Pisa, Cl. Sci. (5), Vol. 8, 117-141, 2009.  


\bibitem{BBM-Airy}
P. Baldi, M. Berti, R. Montalto,
\emph{KAM  for quasi-linear and fully nonlinear forced perturbations of Airy equation}. 
Math. Annalen 359, 471-536, 2014. 
  
\bibitem{BBM-auto}
P. Baldi, M. Berti, R. Montalto,
\emph{KAM for autonomous quasi-linear perturbations of KdV}.
Ann. I. H. Poincar\'e (C) Anal. Non Lin\'eaire 33, 1589-1638, 2016.  

\bibitem{BBM-mKdV}
P. Baldi, M. Berti, R. Montalto,
\emph{KAM for autonomous quasi-linear perturbations of mKdV}.
Boll. Unione Mat. Ital,  9, 143-188, 2016.

\bibitem{KAM ww f depth} P. Baldi, M. Berti, E. Haus, R. Montalto, \emph{Time quasi-periodic gravity water waves in finite depth}. Preprint 2017. 

\bibitem{Bambusi1} D. Bambusi, \emph{Reducibility of 1-d Schr\"odinger equation with time quasiperiodic unbounded perturbations, I}. Trans. Amer. Math. Soc., doi:10.1090/tran/7135, 2017.

\bibitem{Bambusi2} D. Bambusi, \emph{Reducibility of 1-d Schr\"odinger equation with time quasiperiodic unbounded perturbations, II}. Commun. Math. Phys, doi:10.1007/s00220-016-2825-2, 2017.

\bibitem{Bambusi-Graffi} D. Bambusi, S. Graffi, \emph{Time quasi-periodic unbounded perturbations of Schr\"odinger operators and KAM methods}. Comm. Math. Phys. 219, 465-480, 2001. 


\bibitem{Bernstein} S.N. Bernstein, \emph{ Sur une classe d' \'equations fonctionelles aux d\'eriv\'ees partielles}.
Izv. Akad. Nauk SSSR Ser. Mat. 4, 17Ð26, 1940. 



\bibitem{BBiP1}
M. Berti, L. Biasco P., M. Procesi, {\it KAM theory for the Hamiltonian DNLW}.
Ann. Sci. \'Ec. Norm. Sup\'er. (4),  
Vol. 46, fascicule 2, 301-373, 2013. 

\bibitem{BBiP2}
M. Berti, L. Biasco, M. Procesi, {\it KAM theory for the reversible derivative wave equation}.
Arch. Rational Mech. Anal.,  212, 905-955, 2014. 

\bibitem{BB12} M. Berti, P. Bolle, {\it Sobolev quasi periodic solutions 
of multidimensional wave equations with a multiplicative potential}.  
Nonlinearity 25, 2579-2613, 2012.

\bibitem{BBo10} Berti M., Bolle P.,  {\it Quasi-periodic solutions 
with Sobolev regularity of NLS on $ \T^d $ with a multiplicative potential}.
Eur. Jour. Math. 15, 229-286, 2013.

\bibitem{BCP} Berti M., Corsi L., Procesi M., 
\emph{An abstract Nash-Moser theorem and quasi-periodic solutions for NLW and NLS 
 on compact Lie groups and homogeneous manifolds}. 
Comm. Math. Phys. 334, no.\,3, 1413-1454, 2015. 








%



\bibitem{BertiKappelerMontalto}M. Berti, T. Kappeler, R. Montalto, \emph{Large KAM tori for perturbations of the dNLS equation}. Preprint arXiv:1603.09252, 2016. 

\bibitem{BM16} 
M. Berti, R. Montalto, 
\newblock \emph{Quasi-periodic water waves}. J. Fixed Point Theory Appl., 
19, no. 1, 129-156, 2017. 

\bibitem{BertiMontalto} M. Berti, R. Montalto, \emph{Quasi-periodic standing wave solutions for gravity-capillary water waves}, to appear on Memoirs of the Amer. Math. Society. Preprint arXiv:1602.02411v1, 2016. 




\bibitem{Bo1} J. Bourgain, {\it  Construction of quasi-periodic solutions 
for Hamiltonian perturbations of linear equations and applications 
to nonlinear PDE}. Internat. Math. Res. Notices  no.\,11, 1994. 

\bibitem{B3}  J. Bourgain, {\it Quasi-periodic solutions of Hamiltonian
perturbations of $2D$ linear Schr\"odinger equations}.
Annals of Math. 148, 363-439,  1998. 

%

\bibitem{B5}  J. Bourgain,  {\it Green's function estimates for lattice Schr\"odinger 
operators and applications}. Annals of Mathematics Studies 158, 
Princeton University Press, Princeton, 2005.

\bibitem{ChierchiaYou} { L. Chierchia, J. You}, {\it KAM tori for 1D nonlinear wave equations with periodic
boundary conditions}. Comm. Math. Phys. 211, 497-525, 2000.




\bibitem{dancona-spagnolo} P. D'Ancona, S. Spagnolo, \emph{Global solvability for the degenerate Kirchhoff equation
with real analytic data}. Invent. Math. 108, 247-262, 1992. 


\bibitem{dickey} R.W. Dickey, \emph{Infinite systems of nonlinear oscillation equations related to the
string}. Proc. Amer. Math. Soc. 23, no.3, 459-468, 1969. 

\bibitem{EK1} L. H. Eliasson, S. Kuksin, {\it On reducibility of Schr\"odinger equations with
quasiperiodic in time potentials}. Comm. Math. Phys. 286, 125-135, 2009. 

\bibitem{EK} L. H. Eliasson, S. Kuksin,
{\it KAM for non-linear Schr\"odinger equation}.  Annals of Math. 172, 371-435, 2010. 

\bibitem{almostreducibility} L.H. Eliasson,  B. Grebert, S. Kuksin. {\it Almost reducibility for Klein gordon equations with
quasiperiodic in time potentials}, work in preparation. 

\bibitem{KAMbeam} L.H. Eliasson,  B. Grebert, S. Kuksin, {\it KAM for the nonlinear beam equation}. Geom. Funct. Anal. Vol. 26, 1588-1715, 2016.   

\bibitem{Feola}
R. Feola,
\emph{KAM for quasi-linear forced hamiltonian NLS}. Preprint arXiv:1602.01341, 2016. 

\bibitem{Feola-Procesi}
R. Feola, M. Procesi,
\emph{Quasi-periodic solutions for fully nonlinear forced reversible Schr{\"o}dinger equations}. J. Differential Equations 259, no.\,7, 3389-3447, 2015. 
%

\bibitem{Giuliani} F. Giuliani, {\it Quasi-periodic solutions for quasi-linear generalized KdV equations.} J. Differential Equations 262,  5052-5132, 2017.

\bibitem{GrebertPaturel} B. Grebert, E. Paturel, {\it On reducibility of quantum harmonic oscillator on $\R^d$ with quasiperiodic in time potential}. Preprint arXiv:1603.07455, 2016.

\bibitem{Grebert-Paturel-sfera} B. Grebert, E. Paturel, {\it KAM for the nonlinear Klein Gordon equation on $\mathbb S^d$}. Boll. Unione Mat. Ital. 9, 237-288, 2016.  

\bibitem{GT}  B. Grebert, L. Thomann, {\it KAM for the
quantum harmonic oscillator}. Comm. Math. Phys. 307, no.\,2, 383-427, 2011. 

%


\bibitem{Ioo-Plo-Tol} G. Iooss, P.I. Plotnikov, J.F.  Toland, 
{\it Standing waves on an infinitely deep perfect fluid under gravity}.
Arch. Ration. Mech. Anal. 177, no.\,3, 367-478, 2005. 


%

\bibitem{lions} J. L. Lions, \emph{On some questions in boundary value problems of mathematical
physics}. Contemporary developments in continuum mechanics and PDEs,
G.M. de la Penha, L.A. Medeiros eds., North-Holland, Amsterdam, 1978.

\bibitem{Liu-Yuan}
J. Liu , X. Yuan,
\emph{ A {KAM} theorem for {H}amiltonian partial differential equations with   unbounded perturbations}.
 Comm. Math. Phys., 307(3), 629--673, 2011.


\bibitem{KaP} T. Kappeler, J. P\"{o}schel, {\it KAM and KdV}. Springer, 2003.


\bibitem{Ku}  S. Kuksin,
{\it Hamiltonian perturbations of infinite-dimensional linear systems with imaginary spectrum}.
Funktsional. Anal. i Prilozhen. 21, no.\,3, 22--37, 95, 1987. 

\bibitem{K2} S. Kuksin, {\it A KAM theorem for equations of the Korteweg-de Vries type}.
Rev. Math. Math Phys. 10, no.\,3, 1-64, 1998. 


\bibitem{KP} S. Kuksin, J. P\"oschel, {\it
Invariant Cantor manifolds of quasi-periodic oscillations
for a nonlinear Schr\"{o}dinger equation}. Annals of Math. (2) 143, 149-179, 1996. 




\bibitem{LY}  J. Liu, X. Yuan,
{\it A KAM Theorem for Hamiltonian Partial Differential
Equations with Unbounded Perturbations}. Comm. Math. Phys. 307, 629-673, 2011. 

\bibitem{manfrin} R. Manfrin, \emph{On the global solvability of Kirchhoff equation for non-analytic initial
data}. J. Differential Equations 211, 38-60, 2005. 

\bibitem{Ruzansky} T. Matsuyama, M. Ruzhansky, \emph{Global Well-Posedness of the Kirchhoff
Equation and Kirchhoff Systems}. Analytic methods in interdisciplinary applications, Springer proceedings in mathematics and statistics, vol. 116, pp 81-96, 2014.

\bibitem{Montalto} R. Montalto, \emph{Quasi-periodic solutions of forced Kirchhoff equation}. Nonlinear Differ. Equ. Appl. NoDEA, 24:9, DOI:10.1007/s00030-017-0432-3, 2017. 





\bibitem{pozohaev} S.I. Pokhozhaev, \emph{On a class of quasilinear hyperbolic equations}. Mat. Sbornik 96, 152-166, 1975 (English transl.: Mat. USSR Sbornik 25, 145-158, 1975). 


\bibitem{PP} C. Procesi, M. Procesi, {\it A KAM algorithm for the resonant non-linear Schr\"odinger equation}. Advances in Mathematics 272, 399-470, 2015.





\bibitem{W1}  E. Wayne, {\it Periodic and quasi-periodic solutions
of nonlinear wave equations via KAM theory}. Comm. Math. Phys. 127, 479-528, 1990. 

\bibitem{ZGY} J. Zhang, M. Gao, X. Yuan,
{\it KAM tori for reversible partial differential equations}.
Nonlinearity 24, 1189-1228,  2011. 


\end{thebibliography}
\end{document}